\documentclass[a4paper,11pt]{amsart}
\usepackage[utf8]{inputenc}
\usepackage{mathtools}
\usepackage{amsmath, amsthm, amsfonts,amssymb}
\usepackage{enumitem}
\usepackage{graphicx}
\usepackage{colortbl}
\usepackage{tikz}
\usepackage[utf8]{inputenc}
\usepackage{esint}
\usepackage{mathrsfs}
\usepackage{subfig,float}
\usepackage[T1]{fontenc}
\usepackage{mathrsfs}  
\usepackage{bbm} 
\usepackage{enumitem}
\usepackage{mathtools}
\usepackage{dsfont}
\usepackage{xcolor}
\usepackage{stmaryrd}
\newcommand{\set}[1]{\left\{#1\right\}}
\newcommand{\M}{\mathcal{M}}
\newcommand{\cL}{\mathcal{L}}
\newcommand{\cA}{\mathcal{A}}
\newcommand{\cB}{\mathcal{B}}
\newcommand{\sfL}{\mathsf{L}}

\newcommand{\sfH}{\mathsf{H}}
\newcommand{\sfA}{\mathsf{A}}
\newcommand{\sfD}{\mathsf{D}}
\newcommand{\sfR}{\mathsf{R}}
\newcommand{\sfe}{\mathsf{e}}

\newcommand{\dt}{\partial_t}
\newcommand{\Rd}{{\R^d}}
\newcommand{\RRd}{{\R^d\times\R^d}}

\def\gradx{{\nabla_x}}

\usepackage{csquotes}

\numberwithin{equation}{section}
\numberwithin{figure}{section}

\DeclareMathOperator{\dive}{div}

\DeclareMathOperator{\sgn}{sgn}

\def\a{\frac{\gamma}{\gamma+1}}

\def\d{\,\mathrm{d}}
\def\\dv {\d v}
\def \ddt{\frac{\mathrm{d}}{\mathrm{d}t}}
\def \ddt{\frac{\mathrm{d}}{\mathrm{d}t}}

\newcommand{\sign}{\text{sgn}}
\newcommand{\bangle}[1]{\langle #1\rangle}
\newcommand{\wangle}[1]{\lfloor #1 \rceil}

\usepackage[colorlinks=true,linkcolor=blue,citecolor=blue,urlcolor=blue,breaklinks]{hyperref}

\usepackage{hyperref,cleveref}
\usepackage[square,sort,comma,numbers]{natbib}
\usepackage{url}
\usepackage[square,sort,comma,numbers]{natbib}
\DeclarePairedDelimiter\abs{\lvert}{\rvert}%
\DeclarePairedDelimiter\norm{\lVert}{\rVert}%
\DeclarePairedDelimiter\Drange{\llbracket}{\rrbracket}
\newcommand{\inorm}[1]{{\left\vert\kern-0.1ex\left\vert\kern-0.1ex\left\vert #1 \right\vert\kern-0.1ex \right\vert\kern-0.1ex\right\vert}}

\usepackage[left=2cm,right=2cm,top=2cm,bottom=2cm]{geometry}

\definecolor{lpink}{rgb}{0.96, 0.76, 0.76}
\definecolor{dpink}{rgb}{0.97, 0.51, 0.47}
\definecolor{sky}{rgb}{0.53, 0.81, 0.92}
\definecolor{salmon}{rgb}{1.0, 0.55, 0.41}
\definecolor{orman}{rgb}{0.24, 0.7, 0.44}
\definecolor{aciksari}{rgb}{0.91, 0.84, 0.42}
\definecolor{dgrey}{rgb}{0.52, 0.52, 0.51}

\def\R{\mathbb{R}}

\def\S{\mathbb{S}}

\def\NN{\mathbb{N}}

\def\1{\mathds{1}}

\def\d{\,\mathrm{d}}
\def\dv{\,\mathrm{d}v}
\def\dx{\,\mathrm{d}x}

\let\eps\varepsilon

\renewcommand{\(}{\left(}
\renewcommand{\)}{\right)}
\newcommand{\wx}{\lfloor x \rceil}
 
\newtheorem{thm}{Theorem}[section]
\newtheorem{cor}[thm]{Corollary}
\newtheorem{lem}[thm]{Lemma}
\newtheorem{prop}[thm]{Proposition}

\newtheorem{hypothesis}[thm]{Hypothesis}

\theoremstyle{definition}

\theoremstyle{remark}
\newtheorem{remark}[thm]{Remark}

\title{Sub-exponential tails in biased run and tumble equations with unbounded velocities}

\author{\'Emeric Bouin}
\address[E. Bouin]{CEREMADE - Université Paris-Dauphine, PSL Research University, UMR CNRS 7534, Place du Mar\'echal de Lattre de Tassigny, 75775 Paris Cedex 16, France.}
\email{bouin@ceremade.dauphine.fr}

\author{Josephine Evans}
\address[J. Evans]{Warwick Mathematics Institute, Zeeman Building, University of Warwick, CV4 7AL.}
\email{Josephine.Evans@warwick.ac.uk}

\author{Luca Ziviani}
\address[L. Ziviani]{CEREMADE - Université Paris-Dauphine, PSL Research University, UMR CNRS 7534, Place du Mar\'echal de Lattre de Tassigny, 75775 Paris Cedex 16, France.}
\email{ziviani@ceremade.dauphine.fr}

\date\today

\begin{document}

\maketitle
\begin{abstract}
%Run and tumble equations are kinetic models widely used in mathematical biology to describe the motion of bacteria. In this paper, we focus on the mathematical analysis of models of colonies of bacteria that interact with their environment through chemotaxis that has tendency to redirect them towards the center of the space. We show existence, uniqueness and quantitative convergence to an equilibrium state. We work with velocities which are unbounded, so the steady state does not exhibit exponential tails but fatter ones. This is a crucial difficulty in our analysis, added to the fact that the equilibrium is not known explicitly. 
Run and tumble equations are widely used models for bacterial chemotaxis. In this paper, we are interested in the long time behaviour of run and tumble equations with unbounded velocities. We show existence, uniqueness and quantitative convergence towards a steady state. In contrast to the bounded velocity case, the equilibrium has sub-exponential tails and we have sub-exponential rate of convergence to equilibrium. This produces additional technical challenges. We are able to successfully adapt both Harris' type and $\sfL^2-$ hypocoercivity \textit{a la} Dolbeault-Mouhot-Schmeiser techniques.
\end{abstract}
\setcounter{tocdepth}{1}
\tableofcontents

\section{Introduction}

In this paper, we are interested in the behaviour of solutions to the \textit{run and tumble} model, which was first introduced by Stroock \cite{S74}, Alt \cite{A80}, and further studied by Othmer, Dunbar and Alt \cite{ODA88}. This model aims to investigate, at the mesoscopic level, the movement of bacteria in the presence of a chemotactic chemical substance and is given by
\begin{equation*}
    \partial_t f(t,x,v) +v \cdot \nabla_x f(t,x,v) =\int_\mathcal{V}(K(x,v,v')f(t,x,v')-K(x,v',v)f(t,x,v))\dv '.
\end{equation*}
The function $f :=f(t,x,v)$ represents for positive time $t \in \R^+$ the density of bacteria in the phase space $(x,v) \in \R^d \times \mathcal{V}$, where $\mathcal{V}\subseteq\R^d$ is the set of velocities. 

The movement of some microrganisms like \textit{E. Coli} can be seen as an alternation between two main movements: 
the \textit{run} phase, which, consists in a movement in a straight line at a constant speed, and the \textit{tumble}, that is the reorientation from the previous velocity $v$ to another velocity $v'$. The probability of choosing a velocity $v'$ after $v$ is described by the \textit{tumbling frequency} $K(x,v,v')$, that represents the distribution of post-tumbling velocities when the bacterium is located in $x$. This $x$-dependence of the collision operator can also produce a spatial confinement effect, despite the lack of a force field in the transport operator. In fact, depending on the position, a bacteria can be led to change its speed more frequently depending on whether it is heading towards an area with a lower chemo-attractant density. The kernel $K$ is then chosen in such a way that microorganisms move towards the regions with higher chemo-attractant density, causing the formation of clusters.

The novelty and challenge of this paper lies in the fact that we work with unbounded velocities. These were explored in \cite{BCN15} and following this subexponential rates of convergence to equilibrium were expected. Large velocities are a barrier to convergence. This is because even if a bacteria selects a velocity which points towards the peak of the chemoattractant density, if the velocity is sufficiently large (depending on the distance to the peak), the bacteria is likely to travel well past the peak and may be further from the peak when it next takes a tumble. It is a technical challenge in our paper to show that despite this effect we still have convergence to equilibrium and to adapt the technical tools to work with the stretch exponentials appearing in both the weighted functional spaces and rates of convergence. Briefly, our mathematical innovations are as follows.

\textbf{Finding appropriate weights in the space variable}. Run and Tumble equations have an unusual and mathematically challenging mechanism for producing spatial confinement of their solutions. Bacteria tumble more frequently when they are going down the gradient of the chemoattractant than up it. Over time we expect this to produce a skew in their velocity distribution towards velocities going towards peaks of chemoattractant. This skew in the velocity distribution then causes bacteria to aggregate around peaks of the chemoattractant. As described above this process is complicated in the presence of potentially very large velocities. Following the principles of Harris's theorem we weight our spaces with a \emph{Foster-Lyapunov function} which is a moment which remains bounded along the flow. Finding such a weight is challenging and we need to capture the third order confinement effects and deal with the challenges of high velocities. We adapt weights from \cite{MW17, EY24} which work in the case of bounded velocities. The immediate problem is that these weight functions are not positive in the case of unbounded velocities. In order to adapt our weight functions we need to use different reasoning when working in parts of the state space where $|v|$ is much larger than $|x|$ so our arguments are based around a splitting of the state space. Furthermore in order to make use of these alternative confinement mechanisms for large velocities we need to switch exponential terms to stretch exponential terms.

\textbf{Adapting $\cA+\cB$ splitting arguments in order to find $\sfL^\infty$ controls on the semi-group.}  The arguments based on Harris's theorem gives us $\sfL^1$-type information about the semi-group. We can use operator splitting techniques originally due to \cite{GMM17} to bootstrap this to $\sfL^\infty$. Here there are many technical challenges produces by the presence of many stretch exponential. We also needed to repeat similar arguments based around splitting the space into parts where $|v|$ is large compared to $|x|$ and the complement.

\textbf{Adapting $\sfL^2$ hypocoercivity methods to our context.} Hypocoercivity is a name given to a collection of tools (usually entropy based) for showing convergence to equilibrium for \emph{degenerately coercive equations}. Our final results involve adapting $\sfL^2$ hypocoercivity results in the style of Dolbeault-Mouhot-Schmeiser to the equation studied here. Since we already have $\sfL^1$ and $\sfL^\infty$ information it is likely we could get similar outcomes with interpolation. However our main motivation for using $\sfL^2$ hypocoercivity is to develop tools to extend the range of applicability of these techniques to equations with non-explicit steady states and sub-geometric rates of convergence induced by the spatial confinement. This program is particularly important if we wish, in the future, to look at related equations which are non-linear or linearised (so not preserving positivity) as these are outside the scope of the previous techniques. Here there are two main challenges. Firstly the steady state is non-explicit and we have limited information about its behaviour. It was shown in \cite{CRS15, EM24} that $\sfL^2$ style methods work well in this situation provided we can show the steady state still satisfies the appropriate functional inequalities. Secondly, due to the sub-exponential tails of our steady state the perturbation term used most $\sfL^2$ hypocoercivity works will not be bounded by the norm. We need to adapt our perturbation term, here we compare both the challenge and solution to those encountered in \cite{BDLS20, BDL22}.

In this paper, we assume that the set of velocities is $\mathcal{V}=\R^d$ and that the tumbling frequency has the expression
\[
K(x,v,v')=\Lambda\left(-v'\cdot \gradx M\right)\M(v) \qquad x,v,v'\in\R^d
\]
where we use the notation
\[
\wx=\sqrt{1+\abs{x}^2},\qquad \forall x\in\R^d.
\]
The function $\M$ is a fixed probability density of the form
\begin{equation}\label{eq:M}
    \M(v) = c_{0,\gamma}^{-1} \exp\left( - \frac{\abs{v}^\gamma}{\gamma} \right),\qquad \forall v\in\R^d
\end{equation}
for some $\gamma >0$, where 
\begin{equation}\label{eq:c_k}
    c_{k,\gamma} = \int_{v\in\Rd}\abs{v}^k \exp\left( - \frac{\abs{v}^\gamma}{\gamma} \right) \, \dv  = 2\pi^{\tfrac{d}{2}} \gamma^{\frac{d+k}{\gamma}-1} \frac{\Gamma\(\frac{d+k}{\gamma}\)}{\Gamma\(\frac{d}{2}\)}, \qquad \forall k\in\NN.
\end{equation}
The function $\Lambda$ is the \textit{tumbling rate} and it depends on the gradient of the external signal $M= \ln(S)$ along the direction of the velocity, where $S$ is a given
function for the density of the chemoattractant. This expression for $K$ means that, at every tumble, the new velocity is chosen randomly according to the density $\M$, but the frequency of such tumbles also depends on the mutual orientation between the initial velocity and the density of the chemoattractant. In this paper we suppose that the density of the chemoattractant is constant in time $\wangle{\cdot} = -\ln S$, so that our main equation is in fact
\begin{equation}\label{eq:main}
\partial_tf (t,x,v)+ v\cdot\nabla_x f (t,x,v)= \M(v) \int_{\R^d}\Lambda\left( \frac{x\cdot v'}{\wx}\right) f(t,x,v')\dv '-\Lambda\left(  \frac{x\cdot v}{\wx}\right)f(t,x,v),
\end{equation}
for $(t,x,v)\in \R^+\times\RRd$. For further use, we shall denote 
\begin{equation}\label{eq:cL}
    \cL(f)(x,v) :=-v\cdot\nabla_x f (x,v)+ \M(v) \int_{\R^d}\Lambda\left( \frac{x \cdot v'}{\wx}\right) f(x,v')\dv '-\Lambda\left(\frac{x\cdot v}{\wx}\right)f(x,v)
\end{equation}
and we shall call $S_\cL$ the associated semigroup. Moreover, we define the dual operator $\cL^*$ by
\[
\int_\RRd (\cL^*\phi)f\,\dx\dv  = \int_\RRd (\cL f)\phi\,\dx\dv ,
\]
for any $\phi\in W^{1,\infty}(\RRd)$ and $f\in \mathcal{C}_c(\RRd)$, that is
\begin{equation}\label{eq:cL*}
    \mathcal{L}^* \phi = v\cdot \nabla_x \phi + \Lambda\left(\frac{x\cdot v}{\wx}\right)   \left( \int \phi(x,v') \M(v') \mathrm{d}v' - \phi(x,v) \right). 
\end{equation}

\subsection{Summary of previous contributions}

The run and tumble model was first introduced by D. W. Stroock \cite{S74}. In that article, the author provides a stochastic modelling of the motion of E. Coli based on the experimental results of Berg and Brown \cite{BB72} (see also \cite{MK72}). The motion of a bacterium is broken into two modes, the run and the tumble, and can be described as a Markov process, whose backward equation is the run and tumble equation. Initially the set of velocities was considered to be the unit sphere. Some years later, the subject has been pushed further by W. Alt \cite{A80}, with an analysis on different models of biased random walks. In particular, it is shown that solutions to the underlying integro-differential equation allows to derive the well-known Patlak-Keller-Segel diffusion equation through the hydrodynamic limit. A development of this study can be also found in \cite{ODA88}. The non-linear version of the run and tumble has also been studied and it was first proposed by \cite{CMPS04}. 

The majority of the articles deal with the run and tumble equations with a bounded symmetric set of velocities $\mathcal{V}$, equipped with a certain probability measure. This means that bacteria cannot run at a velocity beyond a certain fixed threshold, furthermore they can choose post-reorientation velocities weighted with different probabilities. Among the most important works, we quote Calvez, Raoul and Schmeiser \cite{CRS15}, where they studied a one dimensional run and tumble model where the set of velocities $\mathcal{V}$ is an interval. They managed to prove existence and uniqueness of a stationary solution, showing in addition the exponential decay of its tails. Moreover, they applied the abstract hypocoercivity method developed in \cite{DMS15} in order to show an exponential convergence in time towards the stationary solution in a weighted $\sfL^2$ norm.

The cases of higher dimensions was later achieved by Mischler and Weng \cite{MW17}, thanks to a meticulous study of the semigroup associated to the equation. In that work, the set of velocities taken in consideration was the unit ball and the tumbling kernel is of the same type as in \cite{CRS15}. The authors exploited semigroup theory and the Krein-Rutman theorem to show an exponential convergence towards the stationary solution, effectively extending the same result obtained in dimension one to all dimensions. Many tools and ideas of the present article are inspired by this paper.

Similar results were also obtained by Evans and Yoldas \cite{EY23}, who discussed a linear run and tumble model with more general tumbling kernels which, for example, allow to model a bacterial sensitivity that varies more regularly. They made use of the Harris Theorem, which allows them to inherit the convergence for the non-linear equation too. Moreover, in \cite{EY24}, the same authors also considered non-uniform tumbling kernels, which are more physically relevant in terms of modelling the chemotactic bacterial motion, see \cite{BB72, M80}. In this work, a first result on the linear run and tumble model with unbounded velocities set is discussed, which has strongly inspired our work. The authors proved existence of a stationary solution and convergence estimates in total variation norm.

Another phenomenon that has been studied starting from the run and tumble model is the existence of travelling waves of bacterial concentration, see \cite{C19, BCN15} and the references therein. In addition to the reorientation mechanism due to the chemo-attractant, the presence of a nutrient \cite{C19} or a mechanism of reproduction and saturation of individuals \cite{BCN15} is introduced. It is proved that, always with bounded velocity sets, there exist travelling wave solutions. Concerning the unbounded velocity case, they also proved that there cannot exist travelling waves even in dimension 1, this is due to the fact that larger and larger velocities destroy the wave front.

The different behavior in unbounded velocities domains is the main motivation for this work. Our work can be seen as the continuation of \cite{MW17}, \cite{EY23} and \cite{EY24}. We explore the behaviour of solutions to the run and tumble equation when the set of velocity is the whole space $\mathcal{V}=\R^d$.

%In this paper we are not interested in existence and uniqueness of solution, the global existence of solutions for the \eqref{RT} models we study in this paper can be obtained by following the strategy e.g. \cite{CMPS04,MW17} and \cite{BCGP08,CP96,HKS05,BC09}.

\subsection{Assumptions and main results}

For a positive weight function $m\colon \RRd\to \R$ we denote $\sfL^p(m)$ the functional space defined by the norm
\[
\norm{f}_{\sfL^p(m)}:=\(\iint_\RRd \abs{f(x,v)}^pm(x,v)\dx\dv\)^{\frac{1}{p}}.
\]
In our first result we provide existence of a stationary solution $G$ and convergence rates towards it in $\sfL^1$ norm. The function $\M$ has been introduced in \eqref{eq:M} and we make the following assumption on the tumbling rate $\Lambda$.

\begin{hypothesis}\label{Hyp:psi}
The tumbling rate $\Lambda\colon \R\to(0,\infty)$ is a function of the form
\[
\Lambda=1+\chi\psi,
\]
where $\chi\in (0,1)$ is a fixed value and $\psi$ is a bounded, odd, increasing function and $m\mapsto m\psi(m)\in W^{1,\infty}$, $\norm{\psi}_{L^\infty}\leq 1$, $\lim_{m\to\pm \infty}\psi(m)=\pm 1$. \end{hypothesis}

The parameter $\chi\in(0,1)$ is the \textit{sensitivity} of the bacterial population to the chemo-attractant. We use the notation $\mathsf a\lesssim \mathsf b$ if there exist a constant $\mathsf c>0$ such that $\mathsf a\leq \mathsf{c}\,\mathsf{b}$ and we write $\mathsf a\asymp \mathsf b$
if both $\mathsf a\lesssim \mathsf b$ and $\mathsf b\lesssim \mathsf a$ hold. Our first main result is the following.

\begin{thm}\label{thm:main}
Let $\M$ be defined by \eqref{eq:M} for some $\gamma>0$ and assume that \Cref{Hyp:psi} holds. Then we have the following.
\begin{enumerate}[label=(\roman*)]
    \item There exists a unique, normalised, invariant by rotations, stationary state $G$ to  \eqref{eq:main} whose density $\rho_G$ satisfies the bounds
    \begin{equation}\label{eq:mainbounds}
    e^{-\underline{\nu}\wx^{\frac{\gamma}{1+\gamma}}}\lesssim \rho_G(x)\lesssim e^{-\overline{\nu}\wx^{\frac{\gamma}{1+\gamma}}},
    \end{equation}
    for some constants $\underline \nu> \overline \nu>0$.
    \item For any $m \colon (x,v)\mapsto e^{\nu\wx^a} + e^{b\abs{v}^\gamma}$, with $a\in\Big(0,\tfrac{\gamma}{1+\gamma}\Big]$, $b\in\(0,\tfrac{1}{\gamma}\)$ and $\nu\in(0,\overline{\nu})$, for any normalised $f_0\in\sfL^1(m)$ we have
    \begin{equation}
    \norm{ S_\cL(t) f_0 - G}_{\sfL^1} \lesssim   e^{-\lambda\, t^{a}}\norm{f_0-G}_{\sfL^1(m)}, \qquad \forall t\geq 0,    
    \end{equation}
    where $\lambda= a^{-a}>0$.
    \item For any $m \colon (x,v)\mapsto \wx^k+\wangle{v}^{2k}$, with $k>1$, for any normalised $f_0\in\sfL^1(m)$, we have
    \begin{equation}
        \norm{ S_\cL(t) f_0 - G}_{\sfL^1} \lesssim   \frac{1}{\wangle{t}^k}\norm{f_0-G}_{\sfL^1(m)}, \qquad \forall t\geq 0.
    \end{equation} 
\end{enumerate}
\end{thm}

The second main result of this paper is the adaptation of the Dolbeault-Mouhot-Schmeiser method \cite{DMS15} to the run and tumble equation \eqref{eq:main} to get convergence in weighted $\sfL^2$ norm. This method has been successfully applied to several kinetic equations \cite{DMS15,CRS15,BDLS20,BDL22,BDZ23} and is based on a micro-macro decomposition. Let us consider the Hilbert space $\sfL^2(G^{-1})$, where $G$ is the stationary solution of \eqref{eq:main}, whose existence is provided by \Cref{thm:main}. We denote by $\bangle{\cdot,\cdot}$  the natural scalar product and by $\norm{\cdot}_{\sfL^2(G^{-1})}$ the natural norm. We define the orthogonal projection $\Pi\colon \sfL^2(G^{-1})\to \sfL^2(G^{-1}) $ as 
\begin{equation}
    \Pi f=\frac{\rho_f}{\rho_G}G \qquad \forall f\in\sfL^2(G^{-1}).
\end{equation}
Thanks to the fact $\cL G=0$, we shall show the \textit{microscopic coercivity} property
\[
\bangle{\cL f,f}\lesssim- \norm{(1-\Pi)f}^2_{\sfL^2(G^{-1})},
\]
though it is not enough to deduce a convergence result. The main idea of the method is to define an entropy functional $\sfH\colon\sfL^2(G^{-1})\to \R$ that is equivalent to $\norm{\cdot}^2_{\sfL^2(G^{-1})}$ and for which we can quantify the dissipation
\[
-\ddt \sfH[f] = :\sfD[f]\qquad\forall f\in\sfL^2(G^{-1}), \; \iint_{\RRd}f\dx\dv =0.
\]
By estimating $\sfD[f]$ in terms of some weaker norm of $f$ it is possible to deduce a decay rate for $\sfH[f]$ for any $f$ with zero mass. Unlike many previous works, the fact that the steady state $G$ is not explicit represents an important difficulty for this method, as $G$ is the weight itself of the framework functional space. For these reasons, some features of the steady state must be known with sufficient precision, in particular we need assumptions on the density
\[
\rho_G(x)=\int_\Rd G(x,v)\dv,\qquad  x\in\Rd,
\]
and the matrix
\[
V_G(x)=\int_\Rd v\otimes v\, G(x,v)\dv, \qquad   x\in\Rd.
\]
Here is our precise assumption.

\begin{hypothesis}\label{Hyp:asympRho}
Assume that $\Lambda =1+\chi\,\sgn$ with $\chi\in(0,1)$ and that the density $\rho_G$ satisfies the bounds
\begin{equation}\label{eq:boundRHO}
    \rho_G(x)\asymp \wx^{\ell-\tfrac{\gamma}{1+\gamma}\(d-\tfrac{1}{2}\)} e^{-\nu \abs{x}^{\frac{\gamma}{1+\gamma}}}, \qquad \forall  x\in\R,
\end{equation}
where $\ell \in \(0, \frac{2}{1+\gamma}\)$ and $\nu= \frac{\gamma+1}{\gamma} (1+\chi)^{\frac{\gamma}{1+\gamma}}$. Moreover, the matrix  $V_G(x)$ satisfies
\begin{equation}\label{eq:AssV_G}
    \abs{V_G(x)^{-2}}\abs{\gradx V_G(x)}^2\wx^{\frac{2}{1+\gamma}}\lesssim 1\qquad\forall x\in\Rd.
\end{equation}
\end{hypothesis}
The choice $\Lambda =1 +\chi \,\sign$ is made in order to carry out explicit computations. To this assumption is supported by numerical results as shown in \Cref{fig:rhoG}.
\begin{figure}
    \centering
    \includegraphics[width=0.7\textwidth]{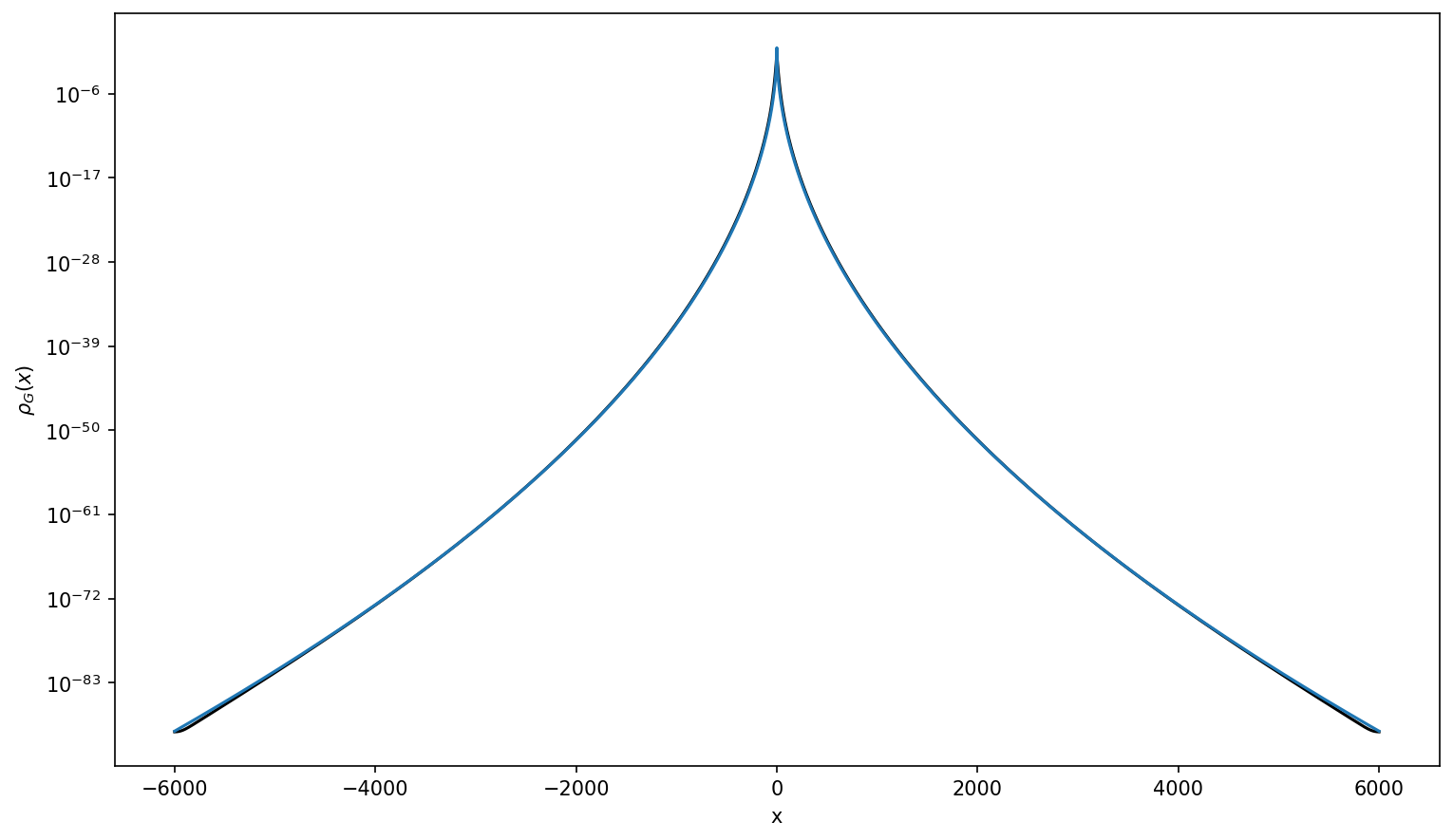}
    \caption{Black line: Plot of the density $\rho_G$ for $\gamma=1$ and $\chi=0.8$. Blue line: expected asymptotic behaviour $x\mapsto \abs{x}^{\frac{1}{4}}\exp(-2\sqrt{1.8}\,\abs{x}^{\frac{1}{2}})$.}
    \label{fig:rhoG}
\end{figure}
We considered a rectangular domain $(x,v)\in [-6000,\, 6000]\times[-100,100]$ and we numerically compute the stationary solution by integrating the \eqref{eq:main}. We can see that the density $\rho_G$ and the bound \eqref{eq:boundRHO}, with $\ell=\frac{1}{1+\gamma}$, almost overlap with each other. Moreover we also included an image of the full steady state $G$ in \Cref{fig:G}.
\begin{figure}
    \centering
    \includegraphics[width=0.7\textwidth]{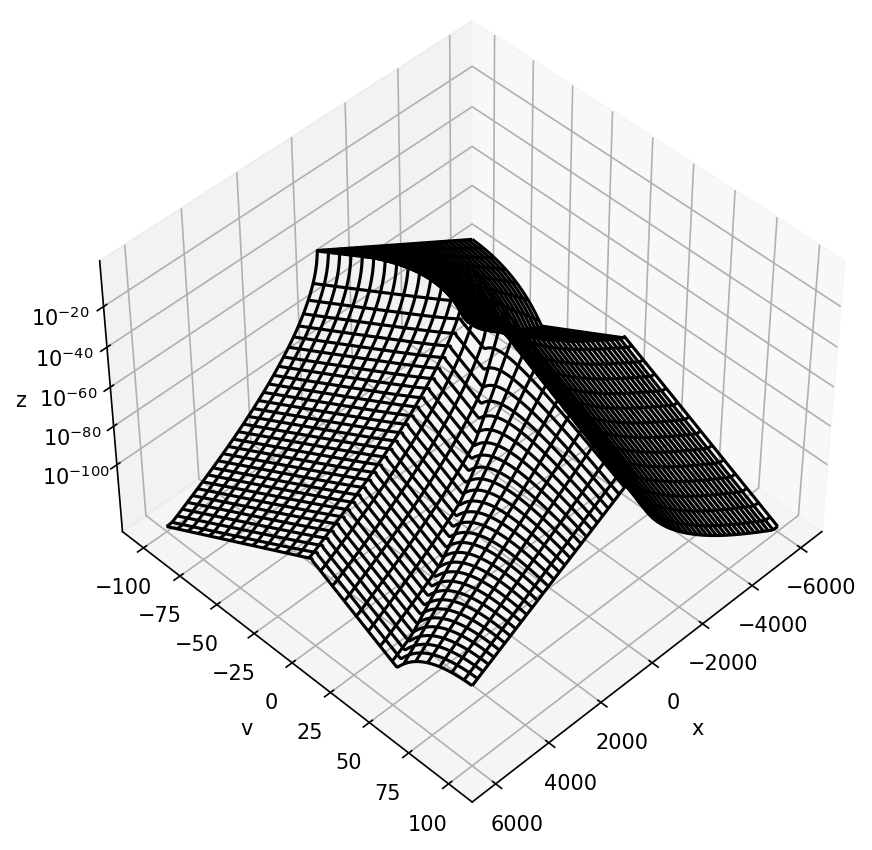}
    \caption{Plot of the stationary solution $G$ for $\gamma=1$ and $\chi=0.8$}
    \label{fig:G}
\end{figure}
It is surprising that, in region of the phase space where $x$ and $v$ have the same sign, the expression of $G$ seems to be piece-wise defined. The curve separating the two expressions has the form $v\approx x^{\frac{1}{1+\gamma}}$ for $x,v\geq 0$.

Our second main result is the following.
\begin{thm}\label{thm:DMS}
Let $\M$ be defined as in \eqref{eq:M} and assume \Cref{Hyp:asympRho}. Then we have the following estimates
\begin{enumerate}[label=(\roman*)]
    \item For any $m \colon (x,v)\mapsto e^{\nu\wx^a} + e^{b\abs{v}^\gamma}$ with $a\in\Big(0,\tfrac{\gamma}{1+\gamma}\Big]$, $b\in\(0,\tfrac{1}{\gamma}\)$ and $\nu\in(0,\overline{\nu})$, there exists a constant $\lambda>0$ such that for any initial datum $f_0\in\sfL^2(m\,G^{-1})$, we have
    \begin{equation}
        \norm{S_\cL(t)f_0-G}_{\sfL^2(G^{-1})}^2\lesssim e^{-\lambda\, t^{-\frac{a}{a+\ell}}}\norm{f_0-G}_{\sfL^2(m\,G^{-1})}^2\qquad\forall t\geq 0.
    \end{equation}
    \item For any $m \colon (x,v)\mapsto \wx^k+\wangle{v}^{2k}$, with $k>1$, there exists a constant $\lambda>0$ such that for any initial datum $f_0\in\sfL^2(m\,G^{-1})$ we have
    \begin{equation}
        \norm{S_\cL(t)f_0-G}_{\sfL^2(G^{-1})}^2\lesssim \frac{1}{(1+\lambda\, t)^{\frac{k}{\ell}}}\norm{f_0-G}_{\sfL^2(m\,G^{-1})}^2\qquad\forall t\geq 0.
    \end{equation}
\end{enumerate}
\end{thm}

Notice that in \Cref{thm:main} we were able to give some bounds on the tails of the density $\rho_G$, however we were not able to identify its precise asymptotic behavior for $\abs{x}\to \infty$. By using the refined bounds \eqref{eq:boundRHO} in place of \eqref{eq:mainbounds}, it is possible to recover some estimates on the moments
\begin{align*}
    P_G(x)&=\int_\Rd \abs{v}^2 G(x,v)\dv & P_G^{(4)}(x)&=\int_\Rd \abs{v}^4 G(x,v)\dv
\end{align*}
which are weights that will often come into play later. This information is later used to deduce that the entropy functional $\sfH$ satisfies 
\begin{equation}\label{eq:firstDISS}
    \ddt \sfH[f]\lesssim - \norm{f}^2_{\sfL^2(G^{-1}\wx^{-\ell})},
\end{equation}
where $\ell$ is defined in \Cref{Hyp:asympRho}. The loss of weight appearing in the norm is coming to the fact
\[
\frac{\rho_G P_G^{(4)}}{P_G^2}\asymp \wangle{\cdot}^{\ell},
\]
which turns out to be a crucial information for the method. In figure \Cref{fig:P4rhoP} we have plotted this ratio of weights and we see that it is actually unbounded and seems to suggest again that $\ell=\frac{1}{1+\gamma}$.
\begin{figure}
    \centering
    \includegraphics[width=0.7\textwidth]{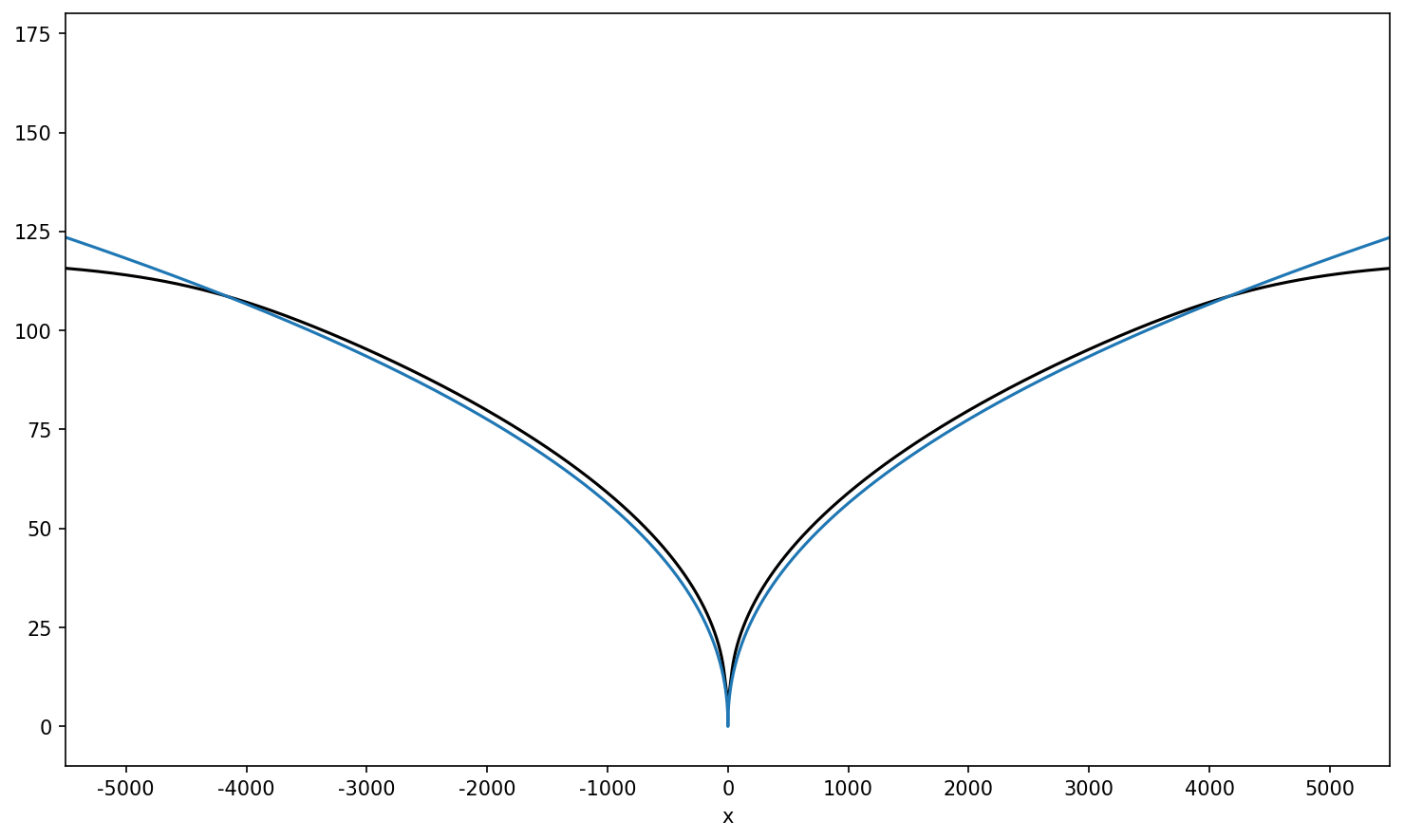}
    \caption{Black line: Plot of $\frac{\rho_G P_G^{(4)}}{P_G^2}$ for $\gamma=1$ and $\chi=0.8$. Blue line: extrapolation of the growth rate $ x\mapsto 2.35 \, \abs{x}^{0.46}$.}
    \label{fig:P4rhoP}
\end{figure}
Furthermore, we shall need the following weighted Poincaré inequality
\begin{equation*}
    \int_\R\abs{u-\Bar{u}}^2\wx^{-\frac{2}{1+\gamma}}P_G\dx\leq C_P \int_\Rd \abs{\gradx u}^2 P_G\dx.
\end{equation*}
where $\Bar{u}=\int_\Rd u\rho_G\dx$. This inequality is reminiscent of the weighted Poincaré inequality with the sub-exponential measure $e^{-\wx^\alpha}\dx$ with $\alpha\in(0,1)$, see \cite{BDLS20,Cao19,KMN21,BDZ23}. Notice that the weights involve the $P_G$ and the average involves $\rho_G$, that are not explicit. Hence we need a further argument to justify the change of measures and average.

Because of the loss of weight in \eqref{eq:firstDISS}, as a side effect we need to restore the norm $\norm{\cdot}^2_{\sfL^2(G^{-1})}$ in \eqref{eq:firstDISS} by some interpolation arguments with the bound $f_0\in\sfL^2(mG^{-1})$ the initial datum, with weight $m$ as in \Cref{thm:DMS}. This idea have been already used in kinetic theory in weak regimes \cite{BDLS20,Cao19,BDZ23}, and consists in proving that if $f_0\in\sfL^2(mG^{-1})$ for such $m$, then the whole solution $f=S_\cL f_0$ remains bounded in $\sfL^2(mG^{-1})$.

\subsection{Organisation of the paper}
In \Cref{Sec:Lyapunov}, we derive a family of Lyapunov functions for the run and tumble equation (in this article we mean Lyapunov functions in the probabilistic sense, also known as Foster-Lyapunov functions). Such functions are fundamental for the whole article, since they provide the correct weighted spaces to work in and determine the rate of convergence towards the equilibrium. In \Cref{Sec:Harris}, we use the sub-geometric Harris theorem to get the existence of a stationary measure and a convergence rate in total variation norm. A brief review of the Harris theorem is presented in \Cref{App:Harris}. In \Cref{sec:Duhamel}, we adapt the strategy of \cite{MW17} to study the semigroup $S_\cL$ in some suitable weighted spaces. The main purpose is to prove that $S_\cL$ is bounded in a weighted space $X=\sfL^1(m)\cap\sfL^\infty(m)$ and, in a second time, to prove existence of a stationary solution $G\in X$ through a fixed point argument. The main tool is the iterated Duhamel formula, which states that if the operator splits as $\cL=\cA+\cB$, then we can rewrite the semigroup $S_\cL$ in terms of $\cA$ and $S_\cB$ as
\[
S_\cL =S_{\cB}+ \sum_{j=1}^{d+1} S_{\cB}\star(\cA S_{\cB})^{\star j} + S_{\cB}\star (\cA S_{\cB})^{\star(d+1)} \star \cA S_\cL\,.
\]
The symbol $\star$ stands for the convolution operator on $\R_+$. In \Cref{Sec:ReprFormula}, we make use of a representation formula for the stationary solution $G$ to prove upper and lower pre-bound on the density $\rho_G$. Despite the fact these bounds are not optimal, we were able to prove them on the more general \Cref{Hyp:psi}. In \Cref{Sec:poincare}, we prepare the way to the proof of \Cref{thm:DMS} by studying the moments in velocity of the steady state $G$. For this we use tools of asymptotic Analysis of integrals of Laplace-type that are breafly recalled in \Cref{App:Laplace}. At the end of the section, we prove a weighted Poincaré inequality for the measure $P_G(x)\dx$. Finally, \Cref{sec:DMS} we give the proof of \Cref{thm:DMS} by adapting the Dolbeault-Mouhot-Schmeiser method. 

\section{Weak-Lyapunov functions in $\sfL^1$}\label{Sec:Lyapunov}

In this section, we find some Lyapunov functions $m$ for \eqref{eq:main}. Before that, we discuss a technical lemma about the function 
\[
\Psi(z) = z \psi(z), \qquad \forall z\in\Rd.
\]
This function plays a fundamental role for the proof of the Lyapunov condition, as has already been shown in \cite[Lemma 1.2]{EY23}.

\begin{lem}\label{lem:Psi}
Assume \Cref{Hyp:psi}. There exists a constant $\zeta>0$ such that 
\begin{equation}
    \int_\Rd\Psi\left(\frac{x\cdot v'}{\wx}\right) \M(v') \mathrm{d}v' \geq \zeta\, \1_{\abs{x}\geq 1 }.
\end{equation} 
\end{lem}
\begin{proof}[{\bf Proof of \Cref{lem:Psi}}]
Since $\lim_{z\to\pm \infty}\psi(z)=\pm 1$, there exists $R>0$ such that $\abs{\psi(z)}\geq \frac{1}{2}$ for all $\abs{z}\geq R$. As a consequence
\[
\Psi(z)= z\psi(z)\geq \frac{\abs{z}}{2},\qquad \text{for all } \abs{z}\geq R.
\]
Then, thanks to the positivity of $\Psi$ and the symmetry of $\M$, we have
\begin{align*}
    \int_\Rd  \Psi\left(\frac{x\cdot v'}{\wx}\right) \M(v') \mathrm{d}v' &\geq \int_{\abs*{\frac{x\cdot v'}{\wx}}\geq R}  \Psi\left(\frac{x\cdot v'}{\wx}\right) \M(v') \mathrm{d}v'\\
    &\geq\frac12 \int_{\frac{\abs{x\cdot v'}}{\wx}\geq R}  \abs*{\frac{x\cdot v'}{\wx}} \M(v') \mathrm{d}v'\\
    &\geq\frac12\frac{\abs{x}}{\wx} \int_{\frac{\abs{x}}{\wx}\abs{\text{e}\cdot v}\geq R}  \abs{\text{e}\cdot v} \M(v') \mathrm{d}v'
\end{align*}
Next, we point out that if $\abs{x}\geq 1$, then $\frac{\abs{x}}{\wx}\geq \frac{1}{\sqrt{2}}$. We conclude
that
\begin{align*}
    \int_\Rd  \Psi\left(\frac{x\cdot v'}{\wx}\right) \M(v') \mathrm{d}v' &\geq\frac12\frac{\abs{x}}{\wx}\1_{\abs{x}\geq 1} \int_{\abs{\text{e}\cdot v}\geq \sqrt{2} R}  \abs{\text{e}\cdot v} \M(v') \mathrm{d}v'\\
    &\geq \zeta \1_{\abs{x}\geq 1}
\end{align*}
where $\zeta =\frac{\sqrt{2}}{4} \int_{\abs{\text{e}\cdot v}\geq \sqrt{2} R}  \abs{\text{e}\cdot v} \M(v') \mathrm{d}v'$.
\end{proof}

\subsection{A sub-exponential weak Lyapunov function}

\begin{prop}\label{Prop:Lyapunov}
Take any $a\in(0,\tfrac{\gamma}{1+\gamma}]$, $b\in(0,\tfrac{1}{\gamma})$ and let \Cref{Hyp:psi} hold. There exist explicit positive constants $B,C, \nu, \epsilon$ such that $m$, defined for any $(x,v) \in \RRd$ by 
\begin{equation}\label{LyapunovFunctionExp}
    m(x,v) = \left(1+ \nu a x\cdot v \wx^{a-2} - \frac{\nu a\chi}{1+\chi} \Psi\left(\frac{x\cdot v}{\wx}\right) \wx^{a-1} + \nu B\abs{v}^2 \wx^{2a-2} \right)e^{\nu \wx^a} +\nu  e^{b\abs{v}^\gamma}
\end{equation}
is positive and is a weak Lyapunov function for $\cL$ in $\mathsf\sfL^1(\RRd)$, in the sense that
\begin{equation}\label{eq:expLyab}
    \cL^\ast m \leq C\nu - \epsilon\wangle{x}^{a-1} m 
\end{equation}
Moreover, there exist constants $\delta_1, \delta_2>0$ such that
\begin{equation}\label{bounds:m}
(1-\delta_1)e^{\nu\wx^a}+\nu e^{b\abs{v}^\gamma}\leq m(x,v)\leq \delta_2(e^{\nu\wx^a}+e^{b\abs{v}^\gamma}), \qquad \forall (x,v)\in\RRd .
\end{equation}
\end{prop}

Possible constants in the statement can be 
\[
B=1+\max\left\{\frac{\nu a^2(1+2\chi)^2}{4(1+\chi)^2}, \frac{a(3-a+\nu a)(2+3\chi)}{2(1-\chi)(1+\chi)}\right\}
\]
and 
\begin{align*}
    \delta_1&=\frac{\nu a^2(1+2\chi)^2}{4(1+\chi)^2B} & \delta_2&=2 +\nu(a^2\nu+B)\max\Big\{\(\tfrac{4}{b\gamma}\)^\frac{2}{\gamma}e^{\frac{4}{\gamma}},(\tfrac{2\nu}{b})^{\frac{2-2a}{a}}\Big\}+ \nu.
\end{align*}

\begin{proof}[{\bf Proof of \Cref{Prop:Lyapunov}}]

Fix $(x,v)\in\RRd$. Write $m= m_1+ \nu a\,m_2+ \tfrac{\nu a \chi}{1+\chi}m_3+ \nu B \,m_4+ \nu\,m_5 $, where $\nu,\, B>0$,
\begin{align*}
    m_1(x,v)&=e^{\nu\wx^a}  & m_2(x,v)&=x\cdot v \wx^{a-2}e^{\nu\wx^a}  & m_3(x,v)&=-\Psi\(\frac{x\cdot v}{\wx}\)\wx^{a-1}e^{\nu\wx^a}
\end{align*}
and
\begin{align*}
    m_4(x,v)&=\abs{v}^2\wx^{2a-2}e^{\nu\wx^a} & m_5(x,v)&= e^{b\abs{v}^\gamma}.
\end{align*}
To get a condition on $B$ such that the bound \eqref{bounds:m} holds, use the Cauchy-Schwarz inequality and $\Psi\leq \abs{\cdot}$ to estimate
\begin{align*}
    \abs{m_2(x,v)}& \leq \wx^{a-1}\abs{v}e^{\nu\wx^a}, & \abs{m_3(x,v)}&\leq \wx^{a-1}\abs{v}e^{\nu\wx^a},
\end{align*}
which implies
\[
m(x,v)\geq \(1-\nu a\frac{1+2\chi}{1+\chi}\wx^{a-1}\abs{v} + \nu B\wx^{2(a-1)}\abs{v}^2\)e^{\nu\wx^a} + \nu \,m_5(x,v).
\]
The first term on the right hand side is, up to a factor $e^{\nu \wx^a}$, a quadratic form in $\abs{v}\wx^{a-1}$, thus it is bounded from below when its discriminant is negative. This condition gives
\[
B>\frac{\nu a^2(1+2\chi)^2}{4(1+\chi)^2}
\]
and, defining $\delta_1=\frac{\nu a^2(1+2\chi)^2}{4(1+\chi)^2B}<1$, we have
\begin{equation}
m\geq (1-\delta_1)e^{-\nu \wx^a} +\nu \,m_5.
\end{equation}
This proves the lower bound in \eqref{bounds:m}. 
To derive the upper bound, we can drop $m_3$ (that is non negative) and bound $m_2$ by the Cauchy-Schwarz inequality, we find
\[
m(x,v)\leq (2+\nu(a^2\nu +B)\abs{v}^2\wx^{2a-2})e^{\nu \wx^a}+ \nu e^{b\abs{v}^\gamma}.
\]
We split into two cases. If $\abs{v}\leq (\tfrac{2\nu}{b})^{\frac{1-a}{a}}\wx^{1-a}$, then
\[
m(x,v)\leq (2+\nu(a^2\nu +B)(\tfrac{2\nu}{b})^{\frac{2-2a}{a}})e^{\nu \wx^a}+ \nu e^{b\abs{v}^\gamma}.
\]
If instead $\abs{v}\geq (\tfrac{2\nu}{b})^{\frac{1-a}{a}}\wx^{1-a}$, we have
\begin{align*}
m(x,v)&\leq (2+\nu(a^2\nu +B)\abs{v}^2)e^{\frac{b}{2} \abs{v}^\frac{a}{1-a}}+ \nu e^{b\abs{v}^\gamma}\leq \(2+\nu(a^2\nu +B)\(\tfrac{4}{b\gamma}\)^\frac{2}{\gamma}e^{\frac{4}{\gamma}}\)e^{b\abs{v}^\gamma}+\nu e^{b\abs{v}^\gamma},
\end{align*}
where we used $a\leq \tfrac{\gamma}{1+\gamma}$ and maximized the function $v\mapsto \abs{v}^2e^{-\frac{b}{2}\abs{v}^\gamma}$. In both cases we have
\begin{equation}\label{eq:2701}
    m(x,v)\leq \delta_2(e^{\nu \wx^a}+e^{b\abs{v}^\gamma})
\end{equation}
where $\delta_2=2 +\nu(a^2\nu+B)\max\Big\{\(\tfrac{4}{b\gamma}\)^\frac{2}{\gamma}e^{\frac{4}{\gamma}},(\tfrac{2\nu}{b})^{\frac{2-2a}{a}}\Big\}+ \nu$.

Let us now check the Lyapunov condition \eqref{LyapunovFunctionExp}. With $\cL^*$ defined in \eqref{eq:cL*}, we have
\[
(\mathcal{L}^*m_1 )(x,v) = \nu a (x\cdot v) \wx^{a-2}e^{\nu \wx^a},
\]
and
\begin{align*} 
(\mathcal{L}^* m_2)(x,v) & = \left( \abs{v}^2\wx^{a-2} - (2-a) (x\cdot v)^2 \wx^{a-4} + a\nu (x\cdot v)^2 \wx^{2(a-2)} \right) e^{\nu \wx^a}  \\
& \quad\quad\quad - \Lambda\left(\frac{x\cdot v}{\wx}\right) (x\cdot v) \wx^{a-2} e^{\nu \wx^a},\\
&  \leq \abs{v}^2\left(\wx^{a-2}+(2-a)\abs{x}^2\wx^{a-4}+ a\nu \abs{x}^2\wx^{2a-4}\right)e^{\nu \wx^a}\\
& \quad\quad\quad -(x\cdot v) \wx^{a-2} e^{\nu \wx^a}-\chi \psi\left(\frac{x\cdot v}{\wx}\right) x\cdot v \wx^{a-2} e^{\nu \wx^a},\\
&  \leq (3-a+a\nu)\abs{v}^2\wx^{2a-2}e^{\nu \wx^a}\\
& \quad\quad\quad -(x\cdot v) \wx^{a-2} e^{\nu \wx^a}-\chi \Psi\left(\frac{x\cdot v}{\wx}\right) \wx^{a-1} e^{\nu \wx^a}
\end{align*}
and
\begin{align*}
(\mathcal{L}^* m_3)(x,v) &=-\left( \Psi'\left(\frac{x\cdot v}{\wx}\right)\left(\frac{\abs{v}^2}{\wx}-\frac{(x\cdot v)^2}{\wx^3}\right)\wx^{a-1}\right)e^{\nu \wx^a}\\
&\quad-\left( (a-1)\Psi\left(\frac{x\cdot v}{\wx}\right)(x\cdot v) \wx^{a-3}+ \nu a\Psi \left(\frac{x\cdot v}{\wx}\right) (x\cdot v) \wx^{2a-3}\right)e^{\nu \wx^a}\\
& \quad - \Lambda\left(\frac{x\cdot v}{\wx}\right) \left(\wx^{a-1} \int_\Rd   \Psi\left(\frac{x\cdot v'}{\wx}\right)\M(v') \mathrm{d}v' -   \Psi\left(\frac{x\cdot v}{\wx}\right)\wx^{a-1}\right)e^{\nu \wx^a}\\
&\leq \left( 2\norm{\Psi'}_{\infty}\abs{v}^2\wx^{a-2} +(1-a)\norm{\psi}_\infty \abs{v}^2 \wx^{a-2}+ \nu a\norm{\psi}_\infty \abs{v}^2 \wx^{2a-2}\right)e^{\nu \wx^a}\\
& \quad -(1-\chi) \left(\int_\Rd   \Psi\left(\frac{x\cdot v'}{\wx}\right)\M(v') \mathrm{d}v'\right)\wx^{a-1}e^{\nu \wx^a} + (1+\chi)  \Psi\left(\frac{x\cdot v}{\wx}\right)\wx^{a-1}e^{\nu \wx^a}\\
&\leq (3-a+\nu a)\abs{v}^2\wx^{2a-2}e^{\nu \wx^a}\\
& \quad -(1-\chi) \left(\int_\Rd   \Psi\left(\frac{x\cdot v'}{\wx}\right)\M(v') \mathrm{d}v'\right)\wx^{a-1}e^{\nu \wx^a} + (1+\chi)  \Psi\left(\frac{x\cdot v}{\wx}\right)\wx^{a-1}e^{\nu \wx^a}
\end{align*}
and
\begin{align*}
(\mathcal{L}^* m_4)(x,v) & = \abs{v}^2 v\cdot\left( (2a-2) x \wx^{2a-4} + \nu a x \wx^{3a-4} \right) e^{\nu \wx^a} \\
& \quad + \Lambda\left(\frac{x\cdot v}{\wx}\right) \left(\int_\Rd\abs{v'}^2\M(v')\dv '-\abs{v}^2 \right) \wx^{2a-2} e^{\nu \wx^a}\\
&\leq (2-2a+\nu a)\abs{v}^3 \wx^{3a-3}e^{\nu \wx^a}+(1+\chi) c_{2,\gamma}\wx^{2a-2} e^{\nu \wx^a}\\
&\quad -(1-\chi)\abs{v}^2\wx^{2a-2} e^{\nu \wx^a}
\end{align*}
where $c_{2,\gamma}$ is defined in \eqref{eq:c_k}. Finally 
\[
(\mathcal{L}^* m_5)(x,v) \leq C_b - (1-\chi) e^{b|v|^\gamma}
\]
where $C_b=(1+\chi)\int_\Rd e^{b|v|^\gamma} \M(v)\dv $. We can put the inequalities together to determine $\cL^*m$. The constants in front of $m_2$ and $m_3$ in the expression of $m$ are meant to simplify respectively the terms $(x\cdot v)\wx^{a-2}e^{\nu\wx^a}$ and $\Psi\(\frac{x\cdot v}{\wx}\)\wx^{a-1}e^{\nu\wx^a}$ after the summation. We have
\begin{align*}
&(\mathcal{L}^*m)(x,v) \\
&\leq \; \nu a(3-a+\nu a)\abs{v}^2\wx^{2a-2}e^{\nu \wx^a}+ \frac{\nu a\chi(3-a+\nu a)}{1+\chi}\abs{v}^2\wx^{2a-2}e^{\nu \wx^a}\\
& -\frac{\nu a\chi(1-\chi)}{1+\chi} \left(\int_\Rd \Psi\left(\frac{x\cdot v'}{\wx}\right)\M(v') \mathrm{d}v'\right)\wx^{a-1}e^{\nu \wx^a} + \nu B(2-2a+\nu a)\abs{v}^3 \wx^{3a-3}e^{\nu \wx^a} \\
&+\nu B(1+\chi)c_{2,\gamma}\wx^{2a-2} e^{\nu \wx^a}-\nu B(1-\chi)\abs{v}^2\wx^{2a-2} e^{\nu \wx^a} +\nu C_b - \nu(1-\chi) e^{b|v|^\gamma}\\[2ex]
\leq & \;\nu C_1\abs{v}^2\wx^{2a-2}e^{\nu \wx^a}+\nu C_2\abs{v}^3 \wx^{3a-3}e^{\nu \wx^a}-\nu B(1-\chi)\abs{v}^2\wx^{2a-2} e^{\nu \wx^a}\\
&-\frac{\nu a\chi(1-\chi)}{1+\chi} \left(\int_\Rd\Psi\left(\frac{x\cdot v'}{\wx}\right)\M(v') \mathrm{d}v'\right)\wx^{a-1}e^{\nu \wx^a} +\nu B(1+\chi)c_{2,\gamma}\wx^{2a-2} e^{\nu \wx^a} \\
& +\nu C_b - \nu(1-\chi) e^{b|v|^\gamma}
\end{align*}
where $C_1= a(3-a+\nu a)\frac{1+2\chi}{1+\chi}$ and $C_2= B(2-2a+\nu a)$. Notice that, if we take $B>\tfrac{C_1}{1-\chi}$, we can estimate the first line by
\begin{align*}
\nu C_1\abs{v}^2 \wx^{2a-2}e^{\nu \wx^a} + &\nu C_2|v|^3\wx^{3a-3}e^{\nu \wx^a}-\nu B(1-\chi)\abs{v}^2\wx^{2a-2} e^{\nu \wx^a}\\
=&\left[-\nu(B(1-\chi)-C_1)(\abs{v}\wx^{a-1})^2+\nu C_2(\abs{v}\wx^{a-1})^3\right]e^{\nu \wx^a}.
\end{align*}
Let now $c>0$ to be fixed later, we need to split into two cases.

If $\abs{v}\wx^{a-1}\leq c$, then we have 
\begin{align*}
\nu C_1\abs{v}^2 \wx^{2a-2}e^{\nu \wx^a} + &\nu C_2|v|^3\wx^{3a-3}e^{\nu \wx^a}-\nu B(1-\chi)\abs{v}^2\wx^{2a-2} e^{\nu \wx^a}\\
=&\;(\abs{v}\wx^{a-1})^2\left [-\nu(B(1-\chi)-C_1)+\nu C_2\,c\right]e^{\nu \wx^a}\leq 0
\end{align*}
by taking $c>0$ small enough. If, instead, $\abs{v}\wx^{a-1}\geq c$, we can simply estimate $\wx^{a-1}\leq 1$ and $e^{\nu\wx^a}\leq e^{\nu\big(\tfrac{\abs{v}}{c}\big)^{\frac{a}{1-a}}}$ and we obtain
\begin{align*}
\nu C_1\abs{v}^2 \wx^{2a-2}e^{\nu \wx^a} + &\nu C_2|v|^3\wx^{3a-3}e^{\nu \wx^a}-\nu B(1-\chi)\abs{v}^2\wx^{2a-2} e^{\nu \wx^a}\\
&\leq \nu \left(C_1\abs{v}^2+C_2\abs{v}^3\right)e^{\nu\big(\tfrac{\abs{v}}{c}\big)^{\frac{a}{1-a}}}
\end{align*}
Thus, in both cases
\begin{align*}
\nu C_1\abs{v}^2 \wx^{2a-2}e^{\nu \wx^a} + &\nu C_2|v|^3\wx^{3a-3}e^{\nu \wx^a}-\nu B(1-\chi)\abs{v}^2\wx^{2a-2} e^{\nu \wx^a}\\
&\leq \nu \left(C_1\abs{v}^2+C_2\abs{v}^3\right)e^{\nu\big(\tfrac{\abs{v}}{c}\big)^{\frac{a}{1-a}}}.
\end{align*}
Substituting above and using \Cref{lem:Psi}, we find
\begin{align*}
(\mathcal{L}^*m)(x,v)\leq&\,\nu\left(C_1\abs{v}^2+C_2\abs{v}^3\right)e^{\nu\big(\tfrac{\abs{v}}{c}\big)^{\frac{a}{1-a}}}-\frac{\nu a\chi(1-\chi)\zeta}{2(1+\chi)} \wx^{a-1}e^{\nu \wx^a}\1_{\abs{x}\geq 1 }\\
&+\nu B(1+\chi)\wx^{2a-2} e^{\nu \wx^a}+\nu C_b - \nu(1-\chi) e^{b|v|^\gamma}\\
\leq&\,\nu\left(C_1\abs{v}^2+C_2\abs{v}^3\right)e^{\nu\big(\tfrac{\abs{v}}{c}\big)^{\frac{a}{1-a}}}+\nu C-C_3 \wx^{a-1}e^{\nu \wx^a} - \nu(1-\chi) e^{b|v|^\gamma}
\end{align*}
where $C_3= \frac{\nu a\chi(1-\chi)\zeta}{4(1+\chi)}$ and for some $C>0$ large enough. Since $0<a\leq\tfrac{\gamma}{1+\gamma}$ and $\nu>0$ small enough for $a=\tfrac{\gamma}{1+\gamma}$, the first term can be dominated by the the last and we conclude that, enlarging $C$ if necessary, and for some $\epsilon>0$ small enough we have
\begin{align*}
\mathcal{L}^* m &\leq  \nu C -\epsilon\delta_2 (\wx^{a-1}e^{\nu\wx^a} +e^{b\abs{v}^\gamma}).
\end{align*}
Finally, using \eqref{eq:2701}, we have \eqref{eq:expLyab}.
\end{proof}

Notice that if $\chi=0$ or $\chi=1$, then $C_3=0$, so we could not conclude the existence of a Lyapunov function. This is coherent with the fact that no confinement holds in this case.

\subsection{A polynomial weak Lyapunov function}

In this section we provide a family of Lyapunov functions with polynomial growth at infinity which is very inspired by \cite{MW17, EY24}. 

\begin{prop}\label{Prop:PolyLyap}
Let $k\in (1,\infty)$ and let \Cref{Hyp:psi} hold. There exists a positive constant $B$ such that $m$ defined for any $(x,v) \in \RRd$ by 
\begin{equation}\label{LyapunovFunctionPoly}
    m(x,v) = \wangle{x}^k + k x\cdot v \wangle{x}^{k-2} - \frac{k\chi}{1+\chi} \Psi\(\frac{x\cdot v}{\wangle{x}}\) \wangle{x}^{k-1} +B \wangle{v}^{2k},
\end{equation}
is positive and 
\begin{equation}\label{eq:polyLyap}
    (\cL^\ast m)(x,v) \leq C\1_R - \eps(\wangle{x}^{k-1} + \wangle{v}^{2k}), \qquad \forall(x,v)\in \RRd,
\end{equation}
for some positive constants $C,R$ and $\eps$. Moreover $m$ satisfies the bound
\begin{equation}\label{bounds}
m(x,v)\asymp  \wangle{x}^k +\wangle{v}^{2k}.
\end{equation}
\end{prop}

\begin{proof}[{\bf Proof of \Cref{Prop:PolyLyap}}]
First of all, we find a condition on $B$ so that $m$ is positive. Using $\Psi(z)\leq \abs{z}$ for all $z\in\Rd$, we have by Young's inequality
\begin{align*}
\abs*{k x\cdot v\wangle{x}^{k-2} - \frac{k\chi}{1+\chi} \Psi\(\frac{x\cdot v}{\wangle{x}}\) \wangle{x}^{k-1}}&\leq k\frac{1+2\chi}{1+\chi} \abs{v}\wangle{x}^{k-1}\\
&\leq \frac12\wangle{x}^k+ \(\frac{1+2\chi}{1+\chi}\)^k\(2k-2\)^{k-1}\wangle{v}^k,
\end{align*}
for $(x,v)\in\RRd$. This inequality gives \eqref{bounds}, provided we take
\[
B> \(\frac{1+2\chi}{1+\chi}\)^k\(2k-2\)^{k-1}.
\]
We compute separately all the different terms in \eqref{LyapunovFunctionPoly}. Let us write $m=m_1 +k\, m_2 + -\tfrac{k\chi}{1+\chi}m_3 + B\,m_4$, where
\begin{align*}
    m_1&=\wx^k, & m_2&=x\cdot v\wx^{k-2}, & m_3&=\Psi\(\frac{x\cdot v}{\wx}\)\wx^{k-1}, & m_4&=\wangle{v}^{2k}
\end{align*}
We have
\begin{align*}
    (\cL^*m_1)(x,v) &= k(x\cdot v)\wangle{x}^{k-2} - \Lambda\(\frac{x\cdot v}{\wx}\)\wangle{x}^k + \Lambda\(\frac{x\cdot v}{\wx}\)\wangle{x}^k\int_{\R^d}\M(v')\dv '\\
    &=k(x\cdot v)\wangle{x}^{k-2},
\end{align*}
\begin{align*}
    (\cL^*m_2)(x,v) &= \abs{v}^2\wangle{x}^{k-2} -(2-k) (x\cdot v)^2\wangle{x}^{k-4} - \Lambda\(\frac{x\cdot v}{\wx}\)(x\cdot v)\wangle{x}^{k-2}\\
    &=\abs{v}^2\wangle{x}^{k-2} -(2-k)  (x\cdot v)^2\wangle{x}^{k-4} - (x\cdot v)\wangle{x}^{k-2} - \chi\Psi\(\frac{x\cdot v}{\wangle{x}}\)\wangle{x}^{k-1},
\end{align*}
\begin{align*}
    (\cL^*m_3)(x,v)&=- \Psi'\left(\frac{x\cdot v}{\wangle{x}}\right)\left(\frac{\abs{v}^2}{\wangle{x}}-\frac{(x\cdot v)^2}{\wangle{x}^3}\right)\wangle{x}^{k-1} -  (k-1)\Psi\left(\frac{x\cdot v}{\wangle{x}}\right)x\cdot v \wangle{x}^{k-3}\\
    &- \Lambda\(\frac{x\cdot v}{\wx}\)\wangle{x}^{k-1}\int_{\R^d}\Psi\(\frac{v'\cdot x}{\wangle{x}}\)\M(v')\dv ' +  \Lambda\(\frac{x\cdot v}{\wx}\)\Psi\(\frac{x\cdot v}{\wangle{x}}\)\wangle{x}^{k-1}
\end{align*}
and
\begin{align}\label{eq:L*m4}
    (\cL^*m_4) &= \Lambda\(\frac{x\cdot v}{\wx}\) c_{2k,\gamma}-  \Lambda\(\frac{x\cdot v}{\wx}\)\wangle{v}^{2k}.
\end{align}
The constants in front of $m_2$ and $m_3$ in the expression of $m$ are meant to simplify respectively the term $(x\cdot v)\wx^{k-2}$ and $\Psi\(\tfrac{x\cdot v}{\wx}\)\wx^{k-1}$ after the summation. Using that $\Psi\geq 0$, we have
\begin{equation}\label{eq:2101}
    \begin{split}
        (\cL^*m)(x,v)\leq &\,  C_2\abs{v}^2\wangle{x}^{k-2}-\frac{k\chi}{1+\chi}\Lambda\(\frac{x\cdot v}{\wx}\)\wangle{x}^{k-1}\int_{\R^d}\Psi\(\frac{v'\cdot x}{\wangle{x}}\)\M(v')\dv ' \\
    &+B\Lambda\(\frac{x\cdot v}{\wx}\) c_{2k,\gamma}-  B\Lambda\(\frac{x\cdot v}{\wx}\)\wangle{v}^{2k}
    \end{split}
\end{equation}
with $C_2= 3k+k^2 + \frac{\chi k(k+1)}{1+\chi}$.
Now, if $k\in[1,2]$, then we estimate
\[
\abs{v}^2\wangle{x}^{k-2}\leq \wangle{v}^{2k},
\]
and we choose $B>\tfrac{2C_2}{1-\chi}$. If $k>2$, Young's inequality implies that for any $\delta>0$
\begin{equation}\label{eq:YoungP}
    \abs{v}^2\wangle{x}^{k-2}\leq \delta\,\frac{k-2}{k-1}\wangle{x}^{k-1} +\frac{1}{k-1} \frac{\wangle{v}^{2(k-1)}}{\delta^{\frac{1}{k-2}}}\leq \delta\,\frac{k-2}{k-1}\wangle{x}^{k-1} +\frac{\wangle{v}^{2k}}{(k-1)\delta^{\frac{1}{k-2}}} .
\end{equation}
Thus, first choosing $\delta<\tfrac{k\chi(k-1)(1-\chi)\zeta}{2\,C_2(k-2)(1+\chi)}$ and then $B>\tfrac{2\,C_2}{(1-\chi)(k-1)}\delta^{-\frac{1}{k-2}}$ we have 
\begin{align*}
    \cL^*m\leq& -\(\frac{k\chi(1-\chi)\zeta}{1+\chi}- \frac{C_2(k-2)\delta}{k-1}\)\wangle{x}^{k-1} - \(B(1-\chi)\abs{v}^{2k} -\frac{C_2}{(k-1)\delta^{\frac{1}{k-2}}}\wangle{v}^{2k}\)+ B(1+\chi)c_{2k, \gamma}\\
    \leq& -\frac{k\chi(1-\chi)\zeta}{2(1+\chi)}\wangle{x}^{k-1} -\frac{B(1-\chi)}{2}\wangle{v}^{2k}+B(1+\chi)c_{2k,\gamma}.
\end{align*}
This inequality proves \eqref{eq:polyLyap} provided $C$ and $R$ are large enough and $\eps>0$ is small enough.
\end{proof}

We conclude this section with some considerations about possible extensions. It is worth noticing that the same Lyapunov function \eqref{LyapunovFunctionPoly} can be used when the run and tumble equation is considered with fat tailed local equilibrium, that is of the form
\begin{equation}\label{eq:PolyM}
    \M(v) = c_\gamma \frac{1}{\wangle{v}^{d+\gamma}}\qquad \gamma>2.
\end{equation}
In fact, the polynomial decay \eqref{eq:PolyM} affects the proof of \Cref{Prop:PolyLyap} only at \eqref{eq:L*m4} when a moment of order $2k$ is computed. Since $k>1$, we have that the coefficient
\[
c_{2k,\gamma}=\int_\Rd \abs{v'}^{2k}\M(v')dv'
\]
is finite only if $\gamma>2$ and $1<k<\frac{\gamma}{2}$.
In such cases, we can reproduce the same proof of \Cref{Prop:PolyLyap} and find some Lyapunov functions, which furthermore can be exploited to show existence and uniqueness of a stationary state and estimate the rate of convergence towards it as in the next section.

%%%%%%%%%%%%%%%%%%%%%%%%%%%%%%%%%%%%%%%%%%%%%%%%%%%%%%%%%%%%%%%%%%%%%%%%%%%%
%%%%%%%%%%%%%%%%%%%%%%%%%%%%%%%%%%%%%%%%%%%%%%%%%%%%%%%%%%%%%%%%%%%%%%%%%%%%

\section{Existence via the Harris theorem}\label{Sec:Harris}

In this section, we give quantitative rates of convergence for the run and tumble model \eqref{eq:main} using the subgeometric version of Harris' Theorem \ref{thm:harris}. The main result of this section is the following.

\begin{prop}\label{thm:harrisapplied}
There exists a unique steady state measure $\mu_G(x,v) = G(x,v)\mathrm{d}x \mathrm{d}v$ to the run and tumble equation \eqref{eq:main}. Furthermore, 
\begin{itemize}
    \item for any $m(x,v) \asymp e^{\nu\wx^a} + e^{b\abs{v}^\gamma}$ defined as in \Cref{Prop:Lyapunov}, for any normalised $f_0$ in $\sfL^1(m)$, we have
    \begin{equation}\label{eq:ExpHarris}
    \norm{ S_\cL(t) f_0 - G}_{\sfL^1} \lesssim   e^{-a^{-a}\, t^{a}}\norm{f_0-G}_{\sfL^1(m)}, \qquad \forall t\geq 0.
    \end{equation}
    \item for any $m(x,v) \asymp \wx^k+\wangle{v}^{2k}$ as defined in \Cref{Prop:PolyLyap}, for any normalised $f_0$ in $\sfL^1(m)$, we have
    \begin{equation}\label{eq:PolyHarris}
        \norm{ S_\cL(t) f_0 - G}_{\sfL^1} \lesssim   \frac{1}{\wangle{t}^k}\norm{f_0-G}_{\sfL^1(m)}, \qquad \forall t\geq 0.
    \end{equation} 
\end{itemize}
\end{prop}

The proof directly follows from Harris theorem \ref{thm:harris}. We have already provided some Lyapunov function in \Cref{Sec:Lyapunov}, so we just have to check  \Cref{Hyp:minorisation}. This has already been shown in \cite{EY24}, we reproduce the proof for completeness.

\begin{lem}\label{Prop:positivity}
Let $m$ be defined as in \Cref{Prop:Lyapunov} or in \Cref{Prop:PolyLyap}. For every set 
\[
\mathcal{C} = \{ (x, v) \in \R^d \times \R^d \,:\, m(x,v) \leq C\},
\]
with $C>0$, there exist constants $X_0, V_0, C_0 >0$ depending on $C$ such that for any positive $f_0$ in $\sfL^1$
\[
\mathcal{S}_\cL(T) f_0 \geq \frac{C_0^2(1-\chi)^2}{4}  e^{-(1-\chi)T} 1_{|v| \leq V_0} 1_{|x| \leq X_0} \int_{\mathcal{C}} f_0(x_0,v_0) \mathrm{d}x_0\mathrm{d}v_0 
\]
for a certain $T>0$.
\end{lem}
\begin{proof}[{\bf Proof of \Cref{Prop:positivity}}]
Consider the splitting $\cL=\cB_0+\cA_1$ where 
\[
\cB_0 f = - v \cdot \nabla_x f - \Lambda f, 
\]
and 
\[
(\cA_1 f)(x,v)=\M(v)\int_{\R^d}\Lambda\(\frac{x\cdot v'}{\wx}\)f(x,v')\dv 
\]
The semigroup $S_{\cB_0}$ represents the solution to a transport equation, therefore its explicit form is
\[
(S_{\cB_0}(t)f_0)(x,v)= f_0(x-vt, v) \exp\(-\int_0^t\Lambda\(\frac{(x-vs)\cdot v}{\wangle{x-vs}}\)\d s\)
\]
Thanks to Duhamel's formula we have
\[
S_\cL=S_{\cB_0}+S_{\cB_0}\star \cA_1 S_\cL,
\]
from which we deduce that, for all positive functions $f_0$, there holds
\begin{align*}
S_\cL(t)f_0&\geq S_{\cB_0}f_0, & &\text{and} & S_\cL(t)f_0&\geq (S_{\cB_0}\star \cA_1 S_\cL)(t)f_0.
\end{align*}
As a consequence, combining both  we deduce
\[
S_\cL(t)f_0\geq (S_{\cB_0}\star (\cA_1 S_{\cB_0})\star (\cA_1 S_{\cB_0}))f_0
\]
that is
\begin{align*}
(S_\cL f_0)(x,v) \geq     \int_{0}^{t} \int_{0}^{s} (S_{\cB_0}(t-s) \cA_1 S_{\cB_0}(s-r)\cA_1  S_{\cB_0}(r) f_0)(x,v) \d r \d s.
\end{align*}
Consider $m$ and $\mathcal{C}$ as in the statement, we first prove the statement with initial state the Dirac measure $\delta_{x_0,v_0}$ with $(x_0,v_0)\in\mathcal{C}$. 
We just need to progressively compute the integrals above. Consider $0\leq r\leq s\leq t$ and $(x,v)\in\RRd$. We have
\begin{align*}
(S_{\cB_0}(r)\delta_{x_0,v_0})(x,v)&\geq e^{-(1+\chi)r}\delta_{x_0,v_0}(x-vr,v)\geq e^{-(1+\chi)r}\delta_{x_0+rv_0,v_0}(x,v),
\end{align*}
then 
\begin{align*}
(\cA_1 S_{\cB_0}(r)\delta_{x_0,v_0})(x,v)&\geq (1-\chi)e^{-(1+\chi)r}\M(v)\int_{\R^d}\delta_{x_0+rv_0,v_0}(x,v')\dv '\\
&\geq (1-\chi)e^{-(1+\chi)r}\M(v)\delta_{x_0+rv_0}(x).
\end{align*}
Next,
\begin{align*}
(S_{\cB_0}(s-r)\cA_1 S_{\cB_0}(r)\delta_{x_0,v_0})(x,v)&\geq (1-\chi)e^{-(1+\chi)(s-r)}e^{-(1+\chi)r}\M(v)\delta_{x_0+rv_0}(x-v(s-r))\\
&\geq (1-\chi)e^{-(1+\chi)s}\M(v)\delta_{x_0+rv_0}(x-v(s-r))
\end{align*}
and
\begin{align*}
(\cA_1 S_{\cB_0}(s-r)\cA_1 S_{\cB_0}(r)\delta_{x_0,v_0})(x,v)&\geq (1-\chi)^2e^{-(1+\chi)s}\M(v)\int_{\R^d}\M(v')\delta_{x_0+rv_0}(x-v'(s-r))\dv '\\
&\geq (1-\chi)^2e^{-(1+\chi)s}\M(v)\M\(\frac{x-x_0-rv_0}{s-r}\)
\end{align*}
and finally
\begin{align*}
(S_{\cB_0}(t-s)\cA_1 S_{\cB_0}(s-r)\cA_1 S_{\cB_0}(r)\delta_{x_0,v_0})(x,v)&\geq (1-\chi)^2e^{-(1+\chi)t}\M(v)\M\(\frac{x-v(t-s)-x_0-rv_0}{s-r}\).
\end{align*}
Thanks to \eqref{bounds:m} and \eqref{bounds}, $m$ has pre-compact level sets, so we can fix two constants $X_0,V_0>0$ large enough such that
\[
\mathcal{C}\subseteq B_{X_0}\times B_{V_0}.
\]
Fix the constant $C_0>0 $ such that $\M(v)\geq C_0\1_{\abs{v}\leq V_0}$, then we have
\[
(S_{\cB_0}(t-s)\cA_1 S_{\cB_0}(s-r)\cA_1 S_{\cB_0}(r)\delta_{x_0,v_0})(x,v)\geq C_0^2(1-\chi)^2e^{-(1+\chi)t}\1_{\abs{v}\leq V_0}\1_{\abs{x-v(t-s)-x_0-rv_0}\leq V_0(s-r)}
\]
Let us restrict to the cases $r\leq \tfrac12$, $(s-r)\geq 1+2\tfrac{X_0}{V_0}$,  and $(t-s)\leq \tfrac12$, then
\begin{align*}
(s-r)V_0- \abs{(t-s)v-x_0-r v_0}&\geq(s-r)V_0- (t-s)\abs{v}-\abs{x_0}-r\abs{v_0}\\
&\geq V_0+2X_0-\frac{1}{2}V_0-X_0-\frac{1}{2}V_0\\
&\geq X_0.
\end{align*}
This means that, for such $t,s,r$, we have the implication
\begin{align*}
\abs{x}\leq X_0&\qquad \Longrightarrow \qquad \abs{x}\leq (s-r)V_0- \abs{(t-s)v-x_0-r v_0}\\
&\qquad \Longrightarrow \qquad \abs{x-v(t-s)-x_0-rv_0}\leq V_0(s-r)
\end{align*}
which allows to deduce
\[
(S_{\cB_0}(t-s)\cA_1 S_{\cB_0}(s-r)\cA_1 S_{\cB_0}(r)\delta_{x_0,v_0})(x,v)\geq C_0^2(1-\chi)^2e^{-(1+\chi)t}\1_{\abs{v}\leq V_0}\1_{\abs{x}\leq X_0}
\]
for such $t,s,r$. Therefore let us set $T = 2+2\tfrac{X_0}{V_0}$. Then we can restrict the time integrals to $r \in \(0,\tfrac12\)$, $s \in \(T-\tfrac{1}{2}, T\)$. Then we get 
\begin{align*}
S_\cL(T)\delta_{x_0,v_0}&\geq \int_{0}^{T} \int_{0}^{s} (S_{\cB_0}(t-s)\cA_1 S_{\cB_0}(s-r)\cA_1 S_{\cB_0}(r)\delta_{x_0,v_0})\d r \d s \\
&\geq C_0^2(1-\chi)^2 e^{-(1+\chi)T}  \int_{T-\tfrac{1}{2}}^{T} \int_{0}^{1/2} \1_{\abs{v}\leq V_0}\1_{\abs{x}\leq X_0} \d r \d s\\
&\geq \frac{1}{4} C_0^2(1-\chi)^2 e^{-(1+\chi)T}\1_{\abs{v}\leq V_0}\1_{\abs{x}\leq X_0}
\end{align*}
We can extend this inequality to more general initial measure by using the fact that $S_\cL$ is a Markov semigroup. More explicitly, since $S_\cL(T)$ is positive and mass preserving, by \cite[Prop. 1.2.3]{BGL13} it can be represented by a probability kernel $p$, i.e.
\[
S_\cL(T)\mu =\iint_\RRd p(x',v',\cdot) \d \mu(x',v')
\]
for all positive measures $\mu$. In particular, for any $(x_0,v_0)\in\RRd$,
\[
S_\cL(T)\delta_{x_0,v_0} = p(x_0,v_0,\cdot).
\]
Therefore, for any positive measure $\mu$, we have
\begin{align*}
    S_{\mathcal{L}}(T) \mu = \iint_\RRd \( S_{\mathcal{L}}(T) \delta_{x', v'} \) \d\mu(x', v')  \geq \frac{C_0^2(1-\chi)^2}{4}  e^{-(1+\chi)T}\1_{\abs{v}\leq V_0}\1_{\abs{x}\leq X_0} \mu(\mathcal{C}) 
\end{align*}
Hence the minorisation condition \eqref{con:minorisation} holds by choosing $T=2+2\tfrac{X_0}{V_0}$, $\mu_\ast = \tfrac{1}{\abs{B_{V_0}}\abs{B_{X_0}}}\1_{\abs{v}\leq V_0}\1_{\abs{x}\leq X_0}$ and $\alpha = \tfrac{C_0^2(1-\chi)^2}{4}  e^{-(1+\chi)T}\abs{B_{V_0}}\abs{B_{X_0}}$.
\end{proof}
%%%%%%%%%%%%%%%%%%%%%%%%%%%%%%%%%%%%%%%%%%%%%%%%%%%%%%%%%%%%%%%%%%%%%%%%%%%%%%%%%%%%%%%%%%%%%%%%%%%%%%%%%%%%%%%%%%%%%%%%%%%%%%%%%%%%%%%%%%%%%%%%%%%%%%%%%%%%%%%%%%%%%%%%%%%%%%%%%%%%%%%%%%%%%%

\begin{proof}[\bf Proof of \Cref{thm:harrisapplied}]
Harris' theorem \ref{thm:harris} requires the minorisation condition \Cref{Hyp:minorisation}, which has been proved in \Cref{Prop:positivity}, and the weak Lyapunov condition \ref{Hyp:Lyapunov}, which has been proved \Cref{Sec:Lyapunov}. Consider first the function $m\asymp e^{\nu\wx^a} + e^{b\abs{v}^\gamma} $ as  defined in \Cref{Prop:Lyapunov}. Then we have, on $\RRd$,
\[
e^{\frac\nu2 \wx^a}\leq e^{\frac\nu2 \wx^a + \frac b 2\abs{v}^\gamma}\lesssim e^{\nu\wx^a}+ e^{b\abs{v}^\gamma} \lesssim m(x,v),
\]
so that 
\[
\wx^a\lesssim \ln(m).
\]
Hence, thanks to \eqref{eq:expLyab}, we have that weak Lyapunov condition \ref{Hyp:Lyapunov} holds with $\phi(m) = m\, (\ln m)^{-\tfrac{1-a}{a}}$ and the corresponding convergence rate is exactly \eqref{eq:ExpHarris}. If we consider the function $m  \asymp \wx^k+\wangle{v}^{2k}$ as defined in \Cref{Prop:PolyLyap}, then the 
weak Lyapunov condition holds with $\phi(m) = m^{1-\tfrac{1}{k}}$, and the corresponding convergence rate is \eqref{eq:PolyHarris}.

Lastly, we comment on the fact that the measure $\mu_G$ we find has a density. We refer to the proof of \Cref{thm:harris} in \cite{CM21} and we observe that the existence of the steady state measure comes from constructing a Cauchy sequence in a weighted total variation norm. If we work in the case where \Cref{Hyp:minorisation} is verified with a lower bound which has a density (as in our case) then we can repeat exactly the same argument but working in weighted $\sfL^1$ rather than weighted total variation. The completeness of $\sfL^1$ then gives us that the steady state is an $\sfL^1$ function. 
\end{proof}

\section{Existence via a splitting $\cL = \cA + \cB$ and the bound above on $G$ in $\sfL^\infty$}\label{sec:Duhamel}

In this section we study the semigroup $S_\cL$ through the same splitting strategy of \cite{MW17} and which is based on the Duhamel formula. Consider the following operators
\begin{equation}
    (\cA f)(x,v):=\M (v)\eta_R(x) \left(\int_{\R^d}\Lambda\left(\frac{x\cdot v'}{\wx}\right)f(x,v')\dv '\right) , 
\end{equation}
and
\begin{equation}
    (\cB  f)(x,v) = - v \cdot \nabla_x f(x,v) - \Lambda\left(\frac{x\cdot v}{\wx}\right) f(x,v) + (1- \eta_R(x))   \M (v)\(\int_{\R^d} \Lambda\left(\frac{x\cdot v'}{\wx}\right) f(x,v') \mathrm{d}v'\),
\end{equation}
where $\eta_R\colon \Rd\to\R$ is a smooth cut-off function such that $\1_{B(0,R)} \leq \eta_R \leq \1_{B(0,2R)}$, with $R>0$ large to be fixed later. Then we can write the generator of the run and tumble equation as $\cL = \cA + \cB$. The main interest of this splitting is that $\cA$ will enjoy good bounds in suitable weighted spaces, and $\cB$ is dissipative. The main result of this section is the following
\begin{prop}\label{Prop:SL_bounded}
Let $m(x,v)\asymp e^{\nu\wx^a}+e^{b\abs{v}^\gamma}$ be defined as in \Cref{Prop:Lyapunov}. The semigroup $S_\cL$ is bounded in $X:=\sfL^1(m)\cap\sfL^\infty(m)$, that is
\begin{equation}\label{SL bounded}
\norm{S_\cL(t)}_{X\to X}\lesssim 1.
\end{equation}
As a consequence, there exists a unique, normalised, invariant by rotation, steady state $G$ in $X$.
\end{prop}
This estimate is carried out through the Duhamel formula
\[
S_\cL =S_{\cB}+ \sum_{j=1}^{d+1} S_{\cB}\star(\cA S_{\cB})^{\star j} + S_{\cB}\star (\cA S_{\cB})^{\star(d+1)} \star \cA S_\cL
\]
and for this reason it requires some preliminary estimates on $S_\cB$ and $\cA$. In \Cref{SubSec:Duhamel1} we derive all the needed estimates about $S_\cB$ and in \Cref{SubSec:Duhamel2} we use them to study $S_\cL$.

\subsection{The semi-group $S_{\cB }$}\label{SubSec:Duhamel1}

The semi-group $S_{\cB }$ will be studied through a further splitting $\cB =\cB_0+\cA_0$ where 
\begin{equation}\label{eq:SB0}
    (\cB _0 f)(x,v) = - v \cdot \nabla_x f(x,v) - \Lambda\left(\frac{x\cdot v}{\wx}\right) f(x,v),\qquad \forall (x,v)\in \RRd 
\end{equation}
and 
\begin{equation}\label{eq:A0}
    (\cA_0 f)(x,v)=(1- \eta_R(x))\M (v)  \int_{\R^d} \Lambda\left(\frac{x\cdot v'}{\wx}\right) f(x,v') \mathrm{d}v' ,\qquad \forall (x,v)\in \RRd.
\end{equation}
Hence, let us first provide some useful lemmas concerning $\cB_0$ and $\cA_0$.

\begin{lem}\label{estimate:S_B0}
Let $m\asymp e^{\nu\wx^a}+e^{b\abs{v}^\gamma}$ be defined as in \Cref{Prop:Lyapunov}. If $\nu>0$ is small enough, then for all $1\leq p\leq \infty$ there holds
\begin{align}\label{SB_0 and A_0}
    \|S_{\cB _0}(t)\|_{\sfL^p (m)\to \sfL^p (m)}&\leq e^{-\frac{1}{2}(1-\chi) t} & \|\cA_0 \|_{\sfL^p (m)\to \sfL^p (m)}&\lesssim 1.
\end{align}

\end{lem}
\begin{proof}[{\bf Proof of \Cref{estimate:S_B0}}]
Let us consider the weight $\widetilde{m}(x,v)=e^{\nu \wx^a} + e^{b \abs{v}^\gamma}$ with $a\in(0,\tfrac{\gamma}{1+\gamma}]$ and $b\in(0,\tfrac{1}{\gamma})$ and notice that it generates an equivalent norm on $\sfL^p (m)$. If $f = S_{\cB_0}(t) f_0$ then we have
\begin{align*}
\frac{\mathrm{d}}{\mathrm{d}t} \int_{\R^d} |f|^p \widetilde{m}^p \,\dx\dv &= p\int_{\R^d} |f|^p \widetilde{m}^{p-1} \left( -v \cdot\nabla_x \widetilde{m} - \Lambda\widetilde{m} \right)\,\dx\dv  \\
&= p\int_{\R^d} |f|^p \widetilde{m}^{p-1} \left( \nu a\, x\cdot v \wx^{a-2} e^{ \nu \wx^a} - \Lambda \widetilde{m} \right)\,\dx\dv \\
&\leq p\int_{\R^d} |f|^p \widetilde{m}^{p-1} \left( \nu a\, \abs{v}  \wx^{a-1} e^{ \nu \wx^a} - \Lambda \widetilde{m} \right)\,\dx\dv .
\end{align*}
Let $ 1\leq c\leq \frac{1-\chi}{2\nu a}$, we distinguish the following two parts of the space. When $|v| \leq c\wx^{1-a}$ we have
\begin{align*}
\nu a\, \abs{v}  \wx^{a-1} e^{ \nu \wx^a} - \Lambda \widetilde{m} &\leq \nu a c e^{ \nu \wx^a}  - (1-\chi) \widetilde{m}\leq  -(1-\chi-\nu c a)\widetilde{m} \leq - \frac{1}{2}(1-\chi)\widetilde{m}.
\end{align*}
Consider now the part of the phase space where $\abs{v}\geq c\wx^{1-a}$, we have
\[
\nu a\, \abs{v}  \wx^{a-1} e^{ \nu \wx^a} - \Lambda \widetilde{m} \leq \nu a\abs{v}e^{\nu c^{-\frac{a}{1-a}}\abs{v}^\frac{a}{1-a}}- (1-\chi) \widetilde{m}.
\]
Since $a\leq \tfrac{\gamma}{1+\gamma}$, then $\tfrac{a}{1-a}\leq \gamma$, and since $\abs{v}\geq 1$ in this part of the space we have
\begin{align*}
\nu a\, \abs{v}  \wx^{a-1} e^{ \nu \wx^a} - \Lambda \widetilde{m} &\leq \nu a\abs{v}e^{\nu c^{-\frac{a}{1-a}}\abs{v}^\gamma}- (1-\chi) (e^{\nu\wx^a}+e^{b\abs{v}^\gamma}) \leq -\frac{1}{2}(1-\chi)\widetilde{m},
\end{align*}
if $\nu>0$ is small enough. In both cases we have
\[
\ddt \int_{\R^d} |f|^p \widetilde{m}^p\,\dx\dv   \leq -\frac{p}{2}(1-\chi) \int_{\R^d} |f|^p \widetilde{m}^p\,\dx\dv 
\]
and the first estimate in \eqref{SB_0 and A_0} is proved. 

Concerning $\cA_0$, it suffices to compute
\begin{align*}
\|\cA_0 f\|_{L^p(\widetilde{m})}^p&=\int_{\R^d \times \R^d} (1- \eta_R(x))^p\M (v)^p  \left(\int_{\R^d} \Lambda\left(\frac{x\cdot v'}{\wx}\right) f(x,v') \mathrm{d}v'\right)^p(e^{\nu\wx^a}+e^{b\abs{v}^2})^p \dx\dv \\
&\leq (1+\chi)^p\int_{\R^d \times \R^d} \(\M(v)e^{\nu\wx^a}+\M(v)e^{b\abs{v}^2}\)^p\left(\int_{\R^d} \abs{f}^p\dv '\right)\dx\dv \\
&\leq (1+\chi)^p\int_{\R^d \times \R^d} \(\M(v)e^{\nu\wx^a+b\abs{v}^2}+\M(v)e^{\nu\wx^a+ b\abs{v}^2}\)^p\left(\int_{\R^d} \abs{f}^p\dv '\right)\dx\dv \\
&\leq (1+\chi)^p\int_{\R^d \times \R^d} 2^p\(\M(v)e^{\nu\wx^a+b\abs{v}^2}\)^p\left(\int_{\R^d} \abs{f}^p\dv '\right)\dx\dv \\
&\leq 2^p(1+\chi)^p\int_{\R^d \times \R^d} \(\M(v)e^{b\abs{v}^2}\)^p\left(\int_{\R^d} \abs{f}^pe^{p\nu\wx^a}\dv '\right)\dx\dv \\
&= 2^p(1+\chi)^p\left(\int_{\R^d} \(\M(v)e^{b\abs{v}^2}\)^p \dv \right)\left(\int_{\R^d \times \R^d} \abs{f}^p e^{p\nu\wx^a}\dv '\dx\right).\\
&= 2^p(1+\chi)^p\Vert \M e^{b (\cdot)^2} \Vert_{\sfL^p}^p \|f\|^p_{\sfL^p(m)}.
\end{align*}
and also the second estimate of \eqref{SB_0 and A_0} is proved.
\end{proof}

\begin{lem}\label{iterates Duhamel}
Let $a\in\left(0,\tfrac{\gamma}{1+\gamma}\right]$, $b\in\(0,\tfrac{1}{\gamma}\)$, $\lambda\in(0,1-\chi)$ and $0<\delta<b$ be fixed. Consider the weight $\omega(x,v)=e^{\nu\wx^a +b|v|^\gamma}$  with $0<\nu\leq \left(\frac{1}{a}\right)^{a}\left(\frac{1}{1-a}\right)^{1-a}\lambda^{a}\delta^{1-a}$. We have the following estimates
\begin{equation}\label{AS_B 1}
\|\cA_0 S_{\cB _0}(t)\|_{\sfL^1_x\sfL^\infty_v(\omega)\longrightarrow \sfL^\infty(\omega)}\lesssim t^{-d}e^{-(1-\chi-\lambda)t},
\end{equation}
\begin{equation}\label{AS_B 2}
\|\cA_0 S_{\cB _0}(t)\|_{\sfL^1_x\sfL^\infty_v(\omega)\longrightarrow\sfL^1_x\sfL^\infty_v(\omega)}\lesssim e^{-(1-\chi-\lambda)t},
\end{equation}
\begin{equation}\label{AS_B 3}
\|\cA_0 S_{\cB _0}(t)\|_{\sfL^\infty(\omega)\longrightarrow \sfL^\infty(\omega)}\lesssim e^{-(1-\chi-\lambda)t},
\end{equation}
\begin{equation}\label{AS_B 4}
\|\cA_0 S_{\cB _0}(t)\|_{\sfL^1(\omega)\longrightarrow \sfL^1(\omega)}\lesssim e^{-(1-\chi-\lambda)t}
\end{equation}
and
\begin{equation}\label{AS_B 5}
\|\cA_0 S_{\cB _0}(t)\|_{\sfL^1_x\sfL^\infty_v(\omega)\longrightarrow \sfL^1(\omega)}\lesssim e^{-(1-\chi-\lambda)t}.
\end{equation}
As a consequence, we also have
\begin{equation}\label{AS_B 6}
\|(\cA_0 S_{\cB _0}(t))^{\star n}\|_{\sfL^1_x\sfL^\infty_v(\omega)\longrightarrow \sfL^\infty(\omega)}\lesssim t^{-(d-n+1)}e^{-(1-\chi-\lambda)t}
\end{equation}
\begin{equation}\label{AS_B 7}
\|(\cA_0 S_{\cB _0}(t))^{\star n}\|_{\sfL^1_x\sfL^\infty_v(\omega)\longrightarrow \sfL^1(\omega)}\lesssim e^{-(1-\chi-\lambda)t}
\end{equation}
for every $n\in\NN$. The same inequalities also hold with $\cA$ in place of $\cA_0$ and in \eqref{AS_B 6} and \eqref{AS_B 7} different combinations of $\cA S_{\cB_0}$ and $\cA_0 S_{\cB_0}$ can be used.
\end{lem}
\begin{proof}[{\bf Proof of \Cref{iterates Duhamel}}]
The method of characteristics gives the representation formula
\[
(S_{\cB _0}(t)f_0)(x, v) =f_0(x-vt, v)\exp\(-\int_0^t\Lambda\(\frac{(x-vs)\cdot v}{\wangle{x-vs}}\)\d s\) \qquad \forall(x,v)\in\RRd,
\]
hence
\begin{align*}
(\cA_0 S_{\cB _0}(t)f_0)(x, v)&=\left(\int_{\R^d}\Lambda(x,v')f_0(x-v't, v')\exp\(-\int_0^t\Lambda\(\frac{(x-v's)\cdot v'}{\wangle{x-v's}}\)\d s\)\dv '\right) \M (v)\\
&\lesssim e^{-(1-\chi)t}e^{-\frac{|v|^\gamma}{\gamma}}\int_{\R^d} f_0(x-v't,v')\dv '\\
&\lesssim e^{-(1-\chi)t}e^{-\frac{|v|^\gamma}{\gamma}}\int_{\R^d} \sup_{w\in\R^d}|f_0(x-v't,w)e^{\nu\wangle{x-v't}^a+b|w|^\gamma}|e^{-\nu\wangle{x-v't}^{a}-b|v'|^\gamma}\dv '.
\end{align*}
By convexity we have
\[
\wangle{x-vt}\geq \wx-\frac{x\cdot v}{\wx}t\geq \wx -t\abs{v}
\]
and by subadditivity and the Young inequality we have
\[
\nu\wx^{a}\leq\nu\wangle{x-v't}^{a} + \nu t^a\abs{v'}^a\leq \nu\wangle{x-v't}^{a}+ \lambda t + \delta \abs{v'}^{\frac{a}{1-a}}.
\]
As a consequence
\begin{align*}
(\cA_0 S_{\cB _0}(t)f_0)(x, v)e^{\nu\wx^{a}+\frac{|v|^\gamma}{\gamma}}&\lesssim e^{-(1-\chi-\lambda)t}\int_{\R^d} \sup_{w\in\R^d}|f_0(x-v't,w)e^{\nu\wangle{x-v't}^a+b|w|^\gamma}|e^{-(b-\delta)\abs{v'}^\gamma}\dv '
\end{align*}
To prove inequality \eqref{AS_B 1}, we estimate $e^{-(b-\delta)\abs{v'}^\gamma}\leq 1$ and we make the change of variables $z=x-v't$ to get
\[
\norm{\cA_0 S_{\cB _0}(t)f_0}_{\sfL^\infty\left(e^{\nu\wx^{a}+\frac{|v|^\gamma}{\gamma}}\right)}\lesssim t^{-d} e^{-(1-\chi-\lambda)t}\norm{f}_{\sfL^1_x\sfL^\infty_v\left(e^{\nu\wx^{a}+b|v|^\gamma}\right)}.
\]
To prove inequality \eqref{AS_B 2} we simply integrate and change variables $z=x-v't$ and $u=v'$
\begin{align*}
\norm{\cA_0 S_{\cB _0}(t)f_0}_{\sfL^1_x\sfL_v^\infty\left(e^{\nu\wx^{a}+\frac{|v|^\gamma}{\gamma}}\right)}&\lesssim e^{-(1-\chi-\lambda)t} \iint_{\R^d\times\R^d}\sup_{w\in\R^d}|f_0(x-v't,w)e^{\nu\wangle{x-v't}^{a}+b|w|^\gamma}|e^{-(b-\delta)\abs{v'}^\gamma}\dv '\dx\\
&\lesssim e^{-(1-\chi-\lambda)t} \iint_{\R^d\times\R^d}\sup_{w\in\R^d}|f_0(z,w)e^{\nu\wangle{z}^{a}+b|w|^\gamma}|e^{-(b-\delta)\abs{u}^\gamma}\d u\d z\\
&\lesssim e^{-(1-\chi-\lambda)t} \int_{\R^d}\sup_{w\in\R^d}|f_0(z,w)e^{\nu\wangle{z}^{a}+b|w|^\gamma}|\d z.
\end{align*}
To prove inequality \eqref{AS_B 3} we compute
\begin{align*}
\sup_{x,v\in\R^d}\abs*{(\cA_0 S_{\cB _0}(t)f_0)(x, v)e^{\nu\wx^{a}+\frac{|v|^\gamma}{\gamma}}}&\lesssim e^{-(1-\chi-\lambda)t} \sup_{z,w\in\R^d}\abs*{f_0(z,w)e^{\nu\wangle{z}^{a}+b|w|^\gamma}}\int_{\R^d}e^{-(b-\delta)\abs{v'}^\gamma}\dv '\\
&\lesssim e^{-(1-\chi-\lambda)t} \sup_{z,w\in\R^d}\abs*{f_0(z,w)e^{\nu\wangle{z}^{a}+b|w|^\gamma}}.
\end{align*}
To prove inequality \eqref{AS_B 4}, we notice that 
\begin{align*}
\abs*{\cA_0 S_{\cB_0}(t)f_0(x,v)}\omega(x,v)&\lesssim e^{-(1-\chi)t}\M(v)e^{b\abs{v}^\gamma}\int_{\R^d}\abs{f_0(x-v't,v')}e^{\nu\wx^a}\dv '\\
&\lesssim e^{-(1-\chi-\lambda)t}\M(v)e^{b\abs{v}^\gamma}\int_{\R^d}\abs{f_0(x-v't,v')}e^{\nu\wangle{x-v't}^a+\delta\abs{v'}^\gamma}\dv ',
\end{align*}
so integrating in $x$ and $v$ we find
\begin{align*}
\norm{\cA_0 S_{\cB_0}(t)f_0}_{\sfL^1(\omega)}&\lesssim e^{-(1-\chi-\lambda)t}\(\int_{\R^d}\M(v)e^{b\abs{v}^\gamma}\dv \)\iint_{\R^d\times\R^d} \abs{f_0(x-v't,v')}e^{\nu\wangle{x-v't}^a+b\abs{v}^\gamma}\dx\dv '\\
&\lesssim e^{-(1-\chi-\lambda)t}\norm{f_0}_{\sfL^1(\omega)}.
\end{align*}
To prove inequality \eqref{AS_B 5} we just integrate
\begin{align*}
\norm{\cA_0 S_{\cB_0}(t)&f_0}_{\sfL^1\(e^{\nu\wx^a+b\abs{v}^\gamma}\)}\\
&\lesssim e^{-(1-\chi-\lambda)t}\(\int_{\R^d}e^{b\abs{v}^\gamma}\M(v)\dv \)\iint_{\R^d\times\R^d}\sup_{w\in\R^d}\abs{f_0(x-v't,w)e^{\nu\wangle{x-v't}^a+b\abs{v'}^\gamma}}e^{-(b-\delta)\abs{v'}^\gamma}\dx\dv \\
&\lesssim e^{-(1-\chi-\lambda)t}\int_{\R^d}\sup_{w\in\R^d}\abs*{f_0(z,w)e^{\nu\wangle{z}^a+b\abs{v'}^\gamma}}\dx
\end{align*}
Finally \eqref{AS_B 6} is a direct consequence of \eqref{AS_B 1}, \eqref{AS_B 2} and \eqref{AS_B 3} and \eqref{AS_B 7} is consequence of \eqref{AS_B 2}, \eqref{AS_B 4} and \eqref{AS_B 5}; see for example \cite[Proposition 2.5]{MQT}.
\end{proof}

\begin{lem}[Dissipativity of $S_{\cB }$]\label{lem:dissipSB}
Let $m(x,v)\asymp e^{\nu\wx^a}+e^{b\abs{v}^\gamma}$ be defined as in \Cref{Prop:Lyapunov}. Once $R>0$ is chosen sufficiently large, $S_{\cB }$ is dissipative in $\sfL^1(m)\cap \sfL^p (m)$ for every $1\leq p\leq \infty$, namely, 
\begin{equation}\label{dissipSB}
\forall t \in \R^+, \qquad \|S_{\cB }(t)\|_{\sfL^1(m)\cap \sfL^p (m) \to\sfL^1 (m)\cap \sfL^p (m)}\lesssim 1.
\end{equation}
\end{lem}
\begin{proof}[{\bf Proof of \Cref{lem:dissipSB}}]

$\#$\textit{Step 1.} The case $p=1$. The dual operator of $\cB$ is defined by
\[
(\cB^*\phi)(x,v)=(\mathcal{L}^*\phi)(x,v)-\eta_R(x)\Lambda(x,v') \int_{\R^d} \M(v')\phi(x,v')\dv '
\]
for every $\phi\in W^{1,\infty}(\RRd)$. Let $R>0$ large enough according to the Lyapunov condition as in \Cref{Prop:Lyapunov}. We have
\begin{align*}
(\cB ^*m)(x,v) &\leq \nu C\1_{B_R}(x,v) -\epsilon\wx^{a-1} m-\eta_R(x)(1-\chi)\int_{\R^d} \left[(1-\delta_1)e^{\nu\wx^a}+ \nu e^{b\abs{v'}^2}\right]\M(v')\dv '\\
&\leq \nu C\1_{B_R}(x,v) -\epsilon\wx^{a-1} m-\eta_R(x)(1-\chi)(1-\delta_1)\int_{\R^d} e^{\nu\wx^a}\M(v')\dv '\\
&\leq \nu C\1_{B_R}(x,v) -\epsilon\wx^{a-1} m-\eta_R(x)(1-\chi)(1-\delta_1)
\end{align*}
Remembering that $\delta_1=\frac{\nu a^2(2+3\chi)^2}{16(1+\chi)^2B} $, we can take $\nu>0$ small enough such that $\nu C\leq (1-\chi)(1-\delta_1)$, and we conclude
\[
\cB ^*m\leq -\epsilon\wx^{a-1} m \leq 0.
\]
and it follows that
\begin{equation}\label{dissipSB L1}
\forall t \in \R^+, \qquad\|S_\cB(t)\|_{\sfL^1(m)\to\sfL^1(m)}\leq 1.
\end{equation}

$\#$\textit{Step 2.} The case $p=\infty$. 
Let us consider the splitting $\cB =\cB_0+\cA_0$ defined in \eqref{eq:SB0} and \eqref{eq:A0}.
From Duhamel's formula we have
\begin{equation*}
S_\cB =S_{\cB_0}+ \sum_{j=1}^{d+1} S_{\cB_0}\star(\cA_0 S_{\cB_0})^{\star j} + S_{\cB_0}\star (\cA_0 S_{\cB_0})^{\star(d+1)} \star \cA_0 S_\cB.
\end{equation*}
Thanks to \Cref{estimate:S_B0} we have 
\begin{equation}\label{est1}
\|S_{\cB_0}\star(\cA_0 S_{\cB_0})^{\star j}\|_{\sfL^\infty(m)\to L^\infty(m)}\leq \|S_{\cB_0}\|_{\sfL^\infty(m)\to \sfL^\infty(m)}\star  \|\cA_0 S_{\cB_0}\|^{\star j}_{\sfL^\infty(m)\to \sfL^\infty(m)}
\lesssim e^{-\frac{1}{2}(1-\chi)t}
\end{equation}
because all the convolution terms are exponentially decaying. Writing this inequality for all the integers $j\in\Drange{0,d+1}$, we see that
\begin{equation}\label{Duhamel1}
\| S_{\cB_0}\|_{\sfL^\infty(m)\to \sfL^\infty(m)}+ \sum_{j=1}^{d+1} \|S_{\cB_0}\star(\cA_0 S_{\cB_0})^{\star j}\|_{\sfL^\infty(m)\to \sfL^\infty(m)}\lesssim e^{\frac{1}{2}(1-\chi)t}
\end{equation}
Concerning the last term of the Duhamel formula, consider the weight $\omega(x,v)=e^{\nu\wx^{a} +b|v|^\gamma}$ as defined in \Cref{iterates Duhamel}. We have
\begin{align}
\|\cA_0 S_{\cB}(t)f\|_{\sfL^1_x \sfL^\infty_v(\omega)}&=\int_{\R^d} \sup_{v\in\R^d} \left\{\abs{(\cA_0 S_\cB(t)f)(x,v)}e^{\nu\wx^{a} +b|v|^\gamma}\right\}\dx \notag\\
&\leq \iint_{\R^d \times \R^d}(1-\eta_R(x))\Lambda(x,v')\sup_{v\in\R^d}\left\{\M( v )e^{\nu\wx^{a} +b|v|^\gamma} \right\}\abs{(S_\cB (t) f)(x,v)}\dx\dv' \notag\\
&\lesssim\norm{\M( v )e^{b|v|^2} }_{\sfL^\infty}\iint_{\R^d \times \R^d} \abs{(S_\cB (t) f)(x,v)}e^{\nu\wx^{a}}\dx\dv' \notag\\
&\lesssim \iint_{\R^d \times \R^d} \abs{(S_\cB (t) f)(x,v)}m(x,v')\dx\dv ' \label{hereDissipativity}\\
&\lesssim\|f\|_{\sfL^1(m)},\notag
\end{align}
where in the last step we used the dissipativity in $\sfL^1(m)$. Moreover, thanks to \eqref{estimate:S_B0} and the fact that $m\leq \delta_2 (e^{\nu\wx^a}+e^{b\abs{v}^\gamma})\leq 2\delta_2\, e^{\nu\wx^a+b\abs{v}^\gamma}$, we also have
\begin{align*}
\|S_{\cB_0} (t)f\|_{\sfL^\infty (m)}&\leq e^{-\frac{1}{2}(1-\chi)t}\|f\|_{\sfL^\infty (m)} \\
&\lesssim e^{-\frac{1}{2}(1-\chi)t}\|f\|_{\sfL^\infty (\omega)}
\end{align*}
Now, the fact that 
\begin{align*}
\|\cA_0 S_\cB(t)\|_{\sfL^1(m)\to\sfL^1_x\sfL^\infty_v(\omega)}\lesssim 1\in \sfL_t^\infty([0,+\infty))
\end{align*}
and
\[
\|S_{\cB_0}(t)\|_{\sfL^\infty(\omega)\to \sfL^\infty(m)}\lesssim e^{-\frac{1}{2}(1-\chi)t}\in \sfL^1_t([0,+\infty))
\]
proved few lines above, and
\[
\|(\cA_0 S_{\cB_0})^{\star(d+1)}\|_{\sfL^1_x\sfL^\infty_v(\omega)\to \sfL^\infty(\omega)}\lesssim e^{-(1-\chi-\gamma)t}\in \sfL^1_t([0,+\infty))
\]
proved in \Cref{iterates Duhamel}, imply that
\begin{equation}\label{Duhamel2}
\|S_{\cB_0}\star (\cA_0 S_{\cB_0})^{\star(d+1)} \star \cA_0 S_\cB\|_{\sfL^1(m)\to \sfL^\infty(m)}\lesssim 1
\end{equation}
Using \eqref{Duhamel1} and \eqref{Duhamel2} in the Duhamel formula we obtain
\begin{align*}
\|S_\cB(t)f\|_{\sfL^\infty(m)}&\lesssim e^{-\frac{1}{2}(1-\chi)t}\|f\|_{\sfL^\infty(m)}+\|f\|_{\sfL^1(m)}\\
&\lesssim\|f\|_{\sfL^1(m)\cap \sfL^\infty(m)}
\end{align*}
that combined with \eqref{dissipSB L1} gives inequality \eqref{dissipSB} for $p=\infty$.

$\#$\textit{Step 3.} We also conclude that \eqref{dissipSB} holds for $1\leq p\leq \infty$ because $\sfL^1(m)\cap \sfL^p (m)$ is an interpolation space between $\sfL^1(m)$ and $\sfL^1(m)\cap \sfL^\infty (m)$. 
\end{proof}

\begin{lem}\label{lem:SB decay}
Let $m(x,v)\asymp e^{\nu\wx^a}+e^{b\abs{v}^\gamma}$ be defined as in \Cref{Prop:Lyapunov}, $\omega=e^{\nu\wx^a+b\abs{v}^\gamma}$ as defined in \Cref{iterates Duhamel} and $R$ chosen sufficiently large. Then for $\ell>1$, sufficiently close to $1$, there exists $\lambda_\ell>0$ such that, for any $1\leq p\leq \infty$, we have the decay estimate 
\begin{equation}\label{SB decay}
\forall t \in \R^+, \qquad \|S_{\cB }(t)\|_{\sfL^1(\omega^\ell)\cap \sfL^p (\omega^\ell) \to\sfL^1 (m)\cap \sfL^p (m)}\lesssim e^{-\lambda_\ell t^a}.
\end{equation}
\end{lem}
\begin{proof}[{\bf Proof of \Cref{lem:SB decay}}]
$\#$\textit{Step 1.} The case $p=1$. We have already seen in \Cref{lem:dissipSB} that
\begin{equation*}
\cB ^* m\leq  -\epsilon\wx^{a-1}m.
\end{equation*}
Set $f=S_\cB (t)f_0$, for any $\rho>0$ we denote by $B_{\rho,\,\rho^{a/\gamma}}=\set{(x,v)\in \R^d\times\R^d\big\vert \abs{x}\leq \rho,\; \abs{v}\leq \rho^{a/\gamma}}$. We have,
\begin{align*}
\ddt\int_{\R^d\times\R^d}\abs{f}m \,\dx\dv &\leq  -\epsilon\int_{\R^d\times\R^d}\abs{f}\wx^{a-1}m \,\dx\dv \\
&\leq-\epsilon\int_{B_{\rho,\,\rho^{a/\gamma}}}\abs{f}\wx^{a-1}m \,\dx\dv \\
&\leq-\epsilon\wangle{\rho}^{a-1}\int_{B_{\rho,\,\rho^{a/\gamma}}}\abs{f}m \,\dx\dv \\
&\leq-\epsilon\wangle{\rho}^{a-1}\int_{\R^d\times\R^d}\abs{f}m \,\dx\dv +\epsilon\wangle{\rho}^{a-1}\int_{B_{\rho,\,\rho^{a/\gamma}}^c}\abs{f}m \,\dx\dv \\
&\leq-\epsilon\wangle{\rho}^{a-1}\int_{\R^d\times\R^d}\abs{f}m \,\dx\dv \\
&\quad\quad+\epsilon\wangle{\rho}^{a-1}\sup_{B_{\rho,\,\rho^{a/\gamma}}^c}\left\{\frac{m}{e^{\ell\nu\wx^a}+e^{\ell b\abs{v}^\gamma}}\right\}\int_{B_{\rho,\,\rho^{a/\gamma}}^c}\abs{f}(e^{\ell\nu\wx^a}+e^{\ell b\abs{v}^\gamma}) \,\dx\dv 
\end{align*}
with $\ell>1$. Since $m\leq \delta_2(e^{\nu\wx^a}+e^{b\abs{v}^\gamma})$, we can easily compute
\[
\frac{m}{e^{\ell\nu\wx^a}+e^{\ell b\abs{v}^\gamma}}\leq \delta_2 \frac{ e^{\nu\wx^a}+e^{b\abs{v}^\gamma}}{e^{\ell\nu\wx^a}+e^{\ell b\abs{v}^\gamma}}\leq 2^{1-1/\ell}\delta_2 \frac{\(e^{\ell\nu\wx^a}+e^{\ell b\abs{v}^\gamma}\)^{1/\ell}}{e^{\ell\nu\wx^a}+e^{\ell b\abs{v}^\gamma}}\leq 2\delta_2 \(e^{\ell\nu\wx^a}+e^{\ell b\abs{v}^\gamma}\)^{-\frac{\ell-1}{\ell}}
\]
Therefore
\[
\sup_{(B_{\rho,\,\rho^{a/\gamma}})^c}\left\{\frac{m}{e^{\ell\nu\wx^a}+e^{\ell b\abs{v}^\gamma}}\right\}\leq 
2 \delta_2 \(e^{\ell\nu\wangle{\rho}^a}+e^{\ell b\rho^a}\)^{-\frac{\ell-1}{\ell}} \leq 2\delta_2 e^{-\lambda(\ell-1) \rho^a}
\]
where $\lambda=\min\{\nu,b\}>0$. Moreover, if $\ell>1$ is sufficiently close to 1, we can use \eqref{dissipSB} so that 
\[
\int_{B_{\rho,\,\rho^{a/\gamma}}^c}\abs{f}(e^{\ell\nu\wx^a}+e^{\ell b\abs{v}^\gamma}) \,\dx\dv \leq \int_{\R^d\times\R^d}\abs{f_0}(e^{\ell\nu\wx^a}+e^{\ell b\abs{v}^\gamma}) \,\dx\dv \leq 2\int_{\R^d\times\R^d}\abs{f_0}\omega^\ell \,\dx\dv  
\]
Substituting we find
\begin{align*}
\ddt\int_{\R^d\times\R^d}\abs{f}m \,\dx\dv &\leq-\epsilon\wangle{\rho}^{a-1}\int_{\R^d\times\R^d}\abs{f}m \,\dx\dv +4\delta_2\epsilon\wangle{\rho}^{a-1} e^{-\lambda(\ell-1) \rho^a}\norm{f_0}_{\sfL^1(\omega^\ell)}
\end{align*}
By Gronwall's Lemma, we find that for all $t\geq 0$,
\begin{align*}
\iint_{\R^d\times \R^d} \abs{f}m\,\dx\dv &\leq 4\delta_2 e^{-\lambda(\ell-1) \rho^a}\norm{f_0}_{\sfL^1(\omega^\ell)}+ e^{-\epsilon\wangle{\rho}^{a-1}t}\norm{f_0}_{\sfL^1 (m)}\\
&\lesssim (e^{-\lambda(\ell-1) \rho^a}+e^{-\epsilon\rho^{a-1}t})\norm{f_0}_{\sfL(\omega^\ell)}.
\end{align*}
Taking $\rho=\frac{\epsilon\,t}{\lambda(\ell-1)} $ we have
\begin{equation}\label{decaySB_L1}
\norm{f}_{\sfL^1(m)}\lesssim e^{-\lambda_\ell t^a}\norm{f_0}_{\sfL^1(\omega^\ell)},
\end{equation}
with $\lambda_\ell=\lambda^{1-a}\epsilon^a(\ell-1)^{1-a}$, that is \eqref{SB decay} for $p=1$.

$\#$\textit{Step 2.} The case $p=\infty$. We start again as in the proof of \textit{Step 2} of \Cref{lem:dissipSB}, in particular we recall the Duhamel formula
\begin{equation}\label{DuhamelFormula2}
S_\cB =S_{\cB_0}+ \sum_{j=1}^{d+1} S_{\cB_0}\star(\cA_0 S_{\cB_0})^{\star j} + S_{\cB_0}\star (\cA_0 S_{\cB_0})^{\star(d+1)} \star \cA_0 S_\cB.
\end{equation}
Since $m\lesssim  \omega^\ell$, we also have $\norm{\mathsf{Id}}_{\sfL^\infty(\omega^\ell)\to\sfL^\infty(m)}\lesssim 1$, therefore by \eqref{est1} we deduce
\begin{align*}
\|S_{\cB_0}\star(\cA_0 S_{\cB_0})^{\star j}\|_{\sfL^\infty(\omega^\ell)\to \sfL^\infty(m)}\leq\|S_{\cB_0}\star(\cA_0 S_{\cB_0})^{\star j}\|_{\sfL^\infty(m)\to \sfL^\infty(m)}\norm{\mathsf{Id}}_{\sfL^\infty(\omega^\ell)\to\sfL^\infty(m)}\lesssim e^{-\frac{1}{2}(1-\chi)t}.
\end{align*}
Writing this estimate for all $j\in\Drange{0,\,d+1}$ we find
\begin{equation}\label{DuhamelDecay1}
\| S_{\cB_0}\|_{\sfL^\infty(\omega^\ell)\to \sfL^\infty(m)}+ \sum_{j=1}^{d+1} \|S_{\cB_0}\star(\cA_0 S_{\cB_0})^{\star j}\|_{\sfL^\infty(\omega^\ell)\to \sfL^\infty(m)}\lesssim e^{-\frac{1}{2}(1-\chi)t},
\end{equation}
which allows us to estimate the first $d+1$ terms of the Duhamel formula \eqref{DuhamelFormula2}. We now have to improve the estimate on the last term. From \eqref{hereDissipativity} we have 
\begin{align*}
\|\cA_0 S_{\cB}(t)f\|_{\sfL^1_x \sfL^\infty_v(\omega)}&\lesssim\iint_{\R^d \times \R^d} \abs{(S_\cB (t) f)(x,v)}m(x,v')\,\dx\dv ' \\
&\lesssim e^{-\lambda_\ell t^a} \norm{f_0}_{\sfL^1(\omega^\ell)},
\end{align*}
where in the last inequality we used \eqref{decaySB_L1}. Therefore now we can concatenate 
\begin{align*}
\|\cA_0 S_\cB(t)\|_{\sfL^1(\omega^\ell)\to\sfL^1_x\sfL^\infty_v(\omega)}\lesssim e^{-\lambda_\ell t^a}
\end{align*}
with
\[
\|(\cA_0 S_{\cB_0})^{\star(d+1)}\|_{\sfL^1_x\sfL^\infty_v(\omega)\to \sfL^\infty(\omega)}\lesssim e^{-(1-\chi-\gamma)t}
\]
and
\[
\|S_{\cB_0}(t)\|_{\sfL^\infty(\omega)\to \sfL^\infty(m)}\lesssim e^{-\frac{1}{2}(1-\chi)t}.
\]
By convolving two exponential decays and a sub-exponential decay we obtain again the sub-exponential decay (see Appendix), therefore 
\begin{equation}\label{DuhamelDecay2}
\|S_{\cB_0}\star (\cA_0 S_{\cB_0})^{\star(d+1)} \star \cA_0 S_\cB\|_{\sfL^1(\omega^\ell)\to \sfL^\infty(m)}\lesssim e^{-\lambda_\ell t^a}.
\end{equation}
Putting \eqref{DuhamelDecay2} together with \eqref{DuhamelDecay1} we finally have
\begin{equation}
\norm{S_\cB}_{\sfL^\infty(\omega^\ell)\cap\sfL^1(\omega^\ell)\to \sfL^\infty(m)}\lesssim e^{-\lambda_\ell t^a},
\end{equation}
that combined with \eqref{decaySB_L1} gives \eqref{SB decay} for $p=\infty$.

$\#$\textit{Step 3.} As in \Cref{lem:dissipSB}, we use the fact that $\sfL^1(m)\cap \sfL^p (m)$ is an interpolation space between $\sfL^1(m)$ and $\sfL^1(m)\cap \sfL^\infty (m)$ and an analogous statements holds with weight $\omega^\ell$. Thus conclude that \eqref{SB decay} also holds for all $1\leq p\leq \infty$. 
\end{proof}

\begin{lem}\label{lem:SB Polydecay}
Let $m\asymp\wangle{x}^k +\wangle{v}^{2k}$ be defined as in \Cref{Prop:PolyLyap}. Then for all $\ell\geq 1$ we have the decay estimate 
\begin{equation}
\forall t \in \R^+, \qquad \|S_{\cB }(t)\|_{\sfL^1(m^\ell) \to\sfL^1 (m)}\lesssim \frac1{\wangle{t}^{k(\ell-1)}}.
\end{equation}
\end{lem}
\begin{proof}[{\bf Proof of \Cref{lem:SB Polydecay}}]
Let $m$ be as in \eqref{LyapunovFunctionPoly}. Then, using \eqref{eq:2101} and \Cref{lem:Psi}, we have
\begin{align*}
    \cB^* m&=\cL^*m - \eta_R(x)\Lambda\(\frac{x\cdot v}{\wx}\)\int_\Rd m(x,v')\M(v')\,\dv '\\
    &=\cL^*m -\eta_R(x)\Lambda\(\frac{x\cdot v}{\wx}\)\( \wx^{k} -\frac{k\chi}{1+\chi}\wx^{k-1}\int_\Rd \Psi\(\frac{x\cdot v'}{\wx}\)\M(v')\dv ' +Bc_{2k,\gamma}\)\\
    &\leq C_2\abs{v}^2\wx^{k-2} - \frac{k\chi(1-\chi)\zeta}{1+\chi}(1-\eta_R(x))\wx^{k-1} + B(1+\chi)c_{2k,\gamma}(1-\eta_R(x))\\
    &\qquad - B(1-\chi)\wangle{v}^{2k} - (1-\chi)\eta_R(x)\wx^k
\end{align*}
If $k\in(1,2]$, we can estimate $\abs{v}^2\wx^{k-2}\leq\wangle{v}^{2k}$ and then take $B\geq \tfrac{2C_2}{1-\chi}$. If $k>2$, we can proceed as in \eqref{eq:YoungP} by Young’s inequality and we have that for any $\delta>0$
\begin{align*}
    \cB^* m &\leq \(\frac{C_2(k-2)}{k-1}\delta - \frac{k\chi(1-\chi)\zeta}{1+\chi}\)(1-\eta_R(x))\wx^{k-1} + \(\frac{C_2(k-2)}{k-1}\delta\wx^{k-1} - (1-\chi)\wx^k\)\eta_R(x)\\
    &\qquad + \(\frac{C_2}{(k-1)\delta^{\tfrac{1}{k-2}}} - B(1-\chi)\)\wangle{v}^{2k}  + B(1+\chi)c_{\gamma,k}(1-\eta_R(x)).
\end{align*}
We can first choose
\[
\delta<\frac{k-1}{2C_2(k-2)}\min\left\{\frac{k\chi(1-\chi)\zeta}{1+\chi}, 1-\chi\right\},
\]
and then enlarge
\[
B> \frac{2C_2}{(1-\chi)(k-1)\delta^{\tfrac{1}{k-2}}}
\]
to get
\begin{align*}
    \cB^*m&\leq -\frac{k\chi(1-\chi)\zeta}{2(1+\chi)}(1-\eta_R(x))\wx^{k-1} -\frac{B(1-\chi)}{2}\wangle{v}^{2k} - \frac{1-\chi}{2}\eta_R(x) +B(1+\chi)c_{\gamma,k}(1-\eta_R(x)).
\end{align*}
Finally, taking $R>0$ large enough such that $\tfrac{k\chi(1-\chi)\zeta}{4(1+\chi)}\wangle{R}^{k-1}\geq B(1+\chi)$, we can conclude
\begin{align*}
    \cB^*m\lesssim -(\wx^{k-1}+\wangle{v}^{2k})\lesssim - m^{1-\tfrac{1}{k}}.
\end{align*}

This means that, for all $k>1$, the semigroup $S_\cB$ is dissipative in $\sfL^1(m)$, i.e.
\[
\forall t \in \R^+, \qquad \|S_{\cB }(t)\|_{\sfL^1(m) \to\sfL^1 (m)}\leq 1.
\]
Moreover, let $g=S_\cB(t)g_0$ and, for any $\ell>0$, we have
\begin{align*}
    \ddt \norm{g}_{\sfL^1(m)}&\leq \int_\RRd \abs{g}\cB^*m\,\dx\dv \\
    &\lesssim-\int_\RRd \abs{g} m^{1-\tfrac{1}{k}}\,\dx\dv .
\end{align*}
By the H\"older inequality and dissipativity, we have
\[
\norm{g}_{\sfL^2(m)}  \lesssim \(\int_\RRd \abs{g} m^{1-\tfrac{1}{k}} \,\dx\dv \)^{\frac{k(\ell-1)}{k(\ell-1)+1}}\(\int_\RRd \abs{g_0} m^\ell\,\dx\dv \)^{\frac{1}{k(\ell-1)+1}},
\]
which gives
\begin{align*}
    \ddt \norm{g}_{\sfL^1(m)}\lesssim \norm{g}_{\sfL^1(m)}^{1+\frac{1}{k(\ell-1)}} \norm{g_0}_{\sfL^1(m^\ell)}^{-\frac{1}{k(\ell-1)}}.
\end{align*}
The conclusion follows by Gronwall's Lemma.
\end{proof}

\subsection{The iterated semi-groups $\left(\cA S_B\right)^{\star j}$ and $S_\cB \star \left(\cA S_B\right)^{\star j}$}\label{SubSec:Duhamel2}

In this section we provide some useful estimates on $\cA$ and $S_{\cB}$. The main difference between $S_\cB$ and $S_{\cB_0}$ is that the former does not have an explicit form, so we cannot carry out the estimates through explicit computations as in the previous sub-section. For this reason, the spirit of this section is to systematically use the Duhamel formula for the splitting $\cB=\cA_0+\cB_0$.

\begin{lem}\label{lem:A}
Let $m$ be the Lyapunov function \eqref{LyapunovFunctionExp} and $\omega$ as defined in \Cref{iterates Duhamel}. Then for $\ell\geq 1$ sufficiently close to 1 we have
\begin{align}\label{estm: A}
\norm{\cA}_{\sfL^1(m)\cap\sfL^\infty(m)\to\sfL^1(\omega^\ell)\cap\sfL^\infty(\omega^\ell)}&\lesssim 1, & \norm{\cA}_{\sfL^1\to\sfL^1_x\sfL^\infty_v(\omega^\ell)}&\lesssim 1 .
\end{align}
\end{lem}
\begin{proof}[{\bf Proof of \Cref{lem:A}}]
We recall that
\[
(\cA f)\omega^\ell=\eta_R(x)\M(v)e^{\ell\nu\wx^a+\ell b\abs{v}^\gamma}\int_{\R^d}\Lambda\(\frac{x\cdot v'}{\wx}\)f(x,v')\dv '.
\]
Then we have
\begin{align*}
\norm{\cA f}_{\sfL^1(\omega^\ell)}&\lesssim \sup_{x\in\R^d}\{\eta_R(x)e^{\ell\nu\wx^a}\}\(\int_{\R^d}e^{\ell b\abs{v}^\gamma}\M(v)\dv \) \iint_{\R^d\times\R^d}\abs{f(x,v')}\dx\dv '\\
&\lesssim \norm{f}_{\sfL^1}\lesssim \norm{f}_{\sfL^1(m)}
\end{align*}
and
\begin{align*}
\norm{\cA f}_{\sfL^\infty(\omega^\ell)}&\lesssim \sup_{v\in\R^d}\left\{e^{\ell b\abs{v}^\gamma}\M(v)\right\}\sup_{x\in\R^d}\left\{ \eta_R(x)e^{\ell\nu\wx^a}\int_{\R^d}\abs{f(x,v')}\dv '\right\}\\
&\lesssim \sup_{x\in\R^d}\left\{ \eta_R(x)e^{\ell\nu\wx^a}\(\int_{\R^d}m(x,v')^{-1}\dv '\)\sup_{v'\in\R^d}\{f(x,v')m(x,v')\}\right\}\\
&\lesssim \sup_{x\in\R^d}\left\{ \eta_R(x)e^{\ell\nu\wx^a}\(\int_{\R^d}m(x,v')^{-1}\dv '\)\right\}\norm{f}_{\sfL^\infty(m)}\lesssim \norm{f}_{\sfL^\infty(m)}.
\end{align*}
Hence the first inequality of \eqref{estm: A} is proved by noting that
\[
\norm{\cA f}_{\sfL^1(\omega^\ell)\cap\sfL^\infty(\omega^\ell)}=\norm{\cA f}_{\sfL^1(\omega^\ell)}+\norm{\cA f}_{\sfL^\infty(\omega^\ell)}\lesssim \norm{f}_{\sfL^1(m)}+\norm{f}_{\sfL^\infty(m)}=\norm{f}_{\sfL^1(m)\cap\sfL^\infty(m)}.
\]
For the last inequality of \eqref{estm: A} we jut have to compute
\begin{align*}
\norm{\cA f}_{\sfL^1_x\sfL^\infty_v}&\lesssim \norm{e^{\ell b\abs{v}^\gamma}\M(v)}_{\sfL^\infty}\iint_{\R^d\times\R^d}\eta_R(x)e^{\ell\nu\wx^a}\abs{f(x,v')}\dx\dv '\\
&\lesssim \norm{\eta_R(x)e^{\ell\nu\wx^a}}_{\sfL^\infty}\iint_{\R^d\times\R^d}\abs{f(x,v')}\dx\dv '\\
&\lesssim \norm{f}_{\sfL^1}.
\end{align*}
\end{proof}

\begin{lem}\label{lem:ASB decay}
Let $m(x,v)\asymp e^{\nu\wx^a}+e^{b\abs{v}^\gamma}$ be defined as in \Cref{Prop:Lyapunov}, $\omega=e^{\nu\wx^a+b\abs{v}^\gamma}$ as defined in \Cref{iterates Duhamel}. Then for $\ell > 1$, sufficiently close to $1$, we have
\begin{align}\label{AS_B 1bis}
\norm{\cA S_{\cB}(t)}_{\sfL^1(\omega^\ell)\cap\sfL^\infty(\omega^\ell) \to\sfL^1(\omega^\ell)\cap\sfL^\infty(\omega^\ell)}&\lesssim e^{-\lambda_\ell t^a},
\end{align}
\begin{align}\label{AS_B 2bis}
\norm{\cA S_{\cB}(t)}_{\sfL^1_x\sfL^\infty_v(\omega^\ell)\to\sfL^1(\omega^\ell)\cap\sfL^\infty(\omega^\ell)}&\lesssim t^{-d}e^{-\lambda_\ell t^a},
\end{align}
and
\begin{align}\label{AS_B 3bis}
\norm{\cA S_{\cB}(t)}_{\sfL^1_x\sfL^\infty_v(\omega^\ell)\to \sfL^1_x\sfL^\infty_v(\omega^\ell)}&\lesssim e^{-\lambda_\ell t^a}.
\end{align}
where $\lambda_\ell$ has been found in \Cref{lem:SB decay}. As a consequence, for every $n\in\NN$,
\begin{equation}\label{AS_B 4bis}
\|(\cA S_{\cB}(t))^{\star n}\|_{\sfL^1_x\sfL^\infty_v(\omega^\ell)\longrightarrow \sfL^1(\omega^\ell)\cap\sfL^\infty(\omega^\ell)}\lesssim t^{-(d-n+1)}e^{-\lambda_\ell t^a}.
\end{equation}
\end{lem}
\begin{proof}[{\bf Proof of \Cref{lem:ASB decay}}]
The idea is to use a variant of the Duhamel formula
\begin{equation}\label{A Duhamel}
\cA S_\cB =\cA S_{\cB_0}+ \sum_{j=1}^{n} \cA S_{\cB_0}\star(\cA_0 S_{\cB_0})^{\star j} + \cA S_{\cB}\star (\cA_0 S_{\cB_0})^{\star(n+1)}.
\end{equation}
We consider $\ell>1$ small enough so that the estimates in \Cref{iterates Duhamel} can still be used. 

$\#$\textit{Step 1.} We prove \eqref{AS_B 1bis}. We have 
\begin{align*}
\norm{\cA S_\cB}_{\sfL^1(\omega^\ell)\cap\sfL^\infty(\omega^\ell)\to \sfL^1(\omega^\ell)\cap\sfL^\infty(\omega^\ell)}&\leq \norm{\cA}_{\sfL^1(m)\cap\sfL^\infty(m)\to \sfL^1(\omega^\ell)\cap\sfL^\infty(\omega^\ell)}\norm{ S_\cB}_{\sfL^1(\omega^\ell)\cap\sfL^\infty(\omega^\ell)\to \sfL^1(m)\cap\sfL^\infty(m)} \\
&\lesssim e^{-\lambda_\ell t^a},
\end{align*}
thanks to \Cref{lem:A} and \Cref{lem:SB decay}.

$\#$\textit{Step 2.} Proof of the \eqref{AS_B 2bis}. We consider \eqref{A Duhamel} with $n=d$. The first terms can be easily estimated by \eqref{AS_B 6} and \eqref{AS_B 7}, indeed we have
\begin{equation}
\norm{\cA S_{\cB_0}\star (\cA_0 S_{\cB_0})^{\star j}}_{\sfL^1_x\sfL^\infty_v(\omega^\ell)\to \sfL^1(\omega^\ell)\cap\sfL^\infty(\omega^\ell)}\lesssim t^{-(d-j)}e^{-(1-\chi-\lambda)t},
\end{equation}
which, for $j\in\Drange{0,\,d}$, yields
\begin{equation}\label{1604}
\norm{\cA S_{\cB_0}}_{\sfL^1_x\sfL^\infty_v(\omega^\ell)\to \sfL^1(\omega^\ell)\cap\sfL^\infty(\omega^\ell)}+\sum_{j=1}^d \norm{\cA S_{\cB_0}\star (\cA_0 S_{\cB_0})^{\star j}}_{\sfL^1_x\sfL^\infty_v(\omega^\ell)\to \sfL^1(\omega^\ell)\cap\sfL^\infty(\omega^\ell)}\lesssim t^{-d}e^{-(1-\chi-\lambda)t}
\end{equation}
Using \eqref{AS_B 1bis} just demonstrated, \eqref{AS_B 6} and \eqref{AS_B 7} we find
\begin{align*}
\norm{\cA S_{\cB}\star &(\cA_0 S_{\cB_0})^{\star d+1}}_{\sfL^1_x\sfL^\infty_v(\omega^\ell)\to \sfL^1(\omega^\ell)\cap\sfL^\infty(\omega^\ell)}\\
&\lesssim \norm{\cA S_\cB}_{\sfL^1(\omega^\ell)\cap\sfL^\infty(\omega^\ell)\to \sfL^1(\omega^\ell)\cap\sfL^\infty(\omega^\ell)}\star \norm{(\cA_0 S_{\cB_0})^{\star d+1}}_{\sfL^1_x\sfL^\infty_v(\omega^\ell)\to \sfL^1(\omega^\ell)\cap\sfL^\infty(\omega^\ell)}\\
&\lesssim e^{-\lambda_\ell t^a}\star e^{-(1-\chi-\lambda)t}\lesssim e^{-\lambda_\ell t^a}
\end{align*}
Substituting this and \eqref{1604} in \eqref{A Duhamel} we prove \eqref{AS_B 2bis}.

$\#$\textit{Step 3.} Proof of the \eqref{AS_B 3bis}. We consider the Duhamel formula \eqref{A Duhamel} for $n=1$. Then we have
\begin{align*}
\norm{\cA S_\cB&}_{\sfL^1_x\sfL^\infty_v(\omega^\ell)\to \sfL^1_x\sfL^\infty_v(\omega^\ell)}\\
&\leq \norm{\cA S_{\cB_0}}_{\sfL^1_x\sfL^\infty_v(\omega^\ell)\to \sfL^1_x\sfL^\infty_v(\omega^\ell)}+\norm{\cA S_\cB}_{\sfL^1(\omega^\ell)\to \sfL^1_x\sfL^\infty_v(\omega^\ell)}\star\norm{\cA_0 S_{\cB_0}}_{\sfL^1_x\sfL^\infty_v(\omega^\ell)\to \sfL^1(\omega^\ell)} 
\end{align*}
We estimate the first term by \eqref{AS_B 2}, the last term by \eqref{AS_B 5} and 
\begin{align*}
\norm{\cA S_\cB}_{\sfL^1(\omega^\ell)\to \sfL^1_x\sfL^\infty_v(\omega^\ell)}&\leq \norm{\cA}_{\sfL^1(m)\to \sfL^1_x\sfL^\infty_v(\omega^\ell)}\norm{S_\cB}_{\sfL^1(\omega^\ell)\to\sfL^1(m)}\\
&\leq \norm{\cA}_{\sfL^1\to \sfL^1_x\sfL^\infty_v(\omega^\ell)}\norm{S_\cB}_{\sfL^1(\omega^\ell)\to\sfL^1(m)}\leq e^{-\lambda_\ell t^a}.
\end{align*}
We conclude that
\begin{align*}
\norm{\cA S_{\cB}(t)}_{\sfL^1_x\sfL^\infty_v(\omega^\ell)\to \sfL^1_x\sfL^\infty_v(\omega^\ell)}\lesssim e^{-(1-\chi-\lambda)t}+e^{-\lambda_\ell t^a}\star e^{-(1-\chi-\lambda)t}\lesssim e^{-\lambda_\ell t^a}.
\end{align*}

$\#$\textit{Step 4.} Finally, inequality  \eqref{AS_B 4bis} is a direct consequence of \eqref{AS_B 1bis}, \eqref{AS_B 2bis} and \eqref{AS_B 3bis} thanks to \cite[Proposition 2.5]{MQT} (see also \cite[Lemma 6.3]{Cao19}).
\end{proof}

%\subsection{A uniform bound on $S_\cL$ and the proof of \Cref{thm:main}}

\begin{proof}[{\bf Proof of \Cref{Prop:SL_bounded}}]
Consider the splitting $\cL=\cA+\cB$, then we have the Duhamel formula
\[
S_\cL =S_{\cB}+ \sum_{j=1}^{d+1} S_{\cB}\star(\cA S_{\cB})^{\star j} + S_{\cB}\star (\cA S_{\cB})^{\star(d+1)} \star \cA S_\cL\,,
\]
and we are interested in the norm $\norm{S_\cL}_{X\to X}$, where $X:=\sfL^1(m)\cap\sfL^\infty(m)$ and $m(x,v)\asymp e^{\nu\wx^a}+e^{b\abs{v}^\gamma}$ is defined as in \Cref{Prop:Lyapunov}. We can bound the first terms as follows
\begin{align*}
\norm{S_\cB\star(\cA S_\cB)^{\star j}}_{X\to X}&=\norm{(S_B \cA)^{\star j}\star S_\cB}_{X\to X}\\
&\leq \(\norm{ S_\cB \cA}_{X\to X}\)^{\star j}\star\norm{S_\cB}_{X\to X}\\
&\leq \(\norm{S_\cB}_{\sfL^1(\omega^\ell)\cap \sfL^\infty(\omega^\ell)\to X}\norm{\cA}_{X\to \sfL^1(\omega^\ell)\cap \sfL^\infty(\omega^\ell)}\)^{\star j}\star \norm{S_\cB}_{X\to X}\\
&\lesssim \(e^{-\lambda_\ell t^a}\)^{\star j}\star 1\\
&\lesssim 1,
\end{align*}
thanks to \Cref{lem:SB decay}, \Cref{lem:A} and \Cref{lem:dissipSB}. Taking $j\in\Drange{0,\,d+1}$ we deduce that
\begin{equation}\label{SL first duhamel}
\norm{S_{\cB}}_{X\to X}+ \sum_{j=1}^{d+1} \norm{S_{\cB}\star(\cA S_{\cB})^{\star j}}_{X\to X}\lesssim 1.
\end{equation}
Concerning the last term of Duhamel's formula, we just have to use the mass conservation of $S_\cL$ as follows
\begin{align*}
\norm{\cA S_\cL(t)f_0}_{\sfL^1_x\sfL^\infty_v(\omega^\ell)}\lesssim \norm{S_\cL(t)f_0}_{\sfL^1}=\norm{f_0}_{\sfL^1}\lesssim \norm{f_0}_{X}.
\end{align*}
In this way we can concatenate 
\[
\norm{\cA S_\cL}_{X\to\sfL^1_x\sfL^\infty_v(\omega^\ell)}\lesssim 1
\]
with 
\begin{align*}
\|(\cA S_{\cB}(t))^{\star d+1}\|_{\sfL^1_x\sfL^\infty_v(\omega^\ell)\longrightarrow \sfL^1(\omega^\ell)\cap\sfL^\infty(\omega^\ell)}&\lesssim e^{-\lambda_\ell t^a}, & \norm{S_\cB}_{\sfL^1(\omega^\ell)\cap \sfL^\infty(\omega^\ell)\to X}&\lesssim e^{-\lambda_\ell t^a},
\end{align*}
proven respectively in \Cref{lem:ASB decay} and in \Cref{lem:SB decay}. We then obtain
\begin{equation}
\norm{S_{\cB}\star (\cA S_{\cB})^{\star(d+1)} \star \cA S_\cL}_{X\to X}\lesssim 1,
\end{equation}
that with \eqref{SL first duhamel} gives \eqref{SL bounded}. 

The proof of the existence of a steady state follows the same steps as in \cite{MW17}. We introduce the following definition
\[
\forall f\in X,\qquad \inorm{f}:=\sup_{t\geq 0}\;\norm{S_{\cL}(t)f}_X.
\]
Clearly $\inorm{\cdot}$ is an equivalent norm to $\norm{\cdot}_X$, indeed
\[
\norm{f}_X=\norm{S_{\cL}(0)f}_X\leq \sup_{t\geq 0}\;\norm{S_{\cL}(t)f}_X=\inorm{f}
\]
and
\[
\inorm{f}=\sup_{t\geq 0}\;\norm{S_{\cL}(t)f}_X\lesssim \sup_{t\geq 0}\;\norm{f}_X=\norm{f}_X.
\]
For any rotation $\Omega\in SO(d)$ we set $f_\Omega(x,v):=f(\Omega x,\Omega v)$, in this way we say that $f$ is rotationally symmetric if $f_\Omega=f$ for every $\Omega\in SO(d)$. Consider the set
\[
\mathcal{C}:=\left\{f\in X\Big| f\geq 0,\; \int_{\R^d\times\R^d}f(x,v)\dx\dv =1,\; \inorm{f}\leq C,\; f_\Omega=f, \;  \forall\, \Omega\in SO(d)\right\}
\]
with $C>0$ large enough so that $\mathcal{C}$ is not empty. In addition, the set $\mathcal{C}$ is convex, closed and bounded for the weak-$*$ topology of $\sfL^\infty(m)$, as a consequence it is compact for the weak-$*$ topology of $\sfL^\infty(m)$ by the Banach-Alaoglu Theorem. We also remind that $S_\cL$ is weakly-$*$ continuous (this comes from the well-posedness Theorem).

Let us show that $\mathcal{C}$ is invariant for $S_\cL$. Clearly $S_\cL$ preserves positivity and total mass, and for any $t\geq 0$,
\[
\inorm{S_\cL(t)f}=\sup_{s\geq 0}\, \norm{S_\cL (t+s)f}_X\leq \sup_{t\geq 0}\norm{S_\cL(t)f}_X=\inorm{f}
\]
Let $f_\Omega(t,x,v)=f(t,\Omega x,\Omega v)$ with $\Omega\in SO(d)$, then er have $\partial_t f_\Omega=(\partial_t f)_\Omega$, 
\[
v\cdot \nabla_x f_\Omega=v\cdot \Omega(\nabla_xf)_\Omega=(v\cdot\nabla_x f)_\Omega
\]
and
\begin{align*}
\int_{\R^d}\Lambda \(\frac{x\cdot v'}{\wx}\)f_\Omega(x,v')\dv '\M(v)&=\int_{\R^d}\Lambda\(\frac{x\cdot v'}{\wx}\)f(\Omega x,\Omega v')\dv '\M(v)\\
&=\int_{\R^d}\Lambda\(\frac{x\cdot \Omega^{-1}v'}{\wx}\) f(\Omega x,v')\dv '\M(v)\\
&=\int_{\R^d}\Lambda \(\frac{\Omega x\cdot v'}{\wangle{\Omega x}}\)f(\Omega x,v')\dv '\M(\Omega v)\\
&=\(\int_{\R^d}\Lambda (x,v')f(x,v')\dv '\M(v)\)_\Omega
\end{align*}
because $\M$ is rotationally symmetric. Therefore if $f(t,x,v)$ is a solution with initial datum $f(0,x,v)$, then $f_\Omega(t,x,v)$ is a solution with initial datum $f_\Omega(0,x,v)$, in other words
\[
S_\cL(t)f_\Omega=(S_\cL(t) f)_\Omega.
\]
Thus we conclude that $\mathcal{C}$ is an invariant set for $S_\cL$. As a consequence, the Brouwer-Schauder-Tychonoff fixed point theorem (see for example \cite{MW17,C21}), there exists at least one non-negative, invariant by rotation and normalized stationary state $G\in X $ to \eqref{eq:main}.
\end{proof}

%%%%%%%%%%%%%%%%%%%%%%%%%%%%%%%%%%%%%%%%%%%%%%%%%%%%%%%%%%%%%%%%%%%%%%%%%%%%%%%%%%%%%%%%%%%%%%%%%%%%%%%%%%%%

\section{Representation formula and pre-bounds above and below for the density}\label{Sec:ReprFormula}

In this section we prove an upper and a lower bound on the density $\rho_G$ of the steady state $G$ of the run and tumble equation, whose existence has been shown in the previous paragraphs. We start with a representation formula.

\begin{lem}[Representation formula]\label{lem:representation}
The stationary solution to \eqref{eq:main}, $G$ satisfies the following equation 
\begin{equation}\label{eq:representation}
\forall x,v\in\R^d \qquad G(x,v)=\int_0^{+\infty}e^{-\int_0^s\Lambda(x-vr,v)dr}\M(v)\theta_G(x-vs)\d s
\end{equation}
where 
\begin{equation}
    \theta_G(x):=\int_{\R^d}\Lambda(x,v')G(x,v')\dv '.
\end{equation}
\end{lem}
\begin{proof}[{\bf Proof of \Cref{lem:representation}}]
Let us remind the splitting $\cL=\cB_0+\cA_1$ where 
\[
(\cB_0 f)(x,v)  = - v \cdot \nabla_x f(x,v) - \Lambda(x,v) f(x,v), 
\]
and 
\[
(\cA_1 f)(x,v)=\M(v)\int_{\R^d}\Lambda(x,v')f(x,v')\dv 
\]
By applying the Duhamel formula to $G$ ad remembering that $S_\cL(t)G=G$, we have
\begin{align*}
G-\int_0^t S_{\cB_0}(s)\cA_1 G \d s &=S_{\cL}(t)G-\int_0^t S_{\cB_0}(s)\cA_1 S_{\cL}(t-s)G \d s\\
&=S_{\cB_0}(t)G
\end{align*}
Thanks to \Cref{estimate:S_B0} and \Cref{thm:main} we have that 
\begin{equation*}
\norm*{G-\int_0^t S_{\cB_0}(s)\cA_1 G \d s}_{\sfL^\infty(m)}\leq e^{-\frac{1}{2}(1-\chi)t}\norm{G}_{\sfL^\infty(m)}\longrightarrow 0 
\end{equation*}
as $t\to 0$, therefore we have 
\[
\forall x,v\in\R^d \qquad G(x,v)= \int_0^\infty S_{\cB_0}(s)\cA_1 G(x,v) \d s
\]
This equation is nothing other than \eqref{eq:representation}.
\end{proof}

\begin{remark}
It is worth emphasizing that some more explicit computations can be made in the one dimensional case. Indeed we have
\begin{equation}\label{eq:representation1D}
G(x,v)=
\begin{cases}
\int_{u=-\infty}^{x}\frac{1}{\abs{v}}\M(v) e^{-\frac{1}{\abs{v}}\int_u^x\Lambda(\tau,v)d\tau}\theta_G(u)\d u & \text{if } v\geq 0\\
\int_{u=0}^{\infty}\frac{1}{\abs{v}}\M(v) e^{-\frac{1}{\abs{v}}\int_x^{x+u}\Lambda(\tau,v)d\tau}\theta_G(x+u)\d u & \text{if } v< 0.
\end{cases}
\end{equation}
This expression is obtained from \eqref{eq:representation} by changing varibles
\begin{align*}
    &\begin{cases}
    u=x-vs\\
    s=\frac{x-u}{v}
    \end{cases}
    &
    &\begin{cases}
    \tau=x-vr\\
    r= \frac{x-\tau}{v}
    \end{cases}
\end{align*}
when $v\leq 0$ and by
\begin{align*}
    &\begin{cases}
    u=\abs{v} s \\
    s=\frac{u}{\abs{v}}
    \end{cases}
    &
    &\begin{cases}
    \tau=x+\abs{v}r\\
    r= \frac{\tau-x}{\abs{v}}
    \end{cases}
\end{align*}
when $v\geq 0$. The representation \eqref{eq:representation1D} will turn out to be useful when $\Lambda(x,v) =1+\chi\sign(xv)$ and we will study the behavior of $\theta_G(x)$ as $x\to \infty$.
\end{remark}

Before moving on to the upper and lower bound on $\rho_G$, we introduce an important and recurrent function that we will often use in the following, $\Phi\colon\Rd\to\R$ defined by 
\begin{equation}\label{eq:PHI}
    \Phi(y):= \frac{1}{\abs{y}^{d-1}}\int_0^\infty u^{d-2} e^{-\frac{\abs{y}}{u}-\frac{u^\gamma}{\gamma}}\d u.
\end{equation}
This definition will allows us to simplify the notation for the results of this section. Thanks to \Cref{lem:phi2} in the Appendix, we deduce that the asymptotic approximations of $\Phi(y)$ are
\begin{equation}\label{eq:d>2 asymptINF}
\Phi(y)\sim \sqrt{\frac{2\pi}{1+\gamma}} \, \abs{y}^{-\frac{\gamma}{\gamma+1}\(d-\frac{1}{2}\)}e^{-\frac{1+\gamma}{\gamma}\abs{y}^{\frac{\gamma}{1+\gamma}}}\qquad \text{as } \abs{y}\to \infty,
\end{equation}
and
\begin{equation}\label{eq:d>2 asympt0}
\Phi(y)\sim
\begin{cases}
    \gamma^{\frac{d-1}{\gamma}-1}\Gamma\(\frac{d-1}{\gamma}\)\frac{1}{\abs{y}^{d-1}} & \text{if } d\geq 2\\
    \abs{\ln\abs{y}} & \text{if } d=1
\end{cases} \qquad \text{as } \abs{y}\to 0.
\end{equation}

\begin{lem}\label{lem:conv_inequalities}
Let $G$ be the stationary solution to \eqref{eq:main}, then the density $\rho_G$ satisfies the following inequalities
\begin{equation}
\rho_G\star \phi_1\leq \rho_G\leq \rho_G\star \phi _2
\end{equation}
where
\begin{equation}\label{eq:phi_1}
\phi_1(y)=(1-\chi)c_{0,\gamma}^{-1}\Phi((1+\chi) y), 
\end{equation}
\begin{equation}\label{eq:phi_2}
\phi_2(y)=(1+\chi)c_{0,\gamma}^{-1} \Phi((1-\chi)y)
\end{equation}
and $\Phi$ is given in \eqref{eq:PHI}.
\end{lem}

\begin{proof}[{\bf Proof of \Cref{lem:conv_inequalities}}]
By the representation formula we have
\begin{align*}
G(x,v)&=\int_0^{+\infty}e^{-\int_0^s\Lambda(x-vr,v)dr}\M(v)\theta_G(x-vs)\d s\\
&\geq (1-\chi) \int_0^{+\infty}e^{-(1+\chi)s}\M(v)\rho_G(x-vs)\d s
\end{align*}
Integrating in $v$ and changing the variables $z=x-sv$, $v=\frac{x-z}{s}$ we have
\begin{align*}
\rho_G(x)&\geq (1-\chi)\int_{\R^d} \int_0^{+\infty}e^{-(1+\chi)s}\M(v)\rho_G(x-vs)\d s\dv \\
&\geq(1-\chi)\int_{\R^d} \int_0^{+\infty}\frac{1}{s^d}e^{-(1+\chi)s}\M\(\frac{x-z}{s}\)\rho_G(z)\d s\d z 
\end{align*}
Next, we change variable $s=\frac{\abs{x-z}}{u}$, $u=\frac{\abs{x-z}}{s}$ and we obtain
\begin{align*}
\rho_G(x)&\geq(1-\chi)\int_{\R^d} \int_0^{+\infty}\frac{u^{d-2}}{\abs{x-z}^{d-1}}e^{-(1+\chi)\frac{\abs{x-z}}{u}}\M\(\frac{x-z}{\abs{x-z}}u\)\rho_G(z)\d u\d z \\
&\geq (\rho_G\star \phi_1)(x)
\end{align*}
where $\phi_1$ is given by \eqref{eq:phi_1}.

In the same way we prove the other inequality
\begin{align*}
\rho_G(x)&\leq (1+\chi)\int_{\R^d} \int_0^{+\infty}e^{-(1-\chi)s}\M(v)\rho_G(x-vs)\d s\dv \\
&\leq(1+\chi)\int_{\R^d} \int_0^{+\infty}\frac{1}{s^d}e^{-(1-\chi)s}\M\(\frac{x-z}{s}\)\rho_G(z)\d s\d z\\
&\leq(1+\chi)\int_{\R^d} \int_0^{+\infty}\frac{u^{d-2}}{\abs{x-z}^{d-1}}e^{-(1-\chi)\frac{\abs{x-z}}{u}}\M\(\frac{x-z}{\abs{x-z}}u\)\rho_G(z)\d u\d z \\
&\leq (\rho_G\star \phi_2)(x)
\end{align*}
where $\phi_2$ is given by \eqref{eq:phi_2}.
\end{proof}

\begin{cor}[Positivity]\label{cor:positivity}
The density $\rho_G$ is strictly positive, i.e.
\[
\rho_G(x)>0 \qquad \forall x\in \R^d.
\]
\end{cor}
\begin{proof}[{\bf Proof of \Cref{cor:positivity}}]
Since $G$ is a non-zero steady state, we have that $\rho_G$ is not identically zero. But in this case, for all possible $x\in\R^d$
\[
\rho_G(x)\geq \int_{\R^d}\rho_G(z)\phi_1(x-z)\d z>0
\]
where in the last inequality we used that $\phi_1>0$.
\end{proof}

\subsection{Lower bound for the density $\rho_G$}
\begin{lem}[Lower bound]\label{lem:lowerbound}
Let $G$ be the stationary solution to \eqref{eq:main}, then the density $\rho_G$ satisfies
\begin{equation}
\rho_G(x)\gtrsim \wx^{-\frac{\gamma}{\gamma+1}\(d-\frac{1}{2}\)}e^{-\nu \wx^{\frac{\gamma}{1+\gamma}}}
\end{equation}
where $\nu=\frac{\gamma+1}{\gamma}(1+\chi)^{\frac{\gamma}{1+\gamma}}$.
\end{lem}
\begin{proof}[{\bf Proof of \Cref{lem:lowerbound}}]
Fix $A>0$ such that the ball in $\R^d$ with radius $A$ has volume 1. From the inequalities of \Cref{lem:conv_inequalities} we have
\begin{align*}
\rho_G(x)&\geq \int_{\R^d}\rho_G(z)\phi_1(x-z)\d z\\
&\geq \int_{\abs{z}\leq A}\rho_G(z)\phi_1(x-z)\d z\\
&\geq \(\min_{\abs{z}\leq A}\rho_G(z)\)\int_{\abs{z}\leq A}\phi_1(x-z)\d z
\end{align*}
The minimum is strictly positive since it is taken over a compact set and $\rho_G$ is strictly positive. Concerning the other factor, notice that $\phi_1$ is radially decreasing, so for $\abs{x}\geq A$ we have 
\begin{align*}
\int_{\abs{z}\leq A}\phi_1(x-z)\d z&\geq\int_{\abs{z}\leq A}\phi_1\(x+A\frac{x}{\abs{x}}\)\d z\geq \phi_1\(x+A\frac{x}{\abs{x}}\)
\end{align*}
Using the definition of $\phi_1$ and \eqref{eq:d>2 asymptINF} we conclude
\[
\phi_1\(x+A\frac{x}{\abs{x}}\)\gtrsim \wangle{\abs{x}+A}^{-\frac{\gamma}{\gamma+1}\(d-\frac{1}{2}\)}e^{-\nu (\abs{x}+A)^{\frac{\gamma}{1+\gamma}}} \sim \wx^{-\frac{\gamma}{\gamma+1}\(d-\frac{1}{2}\)}e^{-\nu \wx^{\frac{\gamma}{1+\gamma}}}.
\]
as $\abs{x}\to\infty$. Thus we have
\[
\rho_G(x)\gtrsim \wx^{-\frac{\gamma}{\gamma+1}\(d-\frac{1}{2}\)}e^{-\nu \wx^{\frac{\gamma}{1+\gamma}}}
\]
for $\abs{x}\geq A$, that gives the statement since $\rho_G$ is bounded on $B_A$.
\end{proof}

\subsection{Upper bound for the density $\rho_G$}

From the main theorem we have that for $m=e^{\nu \wx^\frac{\gamma}{1+\gamma}}+e^{b\abs{v}^\gamma}$ the stationary solution $G$ belongs to $\sfL^1(m)$, as a consequence
\[
\int_{\R^d}\rho_G(x)e^{\nu\wx^\frac{\gamma}{1+\gamma}}\dx\leq \iint_{\R^d\times\R^d}G(x,v)(e^{\nu \wx^\frac{\gamma}{1+\gamma}}+e^{b\abs{v}^\gamma})\dx\dv <\infty.
\]
Thus we have $\rho_G\in\sfL^1(e^{\nu\wx^\a})$ with $\nu>0$ small enough. The aim of this section is to turn this bound into an $\sfL^\infty$ bound, in the sense that we want to prove $\rho_G\in\sfL^\infty(e^{\nu\wx^\frac{\gamma}{1+\gamma}})$.

\begin{lem}\label{lem:eststatupper}
Let $G$ be the stationary solution to \eqref{eq:main}, then the density $\rho_G$ satisfies
\begin{equation}
\rho_G(x)\lesssim e^{-\nu \wx^{\frac{\gamma}{1+\gamma}}}
\end{equation}
where $\nu>0$ is small enough.
\end{lem}
\begin{proof}[{\bf Proof of \Cref{lem:eststatupper}}]
We show the case $d\geq 2$, as the case $d=1$ is very similar. Let $\nu<\overline{\nu}=\frac{\gamma+1}{\gamma}(1-\chi)^\a$ small enough such that $\rho_G\in\sfL^1(e^{\nu\wx^\a})$. Using \Cref{lem:conv_inequalities} and the inequalities \eqref{eq:d>2 asympt0} and \eqref{eq:d>2 asymptINF} we have
\begin{align*}
\rho_G(x) & \leq (\rho_G\star\phi_2)(x)\\
&\lesssim \int_{\abs{x-z}\leq \epsilon}\frac{\rho_G(z)}{\abs{x-z}^{d-1}}e^{-\overline{\nu}\wangle{x-z}^{\frac{\gamma}{1+\gamma}}}\d z+\int_{\abs{x-z}\geq \epsilon}\frac{\rho_G(z)}{\abs{x-z}^{\frac{\gamma}{\gamma+1}\(d-\frac{1}{2}\)}}e^{-\overline{\nu}\wangle{x-z}^{\frac{\gamma}{1+\gamma}}}\d z\\
&\lesssim \int_{\abs{x-z}\leq \epsilon}\frac{\rho_G(z)}{\abs{x-z}^{d-1}}e^{-\nu\wangle{x-z}^{\frac{\gamma}{1+\gamma}}}\d z+\int_{\abs{x-z}\geq \epsilon}\frac{\rho_G(z)}{\abs{x-z}^{\frac{\gamma}{\gamma+1}\(d-\frac{1}{2}\)}}e^{-\nu\wangle{x-z}^{\frac{\gamma}{1+\gamma}}}\d z\\
&\lesssim \int_{\abs{x-z}\leq \epsilon}\frac{\rho_G(z)}{\abs{x-z}^{d-1}}e^{\nu\wangle{z}^{\frac{\gamma}{1+\gamma}}-\nu\wangle{x}^{\frac{\gamma}{1+\gamma}}}\d z+\int_{\abs{x-z}\geq \epsilon}\frac{\rho_G(z)}{\abs{x-z}^{\frac{\gamma}{\gamma+1}\(d-\frac{1}{2}\)}}e^{\nu\wangle{z}^{\frac{\gamma}{1+\gamma}}-\nu\wangle{x}^{\frac{\gamma}{1+\gamma}}}\d z\\
&\leq e^{-\nu\wangle{x}^{\frac{\gamma}{1+\gamma}}}\norm{\rho_G}_{\sfL^\infty\(e^{\nu\wangle{z}^{\frac{\gamma}{1+\gamma}}}\)}\int_{\abs{x-z}\leq \epsilon}\frac{1}{\abs{x-z}^{d-1}}\d z+ \frac{e^{-\nu\wx^{\frac{\gamma}{1+\gamma}}}}{\epsilon^{\frac{\gamma}{\gamma+1}\(d-\frac{1}{2}\)}}\norm{\rho_G}_{\sfL^1\(e^{\nu\wangle{z}^{\frac{\gamma}{1+\gamma}}}\)}
\end{align*}
Then we have
\begin{align*}
\norm{\rho_G}_{\sfL^\infty\(e^{\nu\wangle{x}^{\frac{\gamma}{1+\gamma}}}\)}\lesssim \epsilon \norm{\rho_G}_{\sfL^\infty\(e^{\nu\wangle{z}^{\frac{\gamma}{1+\gamma}}}\)}+\frac{1}{\epsilon^{\frac{\gamma}{\gamma+1}\(d-\frac{1}{2}\)}}\norm{\rho_G}_{\sfL^1\(e^{\nu\wangle{z}^{\frac{\gamma}{1+\gamma}}}\)}
\end{align*}
Then choosing $\epsilon$ sufficiently small allows us to conclude.
\end{proof}

\section{Refined bounds and a Poincaré inequality under \Cref{Hyp:asympRho} }\label{Sec:poincare}

\subsection{Upper and lower bounds for $P^{(k)}_G$}
In this section we show some more precise $\sfL^\infty$ bounds under the stronger assumptions made in \Cref{Hyp:asympRho}. Specifically, we give upper and lower bounds on the quantities
\begin{equation}\label{eq:H_G}
    H_G(x)=\int_\Rd h(v)G(x,v)\dv. 
\end{equation}
where $h\colon\Rd\to \R$ is an homogeneous function of degree $k\in\NN$, that is 
\[
h(\lambda v)=\lambda^kh(v)\qquad \forall \lambda> 0.
\]
The most interesting cases are $h(v)=\abs{v}^k$ with $k\in\NN$ and $h(v)=(\sfe\cdot v)^2$ with $\sfe\in\S^{d-1}$, the former allows to study the  $k$-th moments of $G$ in velocity
\[
P_G^{(k)}(x)=\int_\Rd \abs{v}^kG(x,v)\dv \qquad k\in \NN,
\]
the latter will show the equivalence of the weights $V_G$ and $P_G$.
\begin{lem}\label{lem:asympMomenta}
Assume \Cref{Hyp:asympRho} and let $h\colon\Rd\to\R$ be an homogeneous function of degree $k\geq 1$. Then
\begin{align*}
    H_G(x)\asymp \(\wangle{x}^{\ell} + \wangle{x}^{\frac{k}{1+\gamma} }\)\wx^{-\frac{\gamma}{1+\gamma}\(d-\frac12\)}e^{-\nu\abs{x}^{\frac{\gamma}{1+\gamma}}}
\end{align*}
In particular we have
\begin{align}\label{eq:weightsAsymp}
    \rho_G(x)&\asymp \wx^{\ell-\frac{\gamma}{1+\gamma}\(d-\frac12\)}e^{-\nu\abs{x}^{\frac{\gamma}{1+\gamma}}}, & P_G^{(k)}(x)&\asymp \wx^{\frac{k}{1+\gamma}-\frac{\gamma}{1+\gamma}\(d-\frac12\)}e^{-\nu\abs{x}^{\frac{\gamma}{1+\gamma}}}
\end{align}
for $k\geq 2$.
\end{lem}
\begin{proof}[{\bf Proof of \Cref{lem:asympMomenta}}]
$\#$\textit{Step 1: Rewriting of $H_G$ } Under \Cref{Hyp:asympRho} we have $\Lambda=1+\chi\,\sign$, so  the representation formula \eqref{eq:representation} rewrites as
\[
\qquad G(x,v)=\int_0^{+\infty}e^{-\int_0^s\Lambda(x\cdot v-\abs{v}^2r)dr}\M(v)\theta_G(x-vs)\d s.
\]
Let us denote $\sfe_z=\tfrac{z}{\abs z}$ for any $z\in\Rd$. By changing the variables $z=sv$ and $u=\tfrac{\abs{z}}{s}$, we find
\begin{align*}
    H_G(x)&=\int_\Rd \int_0^{+\infty}h(v)e^{-\int_0^s\Lambda(x\cdot v-\abs{v}^2r)dr}\M(v)\theta_G(x-vs)\d s\dv\\
    &=\int_\Rd \int_0^{+\infty}h\(\frac{z}{s}\)e^{-\int_0^s\Lambda(x\cdot\frac{z}{s}-\frac{\abs{z}^2}{s^2}r)dr}\M\(\frac{z}{s}\)\theta_G(x-z)\frac{1}{s^d}\d s\d z\\
    &=\int_\Rd \int_0^{+\infty}h(\sfe_z)\frac{u^{k+d-2}}{\abs{z}^{d-1}}e^{-\int_0^\frac{\abs{z}}{u}\Lambda(x\cdot\sfe_z-ur)dr}\M\(\sfe_z u\)\theta_G(x-z)\d u\d z\\
    &\asymp \int_\Rd \int_0^{+\infty}h(\sfe_z)\frac{u^{k+d-2}}{\abs{z}^{d-1}}e^{-\frac{u^\gamma}{\gamma}-\frac{1}{u}\int_0^{\abs{z}}\Lambda(x\cdot \sfe_z-r)dr}\theta_G(x-z)\d u\d z
\end{align*}
Take $\sfe_x=\tfrac{x}{\abs{x}}$ fixed and change variables $z=\abs{x}y$ and $r=\abs{x}\tau$. We have
\begin{align*}
    H_G(x)&\asymp \abs{x}\int_\Rd \int_0^{+\infty}h(\sfe_y)\frac{u^{k+d-2}}{\abs{y}^{d-1}}e^{-\frac{u^\gamma}{\gamma}-\frac{1}{u}\int_0^{\abs{x}\abs{y}}\Lambda(x\cdot \sfe_y-r)dr}\theta_G(\abs{x}(\sfe_x-y))\d u\d y\\
    &\asymp \abs{x}\int_\Rd \int_0^{+\infty}h(\sfe_y)\frac{u^{k+d-2}}{\abs{y}^{d-1}}e^{-\frac{u^\gamma}{\gamma}-\frac{\abs{x}}{u}\int_0^{\abs{y}}\Lambda(\sfe_x\cdot \sfe_y-\tau)\d\tau}\theta_G(\abs{x}(\sfe_x-y))\d u\d y\\
    &\asymp \abs{x}\int_\Rd \frac{h(\sfe_y)\theta_G(\abs{x}(\sfe_x-y))}{\abs{y}^{d-1}}\(\int_0^{+\infty}u^{k+d-2}e^{-\frac{u^\gamma}{\gamma}-\frac{\abs{x}}{u}\int_0^{\abs{y}}\Lambda(\sfe_x\cdot \sfe_y-\tau)\d\tau}\d u\)\d y.
\end{align*}
Thanks to \Cref{lem:phi2} in the Appendix, we have
\[
\int_0^\infty u^{k+d-2}e^{-\frac{u^\gamma}{\gamma}-\frac{X}{u}}\d u\asymp  \wangle{X}^{\frac{k}{1+\gamma}+(d-1)-\frac{\gamma}{1+\gamma}\(d-\frac12\)}e^{-\frac{1+\gamma}{\gamma}X^{\frac{\gamma}{1+\gamma}}} \qquad \forall X\in\R_+,
\]
so, by using it with $X=\abs{x} \int_0^{\abs{y}} \Lambda(\sfe_x\cdot \sfe_y-\tau)\d\tau\asymp \abs{x}\abs{y}$, we deduce
\begin{align*}
    H_G(x)&\asymp \abs{x}\int_\Rd\frac{h(\sfe_y)\theta_G(\abs{x}(\sfe_x-y))}{\abs{y}^{d-1}}\wangle{\abs{x}\abs{y}}^{\frac{k}{1+\gamma}+(d-1)-\frac{\gamma}{1+\gamma}\(d-\frac12\)} e^{-\frac{1+\gamma}{\gamma}\( \abs{x} \int_0^{\abs{y}} \Lambda(\sfe_x\cdot \sfe_y-\tau)\d\tau\)^{\frac{\gamma}{1+\gamma}}} \d y
\end{align*}
Use the bound \eqref{eq:boundRHO} and replace $\wangle{\cdot}$ with $\abs{\cdot}$ (this does not modify the asymptotic behavior), so
\begin{align*}
    H_G(x)&\asymp\abs{x}^{d+ \frac{k}{1+\gamma}+\ell-\frac{2\gamma}{1+\gamma}\(d-\frac12\)}\int_\Rd g(y)e^{-\abs{x}^{\frac{\gamma}{1+\gamma}}\(\frac{1+\gamma}{\gamma}\(  \int_0^{\abs{y}} \Lambda(\sfe_x\cdot \sfe_y-\tau)\d\tau\)^{\frac{\gamma}{1+\gamma}} + \nu\abs{\sfe_x-y}^{\frac{\gamma}{1+\gamma}}\)}\d y
\end{align*}
where
\[
g(y)=h(\sfe_y)\abs{\sfe_x-y}^{\ell-\frac{\gamma}{1+\gamma}\(d-\frac12\)} \abs{y}^{\frac{k}{1+\gamma}-\frac{\gamma}{1+\gamma}\(d-\frac12\)}.
\]
Now we study the integral 
\[
I(X)=\int_\Rd g(y)e^{-X w(y)}\d y
\]
where $X=\abs{x}^{\frac{\gamma}{1+\gamma}}$ and
\[
w(y)=\frac{1+\gamma}{\gamma}\(  \int_0^{\abs{y}} \Lambda(\sfe_x\cdot \sfe_y-\tau)\d\tau\)^{\frac{\gamma}{1+\gamma}} + \nu\abs{\sfe_x-y}^{\frac{\gamma}{1+\gamma}}
\]
$\#$\textit{Step 2: Study of $I(X)$.} The function $w$ has two minima in $\Rd$, that are $0$ and $\sfe_x$, and the minimum value is $\nu=\tfrac{1+\gamma}{\gamma}(1+\chi)^{\tfrac{\gamma}{1+\gamma}}$. Let $\epsilon>0$ small and split
\[
I(X)= \int_{w\geq \nu+\epsilon} g(y)e^{-X w(y)}\d y+\int_{w< \nu+\epsilon} g(y)e^{-X w(y)}\d y
\]
The first integral is negligible as $X\to \infty$, indeed
\begin{align*}
    \abs*{\int_{w\geq \nu+\epsilon} g(y)e^{-X w(y)}\d y}\leq {e^{-(\nu+\tfrac{\epsilon}{2})X}}\int_{w\geq \nu+\epsilon} \abs{g(y)}e^{-\frac{\epsilon}{2}X w(y)}\d y=o(X^n{e^{-\nu X}})
\end{align*}
as $X\to\infty$ for all $n\in\NN$.

Concerning the second integral, we take $\epsilon>0$ small enough such that the set $\{w(z)< \nu+\epsilon\}$ is the union of two disjoint neighborhoods $D_0$ and $D_x$ of $0$ and $\sfe_x$ respectively. We now consider separately the two domains 
\[
\int_{w< \nu+\epsilon} g(y)e^{-X w(y)}\d y=\int_{D_0} g(y)e^{-X w(y)}\d y+\int_{D_x} g(y)e^{-X w(y)}\d y=:I_0(X) + I_1(X).
\]
Concerning $I_0(X)$, we pass to spherical coordinates
\[
I_0(X)= \int_{\rho\omega\in D_0} g(\rho\omega)\rho^{d-1}e^{-X \(\frac{1+\gamma}{\gamma}\(  \int_0^{\rho} \Lambda(\sfe_x\cdot \omega-\tau)\d\tau\)^{\frac{\gamma}{1+\gamma}} + \nu\abs{\sfe_x-\rho\omega}^{\frac{\gamma}{1+\gamma}}\)}\d \rho\d \omega
\]
If $\epsilon>0$ is small enough, the function 
\[
\rho\mapsto \frac{1+\gamma}{\gamma}\(  \int_0^{\rho} \Lambda(\sfe_x\cdot \omega-\tau)\d\tau\)^{\frac{\gamma}{1+\gamma}} + \nu\abs{\sfe_x-\rho\omega}^{\frac{\gamma}{1+\gamma}} -\nu
\]
is strictly increasing for any $\omega\in\mathbb{S}^{d-1}$, so we can change variables
\begin{equation*}
    \begin{cases}
        r=\frac{1+\gamma}{\gamma}\(  \int_0^{\rho} \Lambda(\sfe_x\cdot \omega-\tau)\d\tau\)^{\frac{\gamma}{1+\gamma}} + \nu\abs{\sfe_x-\rho\omega}^{\frac{\gamma}{1+\gamma}} -\nu\\
        \sigma=\omega.
    \end{cases}
\end{equation*}
Notice that the Jacobian of the transformation is
\begin{equation*}
\abs{J(r,\sigma)}=\abs*{
\begin{matrix}
    \frac{\partial \rho(r,\sigma)}{\partial r} & \frac{\partial \rho(r,\sigma)}{\partial \sigma}\\[1ex]
    \frac{\partial \omega(r,\sigma)}{\partial r} & \frac{\partial \omega(r,\sigma)}{\partial \sigma}
\end{matrix}}
=\abs*{
\begin{matrix}
    \frac{\partial \rho(r,\sigma)}{\partial r} & \frac{\partial \rho(r,\sigma)}{\partial \sigma}\\[1ex]
    0 & \mathbb{I}_{d-1}
\end{matrix}}
=\abs*{\frac{\partial \rho(r,\sigma)}{\partial r}}
\end{equation*}
So we have
\begin{align*}
I_0(X)&=e^{-\nu X}\int_{r=0}^\epsilon\int_{\sigma\in\mathbb{S}^{d-1}} g(\rho(r,\sigma)\sigma)\rho(r,\sigma)^{d-1}\abs*{\frac{\partial \rho(r,\sigma)}{\partial r}} e^{-X r}\d r d\sigma\\
&=e^{-\nu X}\int_{r=0}^\epsilon\(\int_{\sigma\in\mathbb{S}^{d-1}} g(\rho(r,\sigma)\sigma)\rho(r,\sigma)^{d-1}\abs*{\frac{\partial \rho(r,\sigma)}{\partial r}}\d \sigma\) e^{-X r} \d r
\end{align*}
We want to apply the Watson \Cref{lem:Watson} about the asymptotic behavior of the Laplace integrals as $X\to \infty$, so we have to study the behavior of 
\[
r\mapsto\int_{\sigma\in\mathbb{S}^{d-1}} g(\rho(r,\sigma)\sigma)\rho(r,\sigma)^{d-1}\abs*{\frac{\partial \rho(r,\sigma)}{\partial r}}\d \sigma
\]
as $r\to 0$. Since $D_0$ is a small neighborhood of $0$ we have that, $\rho(r,\sigma)\to0$ if and only if $r\to 0$. Moreover, combining 
\[
\nu\abs{\sfe_x-\rho\omega}^{\frac{\gamma}{1+\gamma}} -\nu=- \frac{\nu\gamma}{1+\gamma}(\sfe_x\cdot\omega)\,\rho +o(\rho)\qquad \text{as } \rho\to 0
\]
and
\[
\(  \int_0^{\rho} \Lambda(\sfe_x\cdot \omega-\tau)\d\tau\)^{\frac{\gamma}{1+\gamma}} = \(1+\chi\sign(\sfe_x\cdot \omega)\)^{\frac{\gamma}{1+\gamma}} \rho^{\frac{\gamma}{1+\gamma}} + o(\rho^\a)\qquad\textbf{as $\rho\to 0$},
\]
we find
\[
r(\rho,\omega)\sim \frac{1+\gamma}{\gamma}\(1+\chi\sign(\sfe_x\cdot \omega)\)^{\frac{\gamma}{1+\gamma}} \rho^{\frac{\gamma}{1+\gamma}} \qquad\textbf{as $\rho\to 0$}
\]
uniformly in $\omega\in\mathbb{S}^{d-1}$. As a consequence
\begin{align*}
    \rho(r,\sigma)&\sim \(\frac{\gamma}{1+\gamma}\)^{\frac{1+\gamma}{\gamma}}\frac{r^{\tfrac{1+\gamma}{\gamma}}}{1+\chi\sign(\sfe_x\cdot \sigma)}& \abs*{\frac{\partial\rho(r,\sigma)}{\partial r}}&\sim \(\frac{\gamma}{1+\gamma}\)^{\frac{1}{\gamma}}\frac{r^{\tfrac{1+\gamma}{\gamma}-1}}{1+\chi\sign(\sfe_x\cdot \sigma)}
\end{align*}
as $r\to 0$, uniformly in $\sigma\in\mathbb{S}^{d-1}$. Therefore, as $r\to0$, we have
\begin{align*}
    g(\rho(r,\sigma)\sigma)\rho(r,\sigma)^{d-1} &\sim h(\sigma) \rho(r,\sigma)^{d-1+\frac{k}{1+\gamma}-\frac{\gamma}{1+\gamma}\(d-\frac12\)}\\
    &\sim\(\frac{\gamma}{1+\gamma}\)^{\frac{1+\gamma}{\gamma}(d-1)+\frac{k}{\gamma}-\(d-\frac12\)}\frac{h(\sigma)r^{\frac{1+\gamma}{\gamma}(d-1)+\frac{k}{\gamma}-\(d-\frac12\)}}{(1+\chi\sign(\sfe_x\cdot \sigma))^{d-1+\frac{k}{1+\gamma}-\frac{\gamma}{1+\gamma}\(d-\frac12\)}}
\end{align*}
and 
\begin{align*}
    \int_{\sigma\in\mathbb{S}^{d-1}} &g(\rho(r,\sigma)\sigma)\rho(r,\sigma)^{d-1}\abs*{\frac{\partial \rho(r,\sigma)}{\partial r}}\d \sigma\sim \alpha_0\, r^{\frac{1+\gamma}{\gamma}d+\frac{k}{\gamma}-\(d-\frac12\)-1}, 
\end{align*}
where the constant $\alpha_0$ is
\[
\alpha_0 =\(\frac{\gamma}{1+\gamma}\)^{\frac{1+\gamma}{\gamma}(d-1)+\frac{k+1}{\gamma}-\(d-\frac12\)}\int_{\sigma\in\mathbb{S}^{d-1}}h(\sigma)(1+\chi\sign(\sfe_x\cdot\sigma))^{-d+1-\frac{k}{1+\gamma}+\frac{\gamma}{1+\gamma}\(d-\frac12\)}\d\sigma.
\]
By the Watson Lemma, we conclude that, up to a constant, 
\[
I_0(X)\sim e^{-\nu X}X^{-\frac{1+\gamma}{\gamma}d-\frac{k}{\gamma} +d-\frac12}=e^{-\nu \abs{x}^{\frac{\gamma}{1+\gamma}}}\abs{x}^{-d-\frac{k}{1+\gamma}+\frac{\gamma}{1+\gamma}\(d-\frac12\)}
\]

In a similar way we can study $I_1(X)$. By considering the spherical coordinates in $\Rd$ centered in $\sfe_x$ we find
\begin{align*}
    I_1(X)&=\int_{\sfe_x-D_x}g(\sfe_x-y)e^{-X\(\frac{1+\gamma}{\gamma}\(  \int_0^{\abs{\sfe_x-y}} \Lambda(\sfe_x\cdot \frac{\sfe_x-y}{\abs{\sfe_x-y}}-\tau)\d\tau\)^{\frac{\gamma}{1+\gamma}} + \nu\abs{y}^{\frac{\gamma}{1+\gamma}}\)}\d y\\
    &=\int_{\rho\omega\in \sfe_x-D_x}g(\sfe_x-\rho\omega)e^{-X\(\frac{1+\gamma}{\gamma}\(  \int_0^{\abs{\sfe_x-\rho\omega}} \Lambda(\sfe_x\cdot \frac{\sfe_x-\rho\omega}{\abs{\sfe_x-\rho\omega}}-\tau)\d\tau\)^{\frac{\gamma}{1+\gamma}} + \nu\rho^{\frac{\gamma}{1+\gamma}}\)}\rho^{d-1}\d \rho\d\omega\\
    &=e^{-\nu X}\int_{r=0}^\epsilon \int_{\sigma\in\mathbb{S}^{d-1}}g(\sfe_x-\rho(r,\sigma)\sigma)\rho(r,\sigma)^{d-1}e^{-Xr}\abs*{\frac{\partial\rho(r,\sigma)}{\partial r}}\d r\d\sigma\\
    &=e^{-\nu X}\int_{r=0}^\epsilon\( \int_{\sigma\in\mathbb{S}^{d-1}}g(\sfe_x-\rho(r,\sigma)\sigma)\rho(r,\sigma)^{d-1}\abs*{\frac{\partial\rho(r,\sigma)}{\partial r}}d\sigma\)e^{-Xr}\d r
\end{align*}
This time 
\begin{align*}
    \rho(r,\sigma)&\sim \frac{1}{\nu}r^{\tfrac{1+\gamma}{\gamma}}& \abs*{\frac{\partial\rho(r,\sigma)}{\partial r}}&\sim \frac{1+\gamma}{\nu\gamma}r^{\tfrac{1+\gamma}{\gamma}-1}
\end{align*}
and
\[
g(\sfe_x-\rho(r,\sigma)\sigma)\rho(r,\sigma)^{d-1}\sim \frac{h(\sfe_x)}{\nu^{\ell- \frac{\gamma}{1+\gamma}\(d-\frac12\)+d-1}}r^{(d-1+\ell)\frac{1+\gamma}{\gamma} - \(d-\frac12\)}
\]
We conclude
\begin{align*}
    \int_{\sigma\in\mathbb{S}^{d-1}}g(\sfe_x-\rho(r,\sigma)\sigma)\rho(r,\sigma)^{d-1}\abs*{\frac{\partial\rho(r,\sigma)}{\partial r}}d\sigma\sim \alpha_1\,r^{(d+\ell)\frac{1+\gamma}{\gamma} - \(d-\frac12\)-1}
\end{align*}
where the constant $\alpha_1$ is 
\[
\alpha_1 = \abs{\mathbb{S}^{d-1}}\frac{1+\gamma}{\gamma \nu^{\ell- \frac{\gamma}{1+\gamma}\(d-\frac12\)+d}}h(\sfe_x).
\]
Hence, up to a constant,
\[
I_1(X)\sim X^{-(d+\ell)\frac{1+\gamma}{\gamma}+\(d-\frac12\)}e^{-\nu X} =\abs{x}^{-d-\ell+\frac{\gamma}{1+\gamma}\(d-\frac12\)}e^{-\nu \abs{x}^{\frac{\gamma}{1+\gamma}}} 
\]
Putting the pieces together
\begin{align*}
    H_G(x)&\asymp \wangle{x}^{d+ \frac{k}{1+\gamma}+\ell-\frac{2\gamma}{1+\gamma}\(d-\frac12\)}\(\wangle{x}^{-d-\frac{k}{1+\gamma}+\frac{\gamma}{1+\gamma}\(d-\frac12\)} + \wangle{x}^{-d-\ell+\frac{\gamma}{1+\gamma}\(d-\frac12\)}\)e^{-\nu\abs{x}^{\frac{\gamma}{1+\gamma}}}\\
    &\asymp \(\wangle{x}^{\ell} + \wangle{x}^{\frac{k}{1+\gamma}}\)\wangle{x}^{-\frac{\gamma}{1+\gamma}\(d-\frac12\)}e^{-\nu\abs{x}^{\frac{\gamma}{1+\gamma}}}
\end{align*}
that is the statement.
\end{proof}

This result is useful for the following. Let the matrix $V_G$ be defined by
\[
V_G(x)=\int_\Rd v\otimes v\,G(x,v)\dv,
\]

\begin{lem}\label{lem:VG_PG}
The matrix $V_G(x)$ is symmetric and positive definite for all $x\in\Rd$. Moreover, we have
\[
\xi\cdot(V_G(x)\xi)\asymp P_G(x)\abs{\xi}^2\qquad \forall \xi\in\Rd,\,\forall x\in\Rd.
\]
\end{lem}
\begin{proof}[{\bf Proof of \Cref{lem:VG_PG}}]
Clearly $V_G$ is symmetric and positive definite because 
\[
\xi\cdot (V_G(x) \xi) =\int_\Rd (\xi\cdot v)^2G(x,v)\dv> 0 \qquad\forall \xi\in\Rd,\;x\in\Rd
\]
Moreover 
\[
\xi\cdot (V_G(x) \xi)\leq \int_\Rd \abs{\xi}^2 \abs{v}^2G(x,v)\dv =\abs{\xi}^2 P_G(x)\qquad\forall \xi\in\Rd,\;x\in\Rd
\]
and
\[
\xi\cdot (V_G(x) \xi)\geq\abs{\xi}^2\min_{\mathsf e\in\S^{d-1}} \int_\Rd (\mathsf e\cdot v)^2G(x,v)\dv \qquad\forall \xi\in\Rd,\;x\in\Rd
\]
Since $\S^{d-1}$ is compact, the minimum is achieved in a point $\Bar{\sfe}$ and we can study the integral $\int_\Rd (\Bar\sfe\cdot v)^2G(x,v)\dv$ as in \Cref{lem:asympMomenta}. Since both functions $v\mapsto (\Bar\sfe\cdot v)^2$ and $v\mapsto\abs v^2$ are homogeneous of degree $2$, we have
\[
\xi\cdot (V_G(x) \xi)\gtrsim \abs{\xi}^2\wx^{\frac{2}{1+\gamma}-\frac{\gamma}{1+\gamma}\(d-\frac12\)}e^{-\nu\abs{x}^{\a}}\gtrsim \abs{\xi}^2 P_G(x).
\]
\end{proof}

\subsection{Weighted Poincaré inequality for the variance $P_G$}

The aim of this section is to establish a weighted Poincaré inequality with weight $P_G$ and averaged by $\rho_G$, in particular we prove the following.
\begin{prop}\label{prop:Poincaré}
Let $\rho_G$ be the density of the steady state $G$ of the run and tumble equation and $P_G$ its variance. Then there exists an optimal constant $C_P>0$ such that for any function $u\in \mathcal C^1_b(\R)$
\begin{equation}\label{eq:Poincaré}
    \int_\Rd\abs{u-\Bar{u}}^2\wx^{-\frac{2}{1+\gamma}}P_G\dx\leq C_P \int_\Rd \abs{\gradx u}^2 P_G\dx.
\end{equation}
where $\Bar{u}=\int_\Rd u\rho_G\dx$.
\end{prop}
Many proves of this kind of inequality are based on estimates on the first and second derivatives of the weight, see for example \cite{KMN21, BDLS20}. However, in our case we cannot determine explicitly the first and second derivatives of $P_G$ and the representation formula \eqref{eq:representation} is not giving suitable information about them.

\begin{proof}[\bf Proof of Proposition \ref{prop:Poincaré}]
We know from \Cref{lem:asympMomenta} that
\begin{equation}\label{eq:varthetaDEF}
    P_G(x)\asymp e^{-\vartheta(x)}.
\end{equation}
where 
\[
\vartheta(x): =  \nu \abs{x}^{\frac{\gamma}{1+\gamma}}-\(\tfrac{2}{1+\gamma}-\tfrac{\gamma}{1+\gamma}\(d-\tfrac{1}{2}\)\)\ln(\wx).
\]
In particular we can see that $\abs{\gradx\vartheta(x)}^2\sim \frac{\nu^2\gamma^2}{(1+\gamma)^2}\abs x^{-\frac{2}{1+\gamma}}$ and $\abs{\Delta_x\vartheta}\sim \frac{\nu\gamma}{(1+\gamma)^2}\abs x^{-\frac{2+\gamma}{1+\gamma}} $ as $\abs{x}\to \infty$, therefore
\begin{equation}\label{eq:vartheta}
    \frac{\abs{\gradx \vartheta}^2}{4} - \frac{\Delta_x \vartheta}{2}\sim \frac{\nu^2\gamma^2}{4(1+\gamma)^2}\wx^{-\frac{2}{1+\gamma}}
\end{equation}
as $\abs x\to \infty$.

We follow the proof of \cite[Cor. 10]{BDLS20}. Without loss of generality we can assume $\Bar{u}=\int_\Rd u\rho_G\dx=0$ up to replacement of $u$ by $u-\Bar u$. Let $\varrho\colon [0,\infty)$ be a cut off such that $0\leq \varrho\leq 1$, $\varrho\equiv 1$ on $[0,1]$, $\varrho\equiv 0$ on $[2,\infty)$ and such that
\[
\frac{(\varrho')^2}{1-\varrho^2}\leq \kappa
\]
for some $\kappa>0$. Then, for any $R>0$, we define $\varrho_R(x)=\varrho(\abs x/R)$ and the functions
\begin{align*}
    u_{1,R}&=u\,\varrho_R & u_{2,R}&=u\,\sqrt{1-\varrho_R^2},
\end{align*}
so that $u_{1,R}$ is supported in $B_{2R}$ and $u_{2,R}$ is supported in $B_R^c$. Elementary computations show that 
$u^2 = u_{1,R}^2 + u_{2,R}^2$ and $\abs{\gradx u_{1,R}}^2+\abs{\gradx u_{1,R}}^2 = \abs{\gradx u}^2  +\abs u^2 \frac{(\varrho_R')^2}{1-\varrho_R^2}$, so that
\begin{equation}\label{eq:U12R}
    \abs*{\abs{\gradx u}^2-\abs{\gradx u_{1,R}}^2-\abs{\gradx u_{1,R}}^2 } =  \abs u^2 \frac{(\varrho_R')^2}{1-\varrho_R^2}\leq \abs u^2 \frac{\kappa}{R^2}.
\end{equation}
We estimate $u_{1,R}$ and $u_{2,R}$ separately. Concerning $u_{2,R}$, we set $u_{2,R}=w \,e^{\vartheta/2}$ and using \eqref{eq:varthetaDEF} we get
\begin{align*}
    \int_\Rd \abs{\gradx u_{2,R}}^2P_G \dx&\gtrsim \int_\Rd\abs{\gradx (w \,e^{\vartheta/2})}^2\,e^{-\vartheta} \dx\\
    &=\int_\Rd \left( \abs{\gradx w}^2+ \abs w^2\left(\frac{\abs{\gradx \vartheta}^2}{4} - \frac{\Delta_x \vartheta}{2}\right) \right) \dx\\
    &\geq\int_\Rd  \abs w^2\left(\frac{\abs{\gradx \vartheta}^2}{4} - \frac{\Delta_x \vartheta}{2}\right) \dx\\
    &=\int_\Rd  \abs{u_{2,R}}^2\left(\frac{\abs{\gradx \vartheta}^2}{4} - \frac{\Delta_x \vartheta}{2}\right)e^{-\vartheta} \dx.
\end{align*}
Thanks to \eqref{eq:vartheta}, for $R>0$ large enough we have
\[
\frac{\abs{\gradx \vartheta}^2}{4} - \frac{\Delta_x \vartheta}{2}\geq 0  \qquad \text{ for all } \abs x\geq R.
\]
Since $u_{2,R}$ is supported in $B_{R}^c$, we have
\begin{align}\label{eq:gradU2}
    \begin{split}
         \int_\Rd \abs{\gradx u_{2,R}}^2P_G \dx&\gtrsim \int_\Rd  \abs{u_{2,R}}^2\left(\frac{\abs{\gradx \vartheta}^2}{4} - \frac{\Delta_x \vartheta}{2}\right)P_G \dx \geq \mathsf q(R)\int_\Rd \abs{u_{2,R}}^2\wx^{-\frac{2}{1+\gamma}}P_G\, \dx
    \end{split}
\end{align}
where
\[
\mathsf q (R)=\inf_{x\in B_{R}^c}\left(\frac{\abs{\gradx \vartheta}^2}{4} - \frac{\Delta_x \vartheta}{2}\right)\wx^{\frac{2}{1+\gamma}} \to \frac{\nu^2\gamma^2}{4(1+\gamma)^2}
\]
as $R\to \infty$.

Concerning $u_{1,R}$, we can use \eqref{eq:vartheta} as in  \cite[Prop. 9]{BDLS20} to get the inequality
\[
\int_\Rd \abs{u_{1,R}- \widetilde{u}_{1,R}}^2\wx^{-\frac{2}{1+\gamma}} e^{-\vartheta}\dx\lesssim \int_{\Rd}\abs{\gradx u_{1,R}}^2 e^{-\vartheta}\dx
\]
where $\widetilde{u}_{1,R}=\frac{\int_\Rd u\wx^{-\frac{2}{1+\gamma}}e^{-\vartheta}\dx}{\int_\Rd\wx^{-\frac{2}{1+\gamma}}e^{-\vartheta}\dx}$. As a consequence, using \eqref{eq:varthetaDEF}, we have
\begin{align*}
    \int_\Rd\abs{\gradx u_{1,R}}^2P_G\dx&\gtrsim \int_\Rd \abs{u_{1,R}-\widetilde{u}_{1,R}}^2\wx^{-\frac{2}{1+\gamma}}P_G\, \dx\\
    &\geq  \int_{B_{2R}} \abs{u_{1,R}-\widetilde{u}_{1,R}}^2\wx^{-\frac{2}{1+\gamma}}P_G\, \dx\\
    &\gtrsim \wangle{2R}^{-\ell} \int_{B_{2R}} \abs{u_{1,R}-\widetilde{u}_{1,R}}^2\rho_G\, \dx\\
    &\geq \wangle{2R}^{-\ell} \min_{c\in\R}\int_{B_{2R}} \abs{u_{1,R}-c}^2\rho_G\, \dx\\
    &= \wangle{2R}^{-\ell}\int_\Rd \abs{u_{1,R}}^2\rho_G\, \dx- \frac{ \wangle{2R}^{-\ell}}{\int_{B_{2R}}\rho_G\dx}\left(\int_\Rd u_{1,R}\rho_G\,\dx\right)^2
\end{align*}
where in the last step we explicitly computed the minimum and used that the $u_{1,R}$ is supported in $B_{2R}$. By the assumption $\Bar{u}=0$ we have
\[
\int_{B_R}u\,\rho_G\,\dx=-\int_{B_R^c} u\,\rho_G\,\dx
\]
from which we deduce
\begin{align*}
    \left(\int_\Rd u_{1,R}\rho_G\,\dx\right)^2 &= \left(\int_{B_R} u\,\rho_G\,\dx + \int_{B_R^c} \varrho_R u\,\rho_G\,\dx\right)^2\\
    &= \left( \int_{B_R^c} (\varrho_R-1) u\,\rho_G\,\dx\right)^2\\
    &\leq \left( \int_\Rd \abs{u}^2\,\wx^{-\frac{2}{1+\gamma}}P_G\,\dx\right)\left( \int_{B_R^c}\rho_G^2 \wx^{\frac{2}{1+\gamma}}\,P_G^{-1}\,\dx\right),
\end{align*}
where the last step is just a Cauchy-Schwarz inequality and $\abs{\varrho_R-1}\leq 1$. Let 
\[
\eps(R)=\frac{ \wangle{2R}^{-\ell}}{\int_{B_{2R}}\rho_G\dx}\left( \int_{B_R^c} \rho_G^2\wx^{\frac{2}{1+\gamma}}\,P_G^{-1}\,\dx\right).
\]
We know that
\begin{align*}
    \lim_{R\to\infty}\eps(R)&=0 & \lim_{R\to\infty}\eps(R)\wangle{2R}^{\ell}&=0
\end{align*}
and
\begin{align}\label{eq:gradU1}
    \int_\R\abs{\gradx u_{1,R}}^2P_G\dx&\gtrsim \wangle{2R}^{-\ell}\int_\Rd \abs{u_{1,R}}^2\rho_G\, \dx- \eps(R)\int_\Rd \abs{u}^2\,\wx^{-\frac{2}{1+\gamma}}P_G\,\dx \notag\\
    &\gtrsim \wangle{2R}^{-\ell}\int_\Rd \abs{u_{1,R}}^2\wx^{-\frac{2}{1+\gamma}}P_G\, \dx- \eps(R)\int_\Rd \abs{u}^2\,\wx^{-\frac{2}{1+\gamma}}P_G\,\dx.
\end{align}
Finally we can use \eqref{eq:U12R}, \eqref{eq:gradU2} and \eqref{eq:gradU1} to find
\begin{align*}
    \int_\R\abs{\gradx u}^2P_G\dx\geq& \int_\R\abs{\gradx u_{1,R}}^2P_G\dx+\int_\R\abs{\gradx u_{2,R}}^2P_G\dx - \frac{\kappa}{R^2}\int_\Rd \abs u^2\,P_G\, \dx\\
    \gtrsim& \,\mathsf q(R)\int_\Rd  \abs{u_{2,R}}^2\wx^{-\frac{2}{1+\gamma}}P_G\, \dx + \wangle{2R}^{-\ell}\int_\Rd \abs{u_{1,R}}^2\wx^{-\frac{2}{1+\gamma}}P_G\, \dx+\\
    &-\eps(R)\int_\Rd \abs{u}^2\,\wx^{-\frac{2}{1+\gamma}}\rho_G\,\dx - \frac{\kappa}{R^2}\int_\Rd \abs u^2\,P_G\, \dx\\
    \geq & \left(\min\{\mathsf{q}(R), \wangle{2R}^{-\ell}\} - \eps(R)- \frac{\kappa}{R^2} \right)\int_\Rd \abs{u}^2\,\wx^{-\frac{2}{1+\gamma}}P_G\,\dx
\end{align*}
and $\min\{\mathsf{q}(R), \wangle{2R}^{-\ell}\} - \eps(R)- \frac{\kappa}{R^2}$ is positive for $R>0$ large enough.

\end{proof}

%%%%%%%%%%%%%%%%%%%%%%%%%%%%%%%%%%%%%%%%%%%%%%%%%%%%%%%%%%%%%%%%%%%%%%%%%%%%%%%%%%%%%%%%%%%%%%%%%%%%%%%%%%%%%%%%%%%%%%%%%%%%%%%%%%%%%%%%%%%%%%%%%%%%

\section{$\sfL^2$ bounds via a Dolbeault - Mouhot - Schmeiser method}\label{sec:DMS}

In this section we prove \Cref{thm:DMS}. Let us define the entropy functional $\sfH\colon\sfL^2(G^{-1})\to\R$ as 
\begin{equation}\label{eq:H}
    \sfH[f]= \norm{f}^2_{\sfL^2(G^{-1})} + \eps\bangle{ m_f(x) ,\,\nabla_x \(B^{-1}\rho_f\)}_{\sfL^2}
\end{equation}
where $\eps>0$,
\begin{align*}
    \rho_f(x)&=\int_\Rd f(x,v)\dv  & m_f(x)&=\int_\Rd v\,f(x,v)\dv 
\end{align*}
and $B$ is the elliptic operator defined by
\begin{equation}\label{eq:B}
    Bu:=\rho_G\,u-  \wangle{x}^{\ell} \dive_x( V_G\nabla_x {u}).
\end{equation}
The diffusivity matrix $V_G$ is defined by
\[
V_G(x)=\int_\Rd v\otimes v\,G(x,v)\dv
\]
and the number $\ell>0$ is defined in \Cref{Hyp:asympRho}.

In most applications of the Dolbeault-Mouhot-Schmeiser method, the quantities $\rho_G$, $P_G$, $P_G^{(4)}$ generate weighted norms that are all equivalent to one another, here that is not the case, as shown in \Cref{lem:asympMomenta}. To compensate for this imbalance, we added the polynomial factor $\wx^\ell$ to our elliptic operator $B$. With this choice, we have been able to get both a perturbation term which is bounded by $\|f\|_{\sfL^2(G^{-1})}^2$ and an estimate on the dissipation of entropy. The main motivation of this correction is that, due to \Cref{lem:asympMomenta}, 
\[
\frac{P_G^{(4)}\rho_G}{P_G^2} \asymp \wangle{\cdot}^\ell
\]
and it turns out that this is actually the loss of weight we will recover in the dissipation estimate. Along the proof we repeatedly use the bounds in \Cref{lem:asympMomenta} in order to switch between $\rho_G, P_G, P^(4)_G$ appearing in the weights in various norms.

First of all we want to prove that $\sfH$ is a bounded perturbation of $\norm{\cdot}^2_{\sfL^2(G^{-1})}$, namely the scalar product in \eqref{eq:H} can be bounded by $\norm{f}^2_{\sfL^2(G^{-1})}$. By \Cref{lem:VG_PG},  the vector norm with weight $V_G$ is equivalent to the vector norm with weight $P_G$. With this definition we can bound $m_f$ by the microscopic component as follows
\begin{align*}
    \norm{m_f}^2_{\sfL^2(P_G^{-1})}&=\int_\Rd \(\int_\Rd v \, f\dv\)^2\,P_G^{-1}\dx\\
    &=\int_\Rd \( \int_\Rd v \,(1-\Pi) f\dv\)^2\,P_G^{-1}\dx\\
    &\leq \int_\Rd \(\int_\Rd\abs{ v}^2G \dv\)\(\int_\Rd \abs{(1-\Pi)f}^2\,G^{-1}\dv\) P_G^{-1}\dx\\
    &= \norm{(1-\Pi)f}^2_{\sfL^2(G^{-1})}.
\end{align*}
Moreover, 
\begin{align}\label{eq:mfMicro}
    \norm{m_f}_{\sfL^2(V_G^{-1})}&\lesssim\norm{m_f}_{\sfL^2(P_G^{-1})} \leq \norm{(1-\Pi)f}_{\sfL^2(G^{-1})}, & \norm{\rho_f}_{\sfL^2(\rho_G^{-1})}&\leq \norm{\Pi f}_{\sfL^2(G^{-1})}.
\end{align}

\begin{lem}\label{lem:Hequiv}
Let $\sfH[\cdot]$ be defined as in \eqref{eq:H}. If $\eps>0$ is small enough, then 
\begin{equation}\label{eq:equiv}
    \norm{f}^2_{\sfL^2(G^{-1})}\lesssim\sfH[f]\lesssim \norm{f}^2_{\sfL^2(G^{-1})}
\end{equation}
for any $f\in\sfL^2(G^{-1})$.
\end{lem}
\begin{proof}[{\bf Proof of \Cref{lem:Hequiv}}]
Let $g=B^{-1}\rho_f$, that is $g$ is the solution to the elliptic equation
\begin{equation*}
    \rho_Gg-\wangle{x}^{\ell} \dive_x (  V_G\nabla_x {g}) = \rho_f.
\end{equation*}
Testing with $g\,\wx^{-\ell}$ and integrating by parts we have
\begin{align*}
    \norm{g}^2_{\sfL^2\(\rho_G\wx^{-\ell}\)} +\norm{\gradx g}^2_{\sfL^2\( V_G\)} &=\bangle{g,\,\rho_f}_{\sfL^2\(\wx^{-\ell}\)} \\
    &\leq \norm{g}^2_{\sfL^2\(\rho_G\wx^{-\ell}\)} +\frac{1}{4}\norm*{\rho_f}^2_{\sfL^2\(\rho_G^{-1}\wx^{-\ell}\)}.
\end{align*}
Dropping the first term we have
\begin{equation*}
    \norm{\gradx g}_{\sfL^2\(V_G\)}\leq \frac{1}{2}\norm*{\rho_f}_{\sfL^2\(\rho_G^{-1}\wx^{-\ell}\)}\lesssim\norm{\Pi f}_{\sfL^2(G^{-1})}.
\end{equation*}
Using \eqref{eq:mfMicro}, we conclude that 
\begin{align*}
    \abs*{\bangle{ m_f(x) ,\,\nabla_x \(B^{-1}\rho_f\)}_{\sfL^2}}&\leq\norm{m_f}_{\sfL^2\(V_G^{-1}\)} \norm{\gradx g}_{\sfL^2\(V_G\)}\\
    &\lesssim \norm{(1-\Pi)f}_{\sfL^2(G^{-1})}\norm{\Pi f}_{\sfL^2(G^{-1})}\\
    &\lesssim \norm{f}_{\sfL^2(G^{-1})}.
\end{align*}
Therefore, choosing $\eps>0$ small enough, we prove \eqref{eq:equiv}. 
\end{proof}
Suppose now that $f$ is a solution to the run and tumble equation \eqref{eq:main} with initial datum $f_0$ such that $\iint_\RRd f_0\dx\dv=0$. We need to determine the dissipation of entropy 
\[
-\ddt\sfH[f]=\sfD[f],
\]
hence, we have to write the evolution equation for the corresponding $\rho_f$ and $m_f$. On the one hand, integrating equation \eqref{eq:main} with respect to $v$ we find
\begin{equation}
    \dt \rho_f+\dive_x m_f=0.
\end{equation}
On the other hand, integrating equation \eqref{eq:main} multiplied by $v$ we find
\begin{equation*}
    \dt m_f + \gradx\cdot V_f = -A_f,
\end{equation*}
where
\begin{align*}
    A_f(x)&= \int_\Rd \Lambda(x,v) v f(x,v)\dv 
\end{align*}
With $\Pi f=\frac{\rho_f}{\rho_G}G$, we see that
\begin{align*}
    V_{\Pi f}&= \int_\Rd v\otimes v\, \frac{\rho_f}{\rho_G}G\dv  = \frac{\rho_f}{\rho_G}V_G, &
    A_{\Pi f} &= \int_\Rd \Lambda \,v\,\frac{\rho_f}{\rho_G}G\dv  = \frac{\rho_f}{\rho_G}A_G,
\end{align*}
and
\begin{align*}
    \gradx\cdot V_G&=\int_\Rd v\,( v\cdot \gradx G)\dv \\
    &=\int_\Rd v\,\M(v)\(\int_\Rd \Lambda'G'\dv '\)\dv - \int_\Rd v\Lambda G\dv \\
    &=-A_G
\end{align*}
This allows us to rewrite the equation for $m_f$ as 
\begin{equation}
    \dt m_f = -\gradx\cdot (V_f-V_{\Pi f}) - V_G\gradx\(\frac{\rho_f}{\rho_G}\) - (A_f-A_{\Pi f}).
\end{equation}
We finally conclude that the dissipation of entropy is
\begin{align*}
    \sfD[f]
    =& \,- 2\bangle{\cL f,f}+ \,\eps\bangle{V_G \nabla_x \(\frac{\rho_f}{\rho_G}\),\,\nabla_x \(B^{-1}\rho_f\) }_{\sfL^2} + \eps\bangle{ \nabla_x \cdot(V_f-V_{\Pi f}),\,\nabla_x \(B^{-1}\rho_f\)}_{\sfL^2} \\
    &+ \eps\bangle{ A_f-A_{\Pi f},\,\nabla_x \(B^{-1}\rho_f\)}_{\sfL^2} +\eps\bangle{ m_f,\, \nabla_x \(B^{-1}\(\dive_x m_f\)\)}_{\sfL^2}
\end{align*}

The advantage of this decomposition is that the matrix $V_f-V_{\Pi f}$ and the vector $A_f-A_{\Pi f}$ can be controlled by the microscopic component $(1-\Pi)f$. Indeed, the euclidean matrix norm of $V_f-V_{\Pi f}$ is 
\begin{align*}
    \abs{V_f-V_{\Pi f}}^2 &\leq\int_\Rd \abs{v\otimes v}^2 \,\abs{(1-\Pi) f}^2\,\dv  \leq P_G^{(4)}\(\int_\Rd\abs{(1-\Pi)f}^2G^{-1}\,\dv \)
\end{align*}
where $P_G^{(4)}(x)= \int_\Rd \abs{v}^4 G(x,v)\dv $.
Integrating in $x$ we have
\begin{align}\label{eq:VG_Micro}
    \norm{V_f-V_{\Pi f}}_{\sfL^2(P_G^{(4){-1}})}&\leq \norm{(1-\Pi)f}_{\sfL^2(G^{-1})}.
\end{align}
Concerning $A_f-A_{\Pi f}$ we have,
\begin{align*}
    \norm{A_f-A_{\Pi f}}^2_{\sfL^2(P_G^{-1})}&=\int_\Rd \( \int_\Rd \Lambda\,v \,(1-\Pi) f\dv\)^2P_G^{-1}\dx\\
    &\lesssim \int_\Rd \(\int_\Rd\abs{v}^2G \dv\)\(\int_\Rd \abs{(1-\Pi)f}^2\,G^{-1}\dv\)P_G^{-1} \dx = \norm{(1-\Pi)f}^2_{\sfL^2(G^{-1})},
\end{align*}
in particular
\begin{equation}\label{eq:AG_Micro}
    \norm{A_f-A_{\Pi f}}^2_{\sfL^2(V_G^{-1})}\lesssim\norm{A_f-A_{\Pi f}}^2_{\sfL^2(P_G^{-1})}\lesssim \norm{(1-\Pi)f}^2_{\sfL^2(G^{-1})}.
\end{equation}

\begin{lem}\label{lem:microcoercivity}
Let $\cL$ be defined by \eqref{eq:cL}. Then
\[
\bangle{\cL f,f}\leq -\frac{1-\chi}{2}\norm{(1-\Pi)f}^2_{\sfL^2(G^{-1})}
\]
\end{lem}
\begin{proof}[{\bf Proof of \Cref{lem:microcoercivity}}]
We have
\begin{align*}
\bangle{\cL f,f}&=-\iint_{\RRd}(v\cdot\gradx f) f\,G^{-1}\dx\dv  +\iiint_{\RRd\times\Rd}\(\Lambda' \M G' \frac{ff'}{GG'}-\Lambda \M' G \frac{f^2}{G^2}\)\dv '\dx\dv \\
&=-\frac{1}{2}\iint_{\RRd}f^2\,\frac{v\cdot\gradx G }{G^2}\dx\dv + \iiint_{\RRd\times\Rd}\(\Lambda' \M G'  \frac{ff'}{GG'}-\Lambda' \M G' \(\frac{f'}{G'}\)^2\)\dv '\dx\dv \\
&=-\frac{1}{2}\iiint_{\Rd\times\Rd\times\Rd}\Lambda'\M G'\(\(\frac{f}{G}\)^2-2\frac{ff'}{GG'}+\(\frac{f'}{G'}\)^2\)\dv '\dx\dv \\
&=-\frac{1}{2}\iiint_{\Rd\times\Rd\times\Rd}\Lambda'\M G'\(\frac{f}{G}-\frac{f'}{G'}\)^2\dv '\dx\dv \\
&=-\frac{1}{4}\iiint_{\Rd\times\Rd\times\Rd}(\Lambda'\M G'+\Lambda \M' G)\(\frac{f}{G}-\frac{f'}{G'}\)^2\dv '\dx\dv \\
&\leq 0.
\end{align*}
Write $f=\Pi f+(1-\Pi)f$, and notice that
\[
\frac{f}{G}-\frac{f'}{G'}=\frac{(1-\Pi)f}{G}-\frac{(1-\Pi)f'}{G'}
\]
and $\int_{\R}(1-\Pi) f \dv =0$. Using these facts we compute
\begin{align*}
    \bangle{\mathcal{L}f,f}&\leq -\frac{1}{4}(1-\chi)\iiint_{\Rd\times\Rd\times\Rd}(\M G'+\M 'G)\left(\frac{f}{G}-\frac{f'}{G'}\right)^2\\
    &=-\frac{1}{4}(1-\chi)\iiint_{\Rd\times\Rd\times\Rd}(\M G'+\M 'G)\left(\frac{(1-\Pi)f}{G}-\frac{(1-\Pi)f'}{G'}\right)^2\\
    &=-\frac{1}{4}(1-\chi)\iiint_{\Rd\times\Rd\times\Rd}(\M G'+\M 'G)\left(\(\frac{(1-\Pi)f}{G}\)^2+\(\frac{(1-\Pi)f'}{G'}\)^2\right)\\
    &=-\frac{1}{2}(1-\chi)\iint_{\Rd\times\Rd}\left(\M \frac{\rho_G}{G}+1\right)\frac{\abs{(1-\Pi)f}^2}{G}\\
    &\leq- \frac{1}{2}(1-\chi)\norm{(1-\Pi)f}_{\sfL^2(G^{-1})}^2.
\end{align*}
\end{proof}

\subsection{An estimate on the entropy production}
Now we study the decay of $\sfH$ and we have the following Lemma
\begin{lem}\label{lem:disspation}
There exists $\kappa>0$ such that
\begin{align*}
    \ddt\sfH[f]\leq -\kappa (\norm{(1-\Pi)f}^2_{\sfL^2(G^{-1})} + I_1),
\end{align*}
where 
\begin{equation}\label{eq:I_1}
    I_1= \norm{\nabla_x  u}^2_{\sfL^2\( V_G \)} + \norm{\dive_x \( V_G\nabla_x  u\)}^2_{\sfL^2\(\rho_G^{-1}\wangle{x}^{-\ell}\)}
\end{equation}
and $u=B^{-1}\rho_f$. 
\end{lem}
\begin{proof}[{ \bf Proof of Lemma \ref{lem:disspation}}]
Using the microscopic coercivity we have
\begin{align*}
    \ddt \sfH[f] \leq -(1-\chi)\norm{(1-\Pi)f}^2_{\sfL^2(G^{-1})} - \eps \sfR[f]
\end{align*}
where 
\begin{align*}
    \sfR[f] 
    =& \,\bangle{V_G \nabla_x \(\frac{\rho_f}{\rho_G}\),\,\nabla_x \(B^{-1}\rho_f\) }_{\sfL^2} + \bangle{ \nabla_x \cdot(V_f-V_{\Pi f}),\,\nabla_x \(B^{-1}\rho_f\)}_{\sfL^2} \\
    &+ \bangle{ A_f-A_{\Pi f},\,\nabla_x \(B^{-1}\rho_f\)}_{\sfL^2} +\bangle{ m_f,\, \nabla_x \(B^{-1}\(\dive_x m_f\)\)}_{\sfL^2}\\
    =&:I_1+I_2+I_3+I_4.
\end{align*}

$\#$\textit{Step 1: Representation of the macroscopic term $I_1$.}
Let $u=u(x)$ be defined by $u=B^{-1}\rho_f$, that is the solution of the elliptic equation
\begin{equation}\label{u}
    \rho_G\,u- \wangle{x}^{\ell} \dive_x (  V_G\nabla_x {u}) = \rho_f.
\end{equation}
Then after integration by part we have
\begin{align*}
    I_1&=-\bangle{ \frac{\rho_f}{\rho_G} ,\,\dive_x \(V_G\nabla_x \(B^{-1}\rho_f\)\)}_{\sfL^2}\\
    &=-\bangle{u- \rho_G^{-1}\,\wangle{x}^{\ell} \dive_x (V_G\nabla_x {u})\,,\, \dive_x\(V_G\nabla_x  u\)} =\norm{\nabla_x  u}^2_{\sfL^2\( V_G \)} + \norm{\dive_x\(V_G\nabla_x  u\)}^2 _{\sfL^2\(\rho_G^{-1}\wx^\ell\)}.
\end{align*}

$\#$\textit{Step 2: Bound on the error terms.} 
Here we have to estimate $I_2$, $I_3$, $I_4$. Concerning $I_2$, we integrate by parts and use \eqref{eq:VG_Micro}
\begin{align*}
    \abs{I_2}&=\abs{\bangle{V_f-V_{\Pi f},\,\nabla_x^2   u}_{\sfL^2}} \leq \norm{V_f-V_{\Pi f}}_{\sfL^2(P_G^{(4)-1})} \norm{\nabla_x^2 u}_{\sfL^2(P_G^{(4)})} \leq \norm{(1-\Pi)f}_{\sfL^2(G^{-1})} \norm{\nabla_x^2 u}_{\sfL^2(P_G^{(4)})},
\end{align*} 
so we need to bound the norm of Hessian of $u$. Rewrite the equation $u=B^{-1}\rho_f$ as
\[
\rho_G\wx^{-\ell}\, u - \dive(V_G\gradx u)=\rho_G\wx^{-\ell} u_f
\]
where $u_f=\tfrac{\rho_f}{\rho_G}$. Multiplying by $u-u_f$ and integrating we have
\begin{align*}
    \norm{u-u_f}^2_{\sfL^2(\rho_G\wx^{-\ell})} &= \bangle{ u-u_f\,,\,\dive(V_G\gradx u)}_{\sfL^2} \leq \norm{u-u_f}_{\sfL^2(\rho_G\wx^{-\ell})}\norm{\dive(V_G\gradx u)}_{\sfL^2(\rho_G^{-1}\wx^\ell)},
\end{align*}
that is
\begin{equation}\label{eq:u-uf}
    \norm{u-u_f}_{\sfL^2(\rho_G\wx^{-\ell})}\leq \norm{\dive(V_G\gradx u)}_{\sfL^2(\rho_G^{-1}\wx^\ell)}\leq I_1^{\tfrac{1}{2}}.
\end{equation}
On the other hand, multiply by $\dive(\wx^{{\frac{2}{1+\gamma}}}\gradx u)$ we get
\begin{align*}
    \rho_G\wx^{-\ell}(u-u_f) \dive(\wx^{{\frac{2}{1+\gamma}}}\gradx u)=\dive(V_G\gradx u) \dive(\wx^{{\frac{2}{1+\gamma}}}\gradx u).
\end{align*}
The choice of the factor $\wx^{{\frac{2}{1+\gamma}}}$ is motivated by the fact that this is exactly the gain of polynomial weight of $P_G^{(4)}$ with respect to $P_G$, as shown in \eqref{eq:weightsAsymp}. After integrating by parts to switch divergences to gradients, we have
\begin{align*}
    \int_\Rd \rho_G\wx^{-\ell}&(u-u_f) \dive(\wx^{{\frac{2}{1+\gamma}}}\gradx u) \dx = \int_\Rd \gradx(V_G\gradx u)\colon \gradx(\wx^{{\frac{2}{1+\gamma}}}\gradx u) \\
    &=\int_\Rd (\nabla_x^2u\,\colon V_G\,\nabla_x^2u)\wx^{\frac{2}{1+\gamma}} \dx + \int_\Rd ((\gradx(\wx^{\frac{2}{1+\gamma}}) \otimes \gradx u)\colon V_G\nabla_x^2 u) \,\dx \\
    & + \int_\Rd ((\gradx V_G\,\cdot \gradx u)\colon\nabla_x^2 u)  \wx^{\frac{2}{1+\gamma}} \dx+ \int_\Rd ((\gradx V_G\cdot\gradx u)\colon(\gradx (\wx^{\frac{2}{1+\gamma}}) \otimes \gradx u) )  \dx
\end{align*}
where $\sfA\colon\mathsf{B}:=\sum_{i,j}\sfA_{ij}\mathsf{B}_{ij}$. Reordering the terms we get
\begin{align*}
    \norm{\nabla_x^2 u}^2_{\sfL^2(V_G\wx^{\frac{2}{1+\gamma}})} =&\,\int_\Rd (\nabla_x^2u\,\colon V_G\,\nabla_x^2u)\wx^{\frac{2}{1+\gamma}} \dx\\
    =&\, \int_\Rd (u-u_f)\Delta_x u \,\rho_G\wx^{{\frac{2}{1+\gamma}}-\ell}\dx + \int_\Rd (u-u_f) \gradx u\cdot \gradx (\wx^{\frac{2}{1+\gamma}})\rho_G\wx^{-\ell}\dx \\
    &\,- \int_\Rd ((\gradx(\wx^{\frac{2}{1+\gamma}}) \otimes \gradx u)\colon V_G\nabla_x^2 u) \,\dx - \int_\Rd ((\gradx V_G\,\cdot \gradx u)\colon\nabla_x^2 u)  \wx^{\frac{2}{1+\gamma}} \dx\\
    &\,- \int_\Rd ((\gradx V_G\cdot\gradx u)\colon(\gradx (\wx^{\frac{2}{1+\gamma}}) \otimes \gradx u) )  \dx=:I_{2,1}+I_{2,2}+I_{2,3}+I_{2,4}+I_{2,5}
\end{align*}
We estimate the first two terms using the Cauchy-Schwarz inequality, \eqref{eq:u-uf} and the bounds \eqref{eq:weightsAsymp}, obtaining
\begin{align*}
    \abs{I_{2,1}} &\leq \norm{u-u_f}_{\sfL^2(\rho_G\wx^{-\ell})}\norm{\Delta u}_{\sfL^2(\rho_G\wx^{{\frac{4}{1+\gamma}}-\ell})} \lesssim I_1^{\tfrac{1}{2}} \norm{\nabla_x^2 u}_{\sfL^2(P_G\wx^{\frac{2}{1+\gamma}})}
\end{align*}
and
\begin{align*}
    \abs{I_{2,2}}&\leq \norm{u-u_f}_{\sfL^2(\rho_G\wx^{-\ell})}\norm{\gradx u}_{\sfL^2(\rho_G\wx^{{\frac{4}{1+\gamma}}-2-\ell})} \lesssim I_1^{\tfrac{1}{2}} \norm{\gradx u}_{\sfL^2(V_G\wx^{{\frac{2}{1+\gamma}}-2})} \leq I_1.
\end{align*}
For the last terms we use the Cauchy-Schwarz inequality
\begin{align*}
    \abs{I_{2,3}}&\leq \norm{\gradx(\wx^{\frac{2}{1+\gamma}})\otimes\gradx u}_{\sfL^2(V_G\wx^{-{\frac{2}{1+\gamma}}})}\norm{\nabla_x^2 u}_{\sfL^2(V_G\wx^{\frac{2}{1+\gamma}})} \\
    &\leq \norm{\gradx u}_{\sfL^2(V_G\wx^{{\frac{2}{1+\gamma}}-2})}\norm{\nabla_x^2 u}_{\sfL^2(V_G\wx^{\frac{2}{1+\gamma}})} \leq I_1^{\frac{1}{2}}\norm{\nabla_x^2 u}_{\sfL^2(V_G\wx^{\frac{2}{1+\gamma}})} ,
\end{align*}
\begin{align*}
    \abs{I_{2,4}}&\leq \norm{\gradx V_G\cdot\gradx u}_{\sfL^2(V_G^{-1}\wx^{{\frac{2}{1+\gamma}}})} \norm{\nabla_x^2 u}_{\sfL^2(V_G\wx^{\frac{2}{1+\gamma}})}\\
    &\leq \norm{\gradx u}_{\sfL^2(V_G^{-1}\abs{\gradx V_G}^2\wx^{{\frac{2}{1+\gamma}}})} \norm{\nabla_x^2 u}_{\sfL^2(V_G\wx^{\frac{2}{1+\gamma}})}
\end{align*}
and
\begin{align*}
    \abs{I_{2,5}}&\leq\norm{\gradx V_G\cdot\gradx u}_{\sfL^2(V_G^{-1}\wx^{{\frac{2}{1+\gamma}}})}\norm{\gradx(\wx^{\frac{2}{1+\gamma}})\otimes\gradx u}_{\sfL^2(V_G\wx^{-{\frac{2}{1+\gamma}}})}\\ 
    &\leq \norm{\gradx u}_{\sfL^2(V_G^{-1}\abs{\gradx V_G}^2\wx^{{\frac{2}{1+\gamma}}})} \norm{\gradx u}_{\sfL^2(V_G\wx^{{\frac{2}{1+\gamma}}-2})}
\end{align*}
By \Cref{Hyp:asympRho} we have
\begin{align*}
    \abs{V_G^{-2}}\abs{\gradx V_G}^2\wx^{\frac{2}{1+\gamma}}\lesssim 1,
\end{align*}
so we have
\begin{align*}
    \abs{I_{2,4}}&\lesssim \norm{\nabla_x^2 u}_{\sfL^2(V_G\wx^{\frac{2}{1+\gamma}})} I_1^{\frac{1}{2}}& \abs{I_{2,5}}&\lesssim\norm{\nabla_x^2 u}_{\sfL^2(V_G\wx^{\frac{2}{1+\gamma}})} I_1^{\frac{1}{2}}
\end{align*}
We conculde
\begin{align*}
    \norm{\nabla_x^2 u}^2_{\sfL^2(V_G\wx^{\frac{2}{1+\gamma}})}\lesssim \norm{\nabla_x^2 u}_{\sfL^2(V_G\wx^{\frac{2}{1+\gamma}})} I_1^{\frac{1}{2}} + I_{1}, 
\end{align*}
that implies
\[
\norm{\nabla_x^2 u}_{\sfL^2(V_G\wx^{\frac{2}{1+\gamma}})}\lesssim I_1^{\frac{1}{2}}.
\]
According to \eqref{eq:weightsAsymp} and \Cref{lem:VG_PG}, the weights $V_G\wx^{\frac{2}{1+\gamma}}$ and $P_G^{(k)}$ generate equivalent norms, so we conclude
\[
\abs{I_2}\leq \norm{(1-\Pi)f}_{\sfL^2(G^{-1})} I_1^{\frac{1}{2}}.
\]

Concerning $I_3$ we have
\begin{align*}
    \abs{I_3} &= \abs{\bangle{ A_f-A_{\Pi f}\,,\,\nabla_x  u \, \,}_{\sfL^2}}\\
    &\leq  \norm{A_f-A_{\Pi f}}_{\sfL^2(V_G^{-1})}\norm{\nabla_x  u}_{\sfL^2(V_G)}\\
    &\lesssim \norm{(1-\Pi)f}_{\sfL^2(G^{-1})}\, \norm{\nabla_x  u}_{\sfL^2(V_G)} \lesssim  \norm{(1-\Pi)f}_{\sfL^2(G^{-1})}\,I_1^{1/2}.
\end{align*}
Consider now 
\[
I_4 = \bangle{ m_f,\, \nabla_x \(B^{-1}(\dive_x  m_f)\) }_{\sfL^2}.
\]
Define $h=B^{-1}(\dive_x m_f)$, that is
\begin{equation*}
    \rho_G\,h- \wangle{x}^{\ell} \dive_x ( V_G\nabla_x {h}) = \dive_x  m_f.
\end{equation*}
Test with $h \wangle{x}^{-\ell} $ and integrate by parts
\begin{align*}
    \norm{h}^2_{\sfL^2(\rho_G\wx^{-\ell})}+ \norm{\nabla_x  h}^2_{\sfL^2(V_G)}  =&\, \bangle{ \dive_x m_f,\,h }_{\sfL^2\( \wx^{-\ell} \)}\\
    =&-\bangle{ m_f ,\,\nabla_x  h\, }_{\sfL^2\(\wx^{-\ell}\)} - \bangle{ m_f,\,\nabla_x\(\wx^{-\ell}\) h}_{\sfL^2}\\
    =&\,I_{4,1} + I_{4,2}.
\end{align*}
We bound the first term as 
\begin{align*}
    \abs{I_{4,1}}&\leq \frac{1}{2} \norm{m_f}^2_{\sfL^2( V_G^{-1} \wx^{-2\ell}) } +\frac{1}{2} \norm{\nabla_x  h}^2_{\sfL^2(V_G)}
\end{align*}
and the second as
\[
\abs{I_{4,2}}\leq \frac{1}{4}  \norm{m_f \cdot\nabla_x \( \wangle{x}^{-\ell}\) }^2_{\sfL^2(\rho_G^{-1}\wx^\ell)} +\norm{ h}^2_{\sfL^2(\rho_G\wx^{-\ell})}.
\]
Therefore 
\begin{align*}
    \norm{\nabla_x  h}^2_{\sfL^2(V_G)} &\leq  \norm{m_f}^2_{\sfL^2(V_G^{-1}\wx^{-2\ell}) } + \frac{1}{2} \norm{m_f \cdot\nabla_x \( \wangle{x}^{-\ell}\) }^2_{\sfL^2(\rho_G^{-1}\wx^\ell)}\\
    &\lesssim\norm{m_f}^2_{\sfL^2(V_G^{-1}) }+\norm{m_f }^2_{\sfL^2(P_G^{-1}\wx^{{\frac{2}{1+\gamma}}-2\ell-2})}\\
    &\lesssim \norm{m_f}^2_{\sfL^2(V_G^{-1}) } \lesssim \norm{(1-\Pi)f}^2_{\sfL^2(G^{-1})}.
\end{align*}
Finally
\begin{align*}
    \abs{I_4}&\leq \norm{m_f}_{\sfL^2(V_G^{-1})}\norm{\gradx h}_{\sfL^2(V_G)} \lesssim \norm{(1-\Pi)f}^2_{\sfL^2(G^{-1})}.
\end{align*}

$\#$\textit{Step 3: Conclusion}
Collecting the estimates of the previous paragraphs we have that, if $\eps>0$ is small enough, there exists $\kappa>0$ such that
\begin{align*}
    \ddt\sfH[f]&\leq -\((1-\chi)-C\eps\)\norm{(1-\Pi)f}^2_{\sfL^2(G^{-1})} + C\eps I_1^{1/2}\norm{(1-\Pi)f}_{\sfL^2(G^{-1})}- \eps I_1\\
    &\leq -\kappa (\norm{(1-\Pi)f}^2_{\sfL^2(G^{-1})} + I_1).
\end{align*}

\end{proof}

\begin{lem}\label{lem:1}
We have the following bound
\begin{equation*}
    \norm{\rho_f}^2_{\sfL^2(\rho_G\wx^{-\ell})} \dx \lesssim I_1.
\end{equation*}
As a consequence, we have the following estimate on the entropy production
\begin{equation}
    \ddt\sfH[f]\lesssim - \(\norm{(1-\Pi)f}^2_{\sfL^2(G^{-1})} + \norm{\Pi f}^2_{\sfL^2(G^{-1}\wx^{-\ell})}\).
\end{equation}

\end{lem}
\begin{proof}[{\bf Proof of \Cref{lem:1}}]
We square the equation \eqref{u} to write 
\begin{align}\label{eq:1654}
    \norm{\rho_f}^2_{\sfL^2(\rho_G^{-1}\wx^{-\ell})}  &= \norm{u}^2_{\sfL^2(\rho_G\wx^{-\ell})}  +2\norm{\nabla_x  u}^2_{\sfL^2(V_G )}  + \norm{\dive_x (V_G\nabla_x  u)}^2_{\sfL^2( \rho_G^{-1}\wx^{\ell})}\\ &\leq \norm{u}^2_{\sfL^2(\rho_G\wx^{-\ell})} +2I_1.
\end{align}
The goal is to also bound the norm of $u$ on the right hand side by $I_1$. For this purpose we will use the weighted Poincaré inequality \eqref{eq:Poincaré} with non classical average $\Bar{u}=\int_\Rd u\rho_G\dx$, that is  
\begin{equation*}
    \norm{u-\Bar{u}}^2_{\sfL^2(P_G \wx^{-{\frac{2}{1+\gamma}}})}\leq C_P \norm{\nabla_x u}^2_{\sfL^2( P_G)}.
\end{equation*}
This inequality implies in particular
\begin{equation}
    \norm{u-\Bar{u}}^2_{\sfL^2(\rho_G \wx^{-\ell})}\lesssim  \norm{\nabla_x u}^2_{\sfL^2( V_G)}.
\end{equation}
Since we are considering functions $f$ with mass zero, we can integrate \eqref{u} and obtain
\begin{align}\label{eq:2511}
\begin{split}
    \Bar{u}&= \int_\Rd u\,\rho_G\dx = \int_\Rd \dive_x (V_G\nabla_x u)\wangle{x}^{\ell}\dx\\
    &\leq \norm{\dive_x(V_G\nabla_x u)}_{\sfL^2(\rho_G^{-1}\wangle{x}^{\ell} )} \(\int_\Rd \rho_G\wangle{x}^{\ell}\dx \)^{1/2} \lesssim I_1^{1/2}
\end{split}
\end{align}
Therefore we have
\begin{align*}
    \norm{\rho_f}^2_{\sfL^2(\rho_G^{-1}\wx^{-\ell})}  &\leq \norm{u}^2_{\sfL^2(\rho_G\wx^{-\ell})} +2I_1 \leq 2\norm{u-\Bar{u}}^2_{\sfL^2(\rho_G\wx^{-\ell})} +2 \norm{\Bar{u}}^2_{\sfL^2(\rho_G\wx^{-\ell})}+2I_1\\
    &\lesssim \norm{\nabla_x u}^2_{\sfL^2( P_G)} + I_1 \lesssim I_1
\end{align*}

\end{proof}

\subsection{Weighted $\sfL^2$ estimates}
The goal of this section is the following Proposition. 
\begin{prop}\label{MomentsProp}
Let $m$ be defined as in \Cref{Prop:Lyapunov} or in \Cref{Prop:PolyLyap} and let $f_0\in \sfL^2(m\,G^{-1}\dx\dv )$. Then 
\begin{equation}
    \iint_{\RRd}\abs{S_\cL(t)f_0}^2 m G^{-1}\dx\dv \lesssim  \iint_{\RRd}\abs{f_0}^2 m G^{-1}\dx\dv .
\end{equation}
for any $t\geq0$. 
\end{prop}

We prove this result following the idea of \cite[Sec. 3]{BDL22}, that is based on the Stein-Weiss interpolation theorem \cite[Thm. 2]{S56}, see also \cite{SW58,BL76}.
\begin{thm}[Stein-Weiss]
Assume that $1\leq p_0,\, p_1,\, q_1,\, q_1\leq \infty$ and that 
\begin{align*}
    T&\colon L^{p_0}(w_0)\to L^{q_0}(\widetilde{w}_0) & T&\colon L^{p_1}(w_1)\to L^{q_1}(\widetilde{w}_1)
\end{align*}
with norms $M_0$ and $M_1$ respectively. Then
\[
T\colon L^{p}(w)\to L^{q}(\widetilde{w})
\]
with norm $M\leq M_0^{1-\theta}M_1^{\theta}$ where
\begin{align*}
    \frac1p&=\frac{1-\theta}{p_0}+\frac{\theta}{p_1} & \frac1q&=\frac{1-\theta}{q_0}+\frac{\theta}{q_1}\\
    w&=w_0^{1-\theta} w_1^{\theta} & \widetilde{w}&=\widetilde{w}_0^{1-\theta} \widetilde{w}_1^{\theta}.
\end{align*}
\end{thm}

We need to prove the boundedness of $S_\cL(t)$ in  $\sfL^{\infty}(G^{-1}\dx\dv )$ and in  $\sfL^1(m\,\dx\dv )$.

\subsubsection{Boundedness in $\sfL^{\infty}(G^{-1}\dx\dv )$}
\begin{lem}\label{lem:LinftyBound}
Let $S_\cL$ be the semigroup associated to the run and tumble equation and $G$ the steady state. Then 
\begin{equation}
    \norm{S_\cL(t)}_{\sfL^\infty(G^{-1}\dx\dv ) \to \sfL^\infty(G^{-1}\dx\dv )}\leq 1
\end{equation}
for all $t\geq 0$ .
\end{lem}
\begin{proof}[{\bf Proof of \Cref{lem:LinftyBound}}]
Let $f_0 \in \sfL^{\infty}(G^{-1}\dx\dv )$, so that
\[
f_0(x,v)\leq \norm{G^{-1}f_0}_{\sfL^\infty(\dx\dv )}G(x,v) \qquad \text{for every } x,v\in\Rd.
\]
Then the function $\norm{G^{-1}f_0}_{\sfL^\infty(\dx\dv )}\,G - f_0$ is positive, and by the positivity of $S_\cL(t)$ we have
\begin{align*}
    \norm{G^{-1}f_0}_{\sfL^\infty(\dx\dv )}\,G - S_\cL(t)f_0 = S_\cL(t)(\norm{G^{-1}f_0}_{\sfL^\infty(\dx\dv )}\,G - f_0) \geq 0.
\end{align*}
This is nothing but
\[
(S_\cL(t)f_0)(x,v) \leq \norm{G^{-1}f_0}_{\sfL^\infty(\dx\dv )}\,G(x,v) \qquad \text{for every } x,v\in\Rd,
\]
that is exactly
\[
\norm{G^{-1}S_\cL(t)f_0}_{\sfL^\infty(\dx\dv )}\leq \norm{G^{-1}f_0}_{\sfL^\infty(\dx\dv )}. \qedhere
\]

\end{proof}

\subsubsection{Boundedness in $\sfL^1(m\,\dx\dv )$}

\begin{lem}\label{lem:L1Bound}
Let $m$ be defined as in \Cref{Prop:Lyapunov} or in \Cref{Prop:PolyLyap}. Then 
\begin{equation}
    \norm{S_\cL(t)}_{\sfL^1(m) \to \sfL^1(m)}\lesssim 1
\end{equation}
for all $t\geq 0$ .
\end{lem}
\begin{proof}[{\bf Proof of \Cref{lem:L1Bound}}]
We split the operator $\cL=\cA + \cB$ and we consider the Duhamel formula
\[
S_\cL=S_\cB + S_\cB\star\cA S_\cL.
\]
Suppose that $m\asymp e^{\nu\wx^a} +e^{b\abs{v}^\gamma}$ is defined as in \Cref{Prop:Lyapunov}. Then, by \Cref{lem:dissipSB}, \Cref{lem:SB decay}, \Cref{lem:A} and the mass conservation property, we have
\begin{align*}
    \norm{S_\cL}_{\sfL^1(m)\to\sfL^1(m)}&\leq \norm{S_\cB}_{\sfL^1(m)\to\sfL^1(m)} + \norm{S_\cB}_{\sfL^1(\omega^\ell)\to\sfL^1(m)}\star\norm{\cA}_{\sfL^1\to\sfL^1(\omega^\ell)} \norm{S_\cL}_{\sfL^1(m)\to\sfL^1}\\
    &\lesssim 1+ \int_0^\infty e^{-\lambda_\ell t^a}dt \lesssim 1.
\end{align*}

Suppose now that $m\asymp \wx^k+\wangle{v}^{2k}$ is defined as in \Cref{Prop:PolyLyap}. From \Cref{lem:SB Polydecay} we have
\begin{align*}
    \norm{S_\cL}_{\sfL^1(m)\to\sfL^1(m)}&\lesssim \norm{S_\cB}_{\sfL^1(m)\to\sfL^1(m)}+  \norm{S_\cL}_{\sfL^1(m^\ell)\to\sfL^1(m)} \star \norm{\cA }_{\sfL^1\to\sfL^1(m^\ell)}\norm{S_\cL}_{\sfL^1(m)\to\sfL^1}\d s\\
    &\lesssim 1+\int_0^\infty \wangle{t}^{-k(\ell-1)}dt \lesssim 1
\end{align*}
by choosing $\ell>1+\frac{1}{k}$.
\end{proof}

\begin{proof}[\bf Proof of Proposition \ref{MomentsProp}]
Consider the semigroup $S(t)$ defined by
\[
(S(t)h)(x,v):=(G^{-1}S_\cL(t)(G\,h))(x,v),
\]
that is the composition of the multiplication by $G$, the semigroup $S_\cL$ and the multiplication by $G^{-1}$. Thanks to Lemma \ref{lem:L1Bound} we have
\begin{align*}
    \norm{S(t)h}_{\sfL^1(m\,G )} =\norm{G^{-1}S_\cL(t)(G\,h)}_{\sfL^1(m\,G )} =\norm{S_\cL(t)(G\,h)}_{\sfL^1(m )} \lesssim\norm{G\,h}_{\sfL^1(m )} =\norm{h}_{\sfL^1(m\,G )},
\end{align*}
namely
\begin{equation*}
    \norm{S(t)}_{\sfL^1(m\,G )\to \sfL^1(m\,G )}\lesssim1.
\end{equation*}
On the other hand, thanks to Lemma \ref{lem:LinftyBound} we have
\begin{align*}
    \norm{S(t)h}_{\sfL^\infty  }&=\norm{G^{-1}S_\cL(t)(G\,h)}_{\sfL^\infty  } =\norm{S_\cL(t)(G\,h)}_{\sfL^\infty(G^{-1})} \leq\norm{G\,h}_{\sfL^\infty(G^{-1})} =\norm{h}_{\sfL^\infty  },
\end{align*}
that is 
\begin{equation*}
    \norm{S(t)}_{\sfL^\infty  \to \sfL^\infty}\leq 1.
\end{equation*}
The Stein-Weiss interpolation theorem implies that
\begin{equation*}
    \norm{S(t)}_{\sfL^2(m\,G )\to \sfL^2(m\,G )}\lesssim1.
\end{equation*}
Now we just have to go back to $S_\cL$ as follows
\begin{align*}
    \norm{S_\cL(t)f}_{\sfL^2(m\,G^{-1} )} &= \norm{G\,S(t)(G^{-1}\,f)}_{\sfL^2(m\,G^{-1} )} =\norm{S(t)(G^{-1}\,f)}_{\sfL^2(m\,G )} \lesssim\norm{G^{-1}\,f}_{\sfL^2(m\,G )} =\norm{f}_{\sfL^2(m\,G^{-1} )}
\end{align*}
and this is the conclusion.

\end{proof}

\begin{proof}[\bf Proof of \Cref{thm:DMS}]

Consider first $m\asymp e^{\nu\wx^a} + e^{b\abs{v}^\gamma}$ as defined in \Cref{Prop:Lyapunov}. Then for any $\varrho>0$ we have
\begin{align*}
    \ddt\sfH[f] &\lesssim - \(\norm{(1-\Pi)f}_{\sfL^2(G^{-1})}^2 + \int_\Rd \abs{\Pi f}^2\wx^{-\ell}G^{-1}\dx\dv\)\\
    &\leq -\iint_{\Rd\times\Rd} \abs{f}^2\, \wx^{-\ell}\,G^{-1}\dx\dv  \\
    &\leq -\iint_{\abs{x}\leq \varrho} \abs{f}^2\, \wx^{-\ell}\,G^{-1}\dx\dv  \\
    &\leq -\wangle{\varrho}^{-\ell}\iint_{\abs{x}\leq \varrho} \abs{f}^2\, \,G^{-1}\dx\dv  \\
    &\leq -\wangle{\varrho}^{-\ell}\iint_{\Rd\times\Rd} \abs{f}^2\, \,G^{-1}\dx\dv  +\wangle{\varrho}^{-\ell}\iint_{\abs{x}\geq\varrho} \abs{f}^2\, \,G^{-1}\dx\dv   \\
    &\lesssim- \wangle{\varrho}^{-\ell}\sfH[f] +\wangle{\varrho}^{-\ell}\max_{\abs{x}\geq \varrho}\{m^{-1}\}\iint_{\abs{x}\geq\varrho} \abs{f}^2\,m\, \,G^{-1}\dx\dv  \\
    &\lesssim- \wangle{\varrho}^{-\ell}\sfH[f] +\wangle{\varrho}^{-\ell}e^{-\nu\wangle{\varrho}^a}\norm{f_0}_{\sfL^2(m\,G^{-1})}^2,
\end{align*}
where in the last step we used \Cref{MomentsProp}. By the Gronwall lemma we deduce
\begin{align*}
    \sfH[f]\leq e^{-\nu\wangle{\varrho}^a}\norm{f_0}_{\sfL^2(m\,G^{-1})}^2 + e^{-\lambda_1 \wangle{\varrho}^{-\ell}t} \sfH[f_0]
\end{align*}
for some constant $\lambda_1>0$. Choosing $\varrho>0$ such that $\lambda\wangle{\varrho}^{-\ell}  t=\nu \wangle{\varrho}^{a}$ we have the decay
\begin{equation*}
    \sfH[f]\lesssim e^{-\lambda t^{\tfrac{a}{a+\ell}}}\norm{f_0}^2_{\sfL^2(m\,G^{-1})}.
\end{equation*}

Consider now $m\asymp \wx^k+\wangle{v}^{2k}$ as defined in \Cref{Prop:PolyLyap} and suppose $f_0\in\sfL^2(m\,G^{-1})$. By  H\"older inequality
\begin{align*}
    \norm{\Pi f}^2_{\sfL^2(G^{-1})} = \int_\Rd \frac{\rho_f^2}{\rho_G}\dx\leq \(\int_\Rd \frac{\rho_f^2}{\rho_G}\wangle{x}^{-\ell} \dx\)^\eta\(\int_\Rd \frac{\rho_f^2}{\rho_G}\wx^k\dx\)^{1-\eta}
\end{align*}
where $\eta= \frac{k}{k+\ell}$. Thanks to \Cref{MomentsProp}, we have 
\begin{align*}
    \int_\Rd \frac{\rho_f^2}{\rho_G}\wx^k\dx&=\int_\Rd\(\int_\Rd f\dv \)^2\frac{\wx^k}{\rho_G}\dx \leq \int_\Rd\(\int_\Rd f^2 G^{-1}\dv \)\(\int_\Rd G\dv  \)\frac{\wx^k}{\rho_G}\dx\\
    &\leq\int_{\Rd\times\Rd}f^2(\wx^k+\wangle{v}^{2k})G^{-1}\dx\dv \\
    &\lesssim\norm{f_0}_{\sfL^2(m\,G^{-1})}^2
\end{align*}
Therefore
\begin{equation}
    \norm{\Pi f}^2_{\sfL^2(G^{-1})} \lesssim  \(\int_\Rd \frac{\rho_f^2}{\rho_G}\wangle{x}^{-\ell} \dx\)^\eta \norm{f_0}_{\sfL^2(m\,G^{-1})}^{2(1-\eta)}
\end{equation}
and we can conclude that
\begin{align*}
    \ddt\sfH[f] &\lesssim - \(\norm{(1-\Pi)f}^2_{\sfL^2(G^{-1})} + \int_\Rd \frac{\rho_f^2}{\rho_G}\wangle{x}^{-\ell} \dx\) \\
    &\lesssim - \(\norm{(1-\Pi)f}^2_{\sfL^2(G^{-1})} + \norm{f_0}_{\sfL^2(m\,G^{-1})}^{-2\frac{1-\eta}{\eta}}\norm{\Pi f}^{2/\eta}_{\sfL^2(G^{-1})}\) \lesssim - \norm{f_0}_{\sfL^2(m\,G^{-1})}^{-2\frac{1-\eta}{\eta}}\sfH[f]^{1/\eta}.
\end{align*}
By applying the Gronwall Lemma we finally have
\begin{equation*}
    \sfH[f]\lesssim \(\sfH[f_0]^{-\frac{1-\eta}{\eta}}+\norm{f_0}_{\sfL^2(m\,G^{-1})}^{-2\frac{1-\eta}{\eta}} \,t \)^{-\frac{\eta}{1-\eta}} \lesssim \frac{1}{\wangle{t}^{\frac{\eta}{1-\eta}}}\norm{f_0}_{\sfL^2(m\,G^{-1})}^{2}=\frac{1}{\wangle{t}^{\frac{k}{\ell}}}\norm{f_0}_{\sfL^2(m\,G^{-1})}^{2} .
\end{equation*}

\end{proof}

%%%%%%%%%%%%%%%%%%%%%%%%%%%%%%%%%%%%%%%%%%%%%%%%%%%%%%%%%%%%%%%%%%%%%%%%%%%%%%%%%%%%%%%%%%%%%%%%%%%%%%%

\section{Acknowledgments}

The authors thank warmly Jean Dolbeault and Stéphane Mischler for stimulating discussions.

J. Evans is supported by a Royal Society University Research Fellowship R1\_251808 (since October 2025) and before this a Leverhulme early career fellowship ECF-2021-134. For the purpose of open access, the author has applied a Creative Commons Attribution (CC-BY) licence to any Author Accepted Manuscript version arising from this submission.

L. Ziviani has
received funding from the European Union’s Horizon 2020 research and innovation programme
under the Marie Skłodowska-Curie grant agreement No 945332.
%%%%%%%%%%%%%%%%%%%%%%%%%%%%%%%%%%%%%%%%%%%%%%%%%%%%%%%%%%%%%%%%%%%%%%%%%%%%%%%%%%%%%%%%%%%%%%%%%%%%%%%

\appendix
\section{Reminder on Harris' theorem}\label{App:Harris}

Harris's type theorems are one of the main tool in the Theory of Markov processes and PDEs. These methods provides existence of a stationary measure and convergence to it under a Lyapunov condition and a minorisation condition, that we are going to recall below. 

The first works concerning stability of Markov processes go back to Doeblin \cite{Doeblin1940} and Harris \cite{H56}. The former is presenting what in many articles \cite{CM21, EY23} is now called as the Doeblin theorem, and is showing exponential convergence for Markov processes whose transition probabilities possess a uniform lower bound. The latter gives some sufficient conditions for the existence of a stationary measure which is unique up to a multiplicative constant, but it is not providing any convergence rate. In more recent years, these results have been reworked and extended to then be used to obtain qualitative convergence rates \cite{HM11,DMT95,MT09,MT93,CM21}. See also \cite{H16} for a proof for Doeblin's theorem and Harris' theorem. The development of these tools was encouraged by their applications in the field of PDEs, as they can complement hypocoercivity methods for kinetic equations. Some examples can be found in Hu and Wang \cite{HW19}, Eberle et al. \cite{EGZ16}, Canizo et al. \cite{CCEY20}, Cao \cite{C21} and Laflèche \cite{L20}.

The Harris theorem is working in the space of measures, for this reason we have to look at the semigroup $S_\cL$ defined in the space of probability measures $\mathcal{P}(\R^d\times\R^d)$, in other words for any probability measure $\mu\in \mathcal{P}(\R^d\times\R^d) $ we define $S_\cL(t)\mu$ as the weak solution to the run and tumble equation with initial data $\mu$. 

In this work, we are going to use the version of Harris' Theorem proposed in \cite[Thm 5.6]{CM21}, so we briefly recall the setting. For every probability measure $\mu$ on $\R^d\times\R^d$ we denote by $\mu_+$ and $\mu_-$ the positive and negative part of $\mu$ as defined for the Hahn-Banach decomposition, so that $\mu=\mu_+-\mu_-$ and $|\mu|=\mu_++\mu_-$. The total variation norm of $\mu$ is defined as 
\[
\norm{\mu}_{TV}:=\int_{\R^d\times\R^d}\d\abs{\mu}.
\]
Moreover, for a positive weight function $m\colon\R^d\times\R^d\to [1,\infty)$ we consider the subspace of probability measures $\mathcal{P}_m$ defined by weighted norm
\[
\norm{\mu}_m:=\int_{\R^d\times\R^d}m\d\abs{\mu}
\]
and we also consider the space $\mathcal{N}$ of all zero mean signed measures, that is $\mu\in\mathcal{N}$ if and only if $\int_{\R^d\times\R^d}d\mu=0$.

The fundamental hypotheses are the following.

\begin{hypothesis}[Weak Lyapunov condition]\label{Hyp:Lyapunov}
There exists a continuous function $m\colon \R^d\times \R^d \to [1,+\infty)$ with pre-compact level sets such that 
\begin{equation}
\cL^* m\leq C -\epsilon\phi(m) 
\end{equation}
for some constants $C,\epsilon>0$ and some strictly concave function $\phi\colon\R_+\to\R$ with $\phi(0)=0$ and increasing to infinity.
\end{hypothesis}

\begin{hypothesis}[Minorisation condition]\label{Hyp:minorisation}
We say that the stochastic semigroup $S_\cL$ satisfies the minorisation condition on a set $\mathcal{C}$ if there exists a probability measure $\mu_*$ and a constant $\alpha\in(0,1)$  such that for a certain $T>0$
\begin{equation}\label{con:minorisation}
S_{\cL}(T)\mu \geq \alpha\mu_*\int_\mathcal{C}\mu 
\end{equation}
for all positive measures $\mu$.
\end{hypothesis}

We can now state the subgeometric version of the Harris Theorem, which in particular is taken from \cite[Thm 5.6]{CM21}.

\begin{thm}[Subgeometric Harris' Theorem] \label{thm:harris}
Consider a stochastic semigroup $S_\cL$ with generator $\cL$ which satisfies \Cref{Hyp:Lyapunov} for a continuous function $m\colon \RRd\to[1,\infty)$ and \Cref{Hyp:minorisation} on a set $\mathcal{C}=\{(x,v)\in\R^d\times\R^d\;\colon\; m(x,v)\leq C \}$ for large enough $C$. Then there exists a unique invariant measure $\mu_G\in \mathcal{P}(\R^d\times\R^d)$ such that
\begin{equation}\label{Harris:L1}
    \int_{\R^d\times\R^d}\phi(m)d\mu_G<\infty
\end{equation}
and there exist a decay
rate function $\Theta(t)$ such that
\begin{equation}\label{Harris:rate}
    \norm{S_{\cL}(t)\mu-\mu_G}_{TV}\lesssim \Theta(t)\norm{\mu-\mu_G}_m
\end{equation}
for any probability measure $\mu$. 
\end{thm}
The $\Theta(t)$ function is constructed from the $\phi$ function that appears in the Lyapunov condition, for all the details on its construction we refer to \cite[Sec. 4]{CM21}, where also some examples can be found. In this article we just recall the two rates of the most recurring examples, namely when $\phi(m)=m^{1-\kappa}$ with $0<\kappa<1$ and $\phi(m)=m/(\log m)^\sigma$ with $\sigma>0$.

Suppose first that the Lyapunov condition \Cref{Hyp:Lyapunov} holds with $\phi(m)=m^{1-\kappa}$ for a certain $0<\kappa<1$, then the decay functon $\Theta(t)$ given from the Harris Theorem is 
\[
\Theta(t)=\frac{1}{(1+t)^{1/\kappa}},
\]
that is a polynomial decay. Suppose now that $\phi(m)=m/(\log m)^\sigma$ with $\sigma>0$, then the function $\Theta(t)$ is given by
\[
\Theta(t)=e^{-\lambda t^{\frac{1}{1+\sigma}}}
\]
for an explicitly computable constant $\lambda>0$. It is worth noticing that the stronger the Lyapunov condition, the better the decay rate in \eqref{Harris:rate}. In the two examples shown, the Lyapunov condition is more restrictive with the function $\phi(m)=m/(\log m)^\sigma$ than with $\phi(m)=m^{\kappa-1}$ , in fact the rate of decay is faster. On the other hand, the stronger the Lyapunov function, the closer (in a weighted $\sfL^1$ sense) to the stationary measure the initial state $\mu$ must be taken. Hence, in the contest of convergence of Markov processes, finding better Lyapunov functions is always of great interest.

%%%%%%%%%%%%%%%%%%%%%%%%%%%%%%%%%%%%%%%%%%%%%%%%%%%%%%%%%%%%%%%%%%%%%%%%%%%%%%%%%%%%%%%%%%%%%%%%%%%%%%%

\section{Asymptotic analysis of integrals of Laplace-type}\label{App:Laplace}

In this appendix we review some useful results about the asymptotic behavior of integrals of Laplace type, namely integrals of the form
\[
I(X)=\int_{a}^{b} t^{\lambda-1} g(t)e^{-Xw(t)}\d t, \qquad X>0
\]
with $\lambda>0$ and $-\infty\leq a<b\leq +\infty$. The main observation, who goes back to Laplace, is that the major contribution to the integral $I(X)$ comes from the neighborhood of the points where $w$ attains its smallest value. By subdividing the interval $[a,b]$ if necessary, one can assume, without loss of generality, that $w$ has only one minimum in $[a,b]$ at $x = a$. Under some smoothness conditions on $w$, Laplace's result is
\[
\int_{a}^{b} g(t)e^{-Xw(t)}\d t \sim g(a)\sqrt{\frac{\pi}{2Xw''(a)}}e^{-w(a)X}\qquad \text{as } X\to\infty,
\]
see for example the monograph \cite{E56} or the books \cite{T14,O74,W01}. The proof is based on the Watson Lemma.
\begin{lem}[Watson \cite{W95}]\label{lem:Watson}
Assume that the function $g\colon [0,\infty)\to\R$ has a finite number of discontinuities, 
\[
g(t)\sim \sum_{n=0}^{\infty}a_nt^n\qquad \text{as } t\to0
\]
for some $a_n\in\R$, and the integral $I(X)$ is convergent for sufficiently large positive $X$. Then
\[
\int_0^\infty t^{\lambda-1}g(t)e^{-Xt}\d t \sim \sum_{n=0}^\infty a_n\frac{\Gamma(n+\lambda)}{X^{n+\lambda}}
\]
as $X\to\infty$.
\end{lem}
See \cite{T14,O74} for a recent proof. The Laplace approximation has been generalised to higher dimension, see for example \cite[Chapter 9]{W01}. Concerning our work, the most useful result is the following lemma.
\begin{lem}\label{lem:phi2}
Let $\gamma>0$ and $n\in\NN$. Then there holds
\begin{equation}\label{eq:AppINF}
\int_{0}^\infty u^{n-1}e^{-\frac{u^\gamma}{\gamma}-\frac{\abs{y}}{u}}\d u\sim \sqrt{\frac{2\pi}{1+\gamma}} \, \abs{y}^{\frac{k}{\gamma+1} - \frac{\gamma}{2(1+\gamma)}}e^{-\frac{1+\gamma}{\gamma}\abs{y}^{\frac{\gamma}{1+\gamma}}}\qquad \text{as } \abs{y}\to \infty,
\end{equation}
and
\begin{equation}\label{eq:App0}
\int_{0}^\infty u^{n-1}e^{-\frac{u^\gamma}{\gamma}-\frac{\abs{y}}{u}}\d u\sim
\begin{cases}
    \gamma^{\frac{n}{\gamma}-1}\Gamma\(\frac{n}{\gamma}\) & \text{if } n\geq 1\\
    \abs{\ln\abs{y}} & \text{if } n=0
\end{cases} \qquad \text{as } \abs{y}\to 0.
\end{equation}
\end{lem}
\begin{proof}[{\bf Proof of \Cref{lem:phi2}}]
Concerning \eqref{eq:AppINF}, we first have to write the integral under Laplace form, so we have to change variable $u=\abs{y}^{\frac{1}{1+\gamma}}z$
\begin{align*}
\int_0^\infty u^{n-1}e^{-\frac{u^\gamma}{\gamma}-\frac{\abs{y}}{u}}\d u= \abs{y}^{\frac{n}{\gamma+1}}\int_0^\infty z^{n-1}e^{-\abs{y}^{\frac{\gamma}{1+\gamma}}\(\frac{z^\gamma}{\gamma}+\frac{1}{z}\)}\d z=\abs{y}^{\frac{n}{\gamma+1}} I(\abs{y}^{\frac{\gamma}{1+\gamma}})
\end{align*}
where
\[
I(Y):=\int_0^\infty z^{n-1}e^{-Y\(\frac{z^\gamma}{\gamma}+\frac{1}{z}\)}\d z.
\]
Now the function $Y\mapsto I(Y)$ can be studied through the Watson \Cref{lem:Watson}, thus we need to study the minima of the function $z\mapsto\frac{1}{z}+\frac{z^\gamma}{\gamma}$. We immediately notice that the point $z_0=1$ is a critical point and we have the Taylor expansion
\begin{align*}
    \frac{z^\gamma}{\gamma}+\frac{1}{z}= \frac{\gamma+1}{\gamma} +\frac{\gamma+1}{2} z^2 +o(z^2)
\end{align*}
as $z\to1$. As a consequence the Watson Lemma gives
\[
I(Y)\sim \sqrt{\frac{2\pi}{\gamma+1}}\, Y^{-\frac{1}{2}}e^{-\frac{\gamma+1}{\gamma} Y}
\]
and substituting $Y=\abs{y}^{\a}$ we obtain \eqref{eq:AppINF}.

For the estimate \eqref{eq:App0} we have to split the cases $n\geq 1$ and $n=0$. If $n\geq 1$, we have that
\[
\int_0^\infty u^{n-1}e^{-\frac{u^\gamma}{\gamma}-\frac{\abs{y}}{u}}\d u\to \int_0^\infty u^{n-1}e^{-\frac{u^\gamma}{\gamma}}\d u= \gamma^{\frac{n}{\gamma}-1}\Gamma\(\frac{n}{\gamma}\)
\]
for $\abs{y}\to 0$, as wanted. In the case $n=0$ we have that the integral $\int_0^\infty \frac{1}{u} e^{ -\frac{\abs{y}}{u} -\frac{u^\gamma}{\gamma}} \d u$ is singular for $y\to 0$. We change variable $u=\abs{y}^{\frac{1}{1+\gamma}}z$ and we write
\begin{align*}
    \int_0^\infty \frac{1}{u} e^{ -\frac{\abs{y}}{u} -\frac{u^\gamma}{\gamma}} \d u &= \int_0^\infty \frac{1}{z}e^{-\abs{y}^{\frac{\gamma}{1+\gamma}} \(\frac{z^\gamma}{\gamma}+\frac{1}{z}\)}\d z\\
    &=\int_0^1 \frac{1}{z}e^{-\abs{y}^{\frac{\gamma}{1+\gamma}} \(\frac{z^\gamma}{\gamma}+\frac{1}{z}\)}\d z+\int_1^\infty \frac{1}{z}e^{-\abs{y}^{\frac{\gamma}{1+\gamma}} \(\frac{z^\gamma}{\gamma}+\frac{1}{z}\)}\d z.
\end{align*}
The function $z\mapsto \frac{z^\gamma}{\gamma}+\frac{1}{z}$ is invertible in $(0,1]$ and in $[1,\infty)$, so we can change variable using these two inverse \begin{align}\label{eq:invert}
    \frac{z_1(w_1)^\gamma}{\gamma}+ \frac{1}{z_1(w_1)}&=w_{1} & \frac{z_2(w_2)^\gamma}{\gamma}+ \frac{1}{z_2(w_2)}&=w_{2}
\end{align}
where $z_1(w_1)\in(0,1]$ and $z_2(w_2)\in[1,\infty)$, with $w_1,w_2\in[\frac{\gamma+1}{\gamma},\infty)$. Thus
\begin{align}\label{eq:0212}
    \int_0^\infty \frac{1}{u} e^{ -\frac{\abs{y}}{u} -\frac{u^\gamma}{\gamma}} \d u &= \int_\infty^\frac{\gamma+1}{\gamma}e^{-\abs{y}^{\frac{\gamma}{1+\gamma}}w_1}\frac{1}{z_1(w_1)} z'_1(w_1)dw_1 + \int_\frac{\gamma+1}{\gamma}^\infty e^{-\abs{y}^{\frac{\gamma}{1+\gamma}}w_2}\frac{1}{z_2(w_2)} z'_2(w_2)dw_2
\end{align}
Notice that if $w_1\to \infty$, then $z_1(w_1)\to 0$, and if $w_2\to \infty$ then $z_2(w_2)\to \infty$. Moreover thanks to \eqref{eq:invert} we also see that
\begin{align*}
    z_1(w_1)\sim \frac{1}{w_1}&\text{ and } z'_1(w_1)\sim -\frac{1}{w_1^2}&&\text{ as } w_1\to \infty\\
    z_2(w_2)\sim (\gamma w_2)^{1/\gamma}&\text{ and } z'_2(w_2)\sim  (\gamma w_2)^{1/\gamma-1}&&\text{ as } w_2\to \infty.
\end{align*}
Therefore, changing variables $w=-\abs{y}^{\frac{\gamma}{1+\gamma}}w_1$ in the first integral in \eqref{eq:0212}, we have
\begin{align*}
    \int_\infty^\frac{\gamma+1}{\gamma}e^{-\abs{y}^{\frac{\gamma}{1+\gamma}}w_1}\frac{1}{z_1(w_1)} z'_1(w_1)dw_1&=\int_\infty^{\frac{\gamma+1}{\gamma}\abs{y}^{\frac{\gamma}{1+\gamma}}}e^{-w}\frac{1}{z_1(\abs{y}^{-\frac{\gamma}{1+\gamma}}w)} z'_1(\abs{y}^{-\frac{\gamma}{1+\gamma}}w)\abs{y}^{-\frac{\gamma}{1+\gamma}}dw\\
    &\sim\int_{\frac{\gamma+1}{\gamma}\abs{y}^{\frac{\gamma}{1+\gamma}}}^\infty e^{-w} \abs{y}^{-\frac{\gamma}{1+\gamma}}w \frac{1}{\(\abs{y}^{-\frac{\gamma}{1+\gamma}}w\)^2}\abs{y}^{-\frac{\gamma}{1+\gamma}}dw\\
    &\sim\int_{\frac{\gamma+1}{\gamma}\abs{y}^{\frac{\gamma}{1+\gamma}}}^\infty \frac{e^{-w}}{w} dw\\
    &\sim \abs*{\ln\(\frac{\gamma+1}{\gamma}\abs{y}^{\frac{\gamma}{1+\gamma}}\)}
\end{align*}
as $y\to 0$. By similar computations for the second integral in \eqref{eq:0212}, we find
\[
\int_\frac{\gamma+1}{\gamma}^\infty e^{-\abs{y}^{\frac{\gamma}{1+\gamma}}w_2}\frac{1}{z_2(w_2)} z'_2(w_2)dw_2\sim \frac{1}{\gamma}\abs*{\ln\(\frac{\gamma+1}{\gamma}\abs{y}^{\frac{\gamma}{1+\gamma}}\)}.
\]
We finally have the asymptotic behavior
\[
\int_0^\infty \frac{1}{u} e^{ -\frac{u^\gamma}{\gamma}-\frac{\abs{y}}{u} } \d u \sim \frac{\gamma+1}{\gamma}\abs*{\ln\(\frac{\gamma+1}{\gamma}\abs{y}^{\frac{\gamma}{1+\gamma}}\)}\sim \abs{\ln(\abs{y})}
\]
as $y\to 0$.
\end{proof}

\begin{lem}[Sub-exponential convolution]\label{lem:subexpconv}
The convolution of two exponential decays is exponentially decaying; whereas the convolution of an exponential decay with a sub-exponential decay is sub-exponentially decaying. More precisely we have
\begin{align*}
\int_0^t e^{-b_1s}e^{-b_2(t-s)}\d s&\lesssim e^{b_3 t}, & \int_0^t e^{-b_1s^a}e^{-b_2(t-s)}\d s&\lesssim e^{b_1 t^a}
\end{align*}
with $a\in(0,1)$ and $b_i>0$
\end{lem}
\begin{proof}[{\bf Proof of \Cref{lem:subexpconv}}]
We just compute
\begin{align*}
\int_0^t e^{-b_1s}e^{-b_2(t-s)}\d s{b_2-b_1}=\frac{1}{b_2-b_1}(e^{-b_1t}-e^{-b_2t})\lesssim e^{-b_3 t}
\end{align*}
with $b_3=\max\{b_1,b_2\}$. 

For the second case we notice that 
\begin{align*}
\lim_{t\to \infty}\frac{\int_0^t e^{-b_1s^a+b_2s}\d s}{e^{-b_1t^a+b_2t}}=\lim_{t\to\infty} \frac{e^{-b_1t^a+b_2t}}{(-ab_1t^{a-1}+b_2)e^{-b_1t^a+b_2t}}=\frac{1}{b_2}
\end{align*}
Then by definition of limit there exists $R>0$ large enough such that 
\[
\int_0^t e^{-b_1s^a+b_2s}\d s\leq\(\frac{1}{b_2}+1\)e^{-b_1t^a+b_2t}
\]
for $t\geq R$. Moreover, since both functions are continuous, we can extend this estimate up to a constant to the remaining positive $t$, i.e. for any $t\geq 0$ we have
\[
\int_0^t e^{-b_1s^a+b_2s}\d s\lesssim e^{-b_1t^a+b_2t}.
\]
This is nothing but the second inequality in the statement.
\end{proof}

\bibliographystyle{plain} % We choose the "plain" reference style
\bibliography{Run-and-Tumble} % Entries are in the refs.bib file

\begin{thebibliography}{10}

\bibitem{A80}
W.~Alt.
\newblock Biased random walk models for chemotaxis and related diffusion
  approximations.
\newblock {\em J. Math. Biol.}, 9(2):147--77, 1980.

\bibitem{BGL13}
D.~Bakry, I.~Gentil, and M.~Ledoux.
\newblock {\em Analysis and Geometry of Markov Diffusion Operators}.
\newblock Springer Cham, 1 edition, 2013.

\bibitem{BB72}
H.~C. Berg and D.~A. Brown.
\newblock Chemotaxis in \emph{{Eschericha Coli}} analysed by three-dimensional
  tracking.
\newblock {\em Nature}, 239:500--504, 1972.

\bibitem{BL76}
J.~Bergh and J.~Löfström.
\newblock {\em Interpolation Spaces}.
\newblock Springer Berlin, Heidelberg, 1 edition, 1976.

\bibitem{BCN15}
E.~Bouin, V.~Calvez, and G.~Nadin.
\newblock Propagation in a kinetic reaction-transport equation: Travelling
  waves and accelerating fronts.
\newblock {\em Archive for Rational Mechanics and Analysis}, 217(2):571--617,
  2015.

\bibitem{BDLS20}
E.~Bouin, J.~Dolbeault, L.~Lafleche, and C.~Schmeiser.
\newblock Hypocoercivity and sub-exponential local equilibria.
\newblock {\em Monatshefte f{\"u}r Mathematik}, nov 2020.

\bibitem{BDL22}
E.~Bouin, J.~Dolbeault, and L.~Laflèche.
\newblock Fractional hypocoercivity.
\newblock {\em Communications in Mathematical Physics}, 390(3):1369--1411,
  2022.

\bibitem{BDZ23}
E.~Bouin, J.~Dolbeault, and L.~Ziviani.
\newblock $\mathrm l^2$ hypocoercivity methods for kinetic {F}okker-{P}lanck
  equations with factorised gibbs states.
\newblock {\em Kolmogorov Operators and Their Applications}, 04 2023.

\bibitem{C19}
V.~Calvez.
\newblock Chemotactic waves of bacteria at the mesoscale.
\newblock {\em J. Eur. Math. Soc.}, 22:593--668, 2019.

\bibitem{CRS15}
V.~Calvez, G.~Raoul, and C.~Schmeiser.
\newblock Confinement by biased velocity jumps: {Aggregation} of
  \emph{{Escherichia Coli}}.
\newblock {\em Kinet. Relat. Models}, 8(4):651--666, 2015.

\bibitem{Cao19}
C.~Cao.
\newblock The kinetic {F}okker–{P}lanck equation with weak confinement force.
\newblock {\em Communications in Mathematical Sciences}, 17:2281--2308, 2019.

\bibitem{C21}
C.~Cao.
\newblock The kinetic {F}okker–{P}lanck equation with general force.
\newblock {\em Journal of Evolution Equations}, 21:2293–2337, 06 2021.

\bibitem{CCEY20}
J.~A. Cañizo, C.~Cao, J.~Evans, and H.~Yoldaş.
\newblock Hypocoercivity of linear kinetic equations via {H}arris's theorem.
\newblock {\em Kinet. Relat. Models}, 13(1):97--128, 2020.

\bibitem{CM21}
J.~A. Cañizo and S.~Mischler.
\newblock Harris-type results on geometric and subgeometric convergence to
  equilibrium for stochastic semigroups.
\newblock {\em J. Funct. Anal.}, 284(7):109830, 2021.

\bibitem{CMPS04}
F.~Chalub, P.~Markowich, B.~Perthame, and C.~Schmeiser.
\newblock Kinetic models for chemotaxis and their drift-diffusion limits.
\newblock {\em Monatsh. Math.}, 142:123--141, 2004.

\bibitem{Doeblin1940}
W.~Doeblin.
\newblock Éléments d'une théorie générale des chaînes simples constantes
  de markoff.
\newblock {\em Annales scientifiques de l'École Normale Supérieure},
  57:61--111, 1940.

\bibitem{DMS15}
J.~Dolbeault, C.~Mouhot, and C.~Schmeiser.
\newblock Hypocoercivity for linear kinetic equations conserving mass.
\newblock {\em Trans. Am. Math. Soc.}, 367(6):3807--3828, 2015.

\bibitem{DMT95}
D.~Down, S.~P. Meyn, and R.~L. Tweedie.
\newblock Exponential and uniform ergodicity of {M}arkov processes.
\newblock {\em Ann. Probab.}, 23(4):1671--1691, 1995.

\bibitem{EGZ16}
A.~Eberle, A.~Guillin, and R.~Zimmer.
\newblock Quantitative harris type theorems for diffusions and mckean-vlasov
  processes.
\newblock {\em Transactions of the American Mathematical Society}, 371, 06
  2016.

\bibitem{E56}
A.~Erd{\'e}lyi.
\newblock {\em Asymptotic Expansions}.
\newblock Dover Books on Mathematics. Dover Publications, 1956.

\bibitem{EM24}
J.~Evans and A.~Menegaki.
\newblock Properties of non-equilibrium steady states for the non-linear bgk
  equation on the torus.
\newblock {\em Ann. Inst. H. Poincaré C Anal. Non Linéaire (2024), published
  online first}, 2024.

\bibitem{EY23}
J.~Evans and H.~Yoldaş.
\newblock On the asymptotic behavior of a run and tumble equation for bacterial
  chemotaxis.
\newblock {\em SIAM Journal of Mathematical Analysis}, 55(6), 2023.

\bibitem{EY24}
J.~Evans and H~Yoldaş.
\newblock Trend to equilibrium for run and tumble equations with non-uniform
  tumbling kernels.
\newblock {\em Acta Applicandae Mathematicae}, 191(6), 2024.

\bibitem{GMM17}
M.~P. Gualdani, S.~Mischler, and C.~Mouhot.
\newblock Factorization of non-symmetric operators and exponential
  {$H$}-theorem.
\newblock {\em M\'em. Soc. Math. Fr. (N.S.)}, (153):137, 2017.

\bibitem{H16}
M.~Hairer.
\newblock Convergence of {M}arkov processes.
\newblock {\em Lecture notes.}, 2016.

\bibitem{HM11}
M.~Hairer and J.~C. Mattingly.
\newblock Yet another look at harris' ergodic theorem for markov chains.
\newblock In Robert Dalang, Marco Dozzi, and Francesco Russo, editors, {\em
  Seminar on Stochastic Analysis, Random Fields and Applications VI}, pages
  109--117, Basel, 2011. Springer Basel.

\bibitem{H56}
T.~E. Harris.
\newblock The existence of stationary measures for certain {M}arkov processes.
\newblock In {\em Proceedings of the {T}hird {B}erkeley {S}ymposium on
  {M}athematical {S}tatistics and {P}robability, 1954--1955, vol. {II}}, pages
  113--124. University of California Press, Berkeley and Los Angeles, 1956.

\bibitem{HW19}
S.~Hu and X.~Wang.
\newblock Subexponential decay in kinetic {F}okker–{P}lanck equation: Weak
  hypocoercivity.
\newblock {\em Bernoulli}, 25:174--188, 02 2019.

\bibitem{KMN21}
O.~Kavian, S.~Mischler, and M.~Ndao.
\newblock The {F}okker-{P}lanck equation with subcritical confinement force.
\newblock {\em Journal de Mathématiques Pures et Appliquées}, 151:171--211, 7
  2021.

\bibitem{L20}
L.~Lafleche.
\newblock Fractional {F}okker--{P}lanck equation with general confinement
  force.
\newblock {\em SIAM Journal on Mathematical Analysis}, 52(1):164--196, 2020.

\bibitem{M80}
R.~M. Macnab.
\newblock Sensing the environment.
\newblock {\em Biological Regulation and Development}, 2:377--412, 1980.

\bibitem{MK72}
R.~M. Macnab and D.~E. Koshland.
\newblock The gradient-sensing mechanism in bacterial chemotaxis.
\newblock {\em Proceedings of the National Academy of Sciences},
  69(9):2509--2512, 1972.

\bibitem{MT93}
S.~P. Meyn and R.~L. Tweedie.
\newblock {Stability of {M}arkovian processes III: Foster-Lyapunov criteria for
  continuous-time processes.}
\newblock {\em Adv. Appl. Prob.}, pages 518--548, 1993.

\bibitem{MT09}
S.~P. Meyn and R.~L. Tweedie.
\newblock {\em Markov {C}hains and {S}tochastic {S}tability}.
\newblock 2nd ed., Cambridge Mathematical Library. Cambridge: Cambridge
  University Press, 2019.
\newblock With a prologue by P. W. Glynn.

\bibitem{MQT}
S.~Mischler, C.~Quiñinao, and J.~Touboul.
\newblock On a kinetic {F}itzhugh-{N}agumo model of neuronal network.
\newblock {\em Comm. Math. Phys.}, 342(3):1001–1042, 2016.

\bibitem{MW17}
S.~Mischler and Q.~Weng.
\newblock On a linear run and tumble equation.
\newblock {\em Kinet. Relat. Models}, 10(3):799--822, 2017.

\bibitem{O74}
F.W.J. Olver.
\newblock {\em Asymptotics and Special Functions}.
\newblock Computer science and applied mathematics : a series of monographs and
  textbooks. Academic Press, 1974.

\bibitem{ODA88}
H.~G. Othmer, S.~R. Dunbar, and W.~Alt.
\newblock Models of dispersal in biological systems.
\newblock {\em J. Math. Biol.}, 26:263--298, 1988.

\bibitem{S56}
E.~M. Stein.
\newblock Interpolation of linear operators.
\newblock {\em Transactions of the American Mathematical Society},
  83(2):482--492, 1956.

\bibitem{SW58}
E.~M. Stein and G.~Weiss.
\newblock Interpolation of operators with change of measures.
\newblock {\em Transactions of the American Mathematical Society},
  87(1):159--172, 1958.

\bibitem{S74}
D.~W. Stroock.
\newblock Some stochastic processes which arise from a model of the motion of a
  bacterium.
\newblock {\em Z. Wahrscheinlichkeitstheorie verw. Gebiete}, 28(4):305--315,
  1974.

\bibitem{T14}
N.M. Temme.
\newblock {\em Asymptotic Methods for Integrals}.
\newblock World Scientific, 12 2014.

\bibitem{W95}
G.N. Watson.
\newblock {\em A Treatise on the Theory of Bessel Functions}.
\newblock Cambridge Mathematical Library. Cambridge University Press, 1995.

\bibitem{W01}
R.~Wong.
\newblock {\em Asymptotic Approximations of Integrals}.
\newblock Society for Industrial and Applied Mathematics, 2001.

\end{thebibliography}

%\begin{thebibliography}{9}

%\end{thebibliography}

\end{document}